\documentclass[a4paper,english,reqno]{amsart}

\usepackage[english]{babel}
\usepackage[applemac]{inputenc}
\usepackage[left=3cm,right=3cm,top=2.5cm,bottom=2.5cm]{geometry}
\usepackage{amsmath,amsthm, amssymb,amsfonts,nicefrac}
\usepackage[all]{xy}
\usepackage{amsmath}
\usepackage{graphicx}
\usepackage{textcomp}
\usepackage{color}
\definecolor{burgundy}{rgb}{0.5, 0.0, 0.13}
\usepackage{enumerate}
\usepackage{graphpap}
\usepackage[colorlinks, linktocpage, citecolor = cyan, linkcolor = blue]{hyperref}
\usepackage{framed}
\usepackage{tikz}
\usepackage{caption}
\usepackage{subcaption}

\newcommand{\N}{\ensuremath{\mathbb{N}}}
\newcommand{\Z}{\ensuremath{\mathbb{Z}}}

\newcommand{\R}{\ensuremath{\mathbb{R}}}
\newcommand{\C}{\ensuremath{\mathbb{C}}}
\newcommand{\PP}{\ensuremath{\mathbb{P}}}

\newcommand{\FF}{\ensuremath{\mathbb{F}}}

\renewcommand{\k}{\mathrm{k}}

\newcommand{\bk}{\bar{\mathrm{k}}}
\newcommand{\QQ}{\mathcal{Q}}
\DeclareMathOperator{\Aut}{Aut}
\DeclareMathOperator{\Bir}{Bir}
\newcommand{\Pic}{{\mathrm{Pic}}}
\newcommand{\rk}{{\mathrm{rk}}}

\DeclareMathOperator{\Id}{Id}

\DeclareMathOperator{\pr}{pr}
\newcommand\SB[1][\scalebox{0.6}]{#1}
\newcommand\SBa[1][\scalebox{0.4}]{#1}

\def\dashmapsto{\mapstochar\dashrightarrow}
\newcommand\dashmapsfrom{\mathrel{\reflectbox{\ensuremath{\dashmapsto}}}}

\newtheorem{Thm}{Theorem}[section]
\newtheorem*{Thm*}{Theorem}
\newtheorem{Cor}[Thm]{Corollary}
\newtheorem{Lem}[Thm]{Lemma}
\newtheorem{Prop}[Thm]{Proposition}
\theoremstyle{definition}
\newtheorem{Def}[Thm]{Definition}

\newtheorem{Rmk}[Thm]{Remark}

\title{Infinite algebraic subgroups of the real Cremona group}
\author{Maria Fernanda Robayo and Susanna Zimmermann}

\address{Maria Fernanda Robayo\\
Basel, Switzerland}
\email{maferobayo@gmail.com}

\address{Susanna Zimmermann\\
Institut de Math\'ematiques de Toulouse\\
Universit\'e de Toulouse Paul Sabatier, France}
\email{susanna.zimmermann@math.univ-toulouse.fr}

\thanks{The first author gratefully acknowledges support from the Department of Mathematics and Computer Science, University of Basel. The second author gratefully acknowledges support by the Swiss National Science Foundation via grants PP00P2\_153026/1 and P2BSP2\_168743}

\begin{document}
\maketitle
\thispagestyle{empty}

\begin{abstract}
We give the classification of the maximal infinite algebraic subgroups of the real Cremona group of the plane up to conjugacy and present a parametrisation space of each conjugacy class. Moreover, we show that the real plane Cremona group is not generated by a countable union of its infinite algebraic subgroups.
\end{abstract}

\tableofcontents

\section{Introduction} 
It is most natural to study actions of algebraic groups on algebraic varieties. One tends to assume the action to be regular, but this is quite restrictive; any algebraic group acting regularly on the $n$-dimensional projective space $\PP^n$ is a subgroup of $\Aut_{\k}(\PP^n)=\mathrm{PGL}_{n+1}(\k)$. It is therefore interesting to study rational group actions on $\PP^n$, or, equivalently, algebraic subgroups of $\Bir_{\k}(\PP^n)$. Algebraic subgroups of the complex Cremona group have been studied by many mathematicians; we refer to \cite{B10} for a historical note, where also the complete classification of the maximal algebraic subgroups of the complex Cremona group $\Bir_{\C}(\PP^2)$ is presented. \par
In this paper, we give a classification of the maximal infinite algebraic subgroups of the real Cremona group $\Bir_{\R}(\PP^2)$. The classification is built on the classification of minimal real smooth projective surfaces given in \cite{Com12}, which is manageable, quite opposed to the case of a general perfect field. \par
The finite subgroups of the real Cremona group of odd order have already been classified in \cite{Yas16}; they are conjugate to a subgroup of the automorphism group of a real del Pezzo surface. It is also interesting to look at the group of birational transformations of a rational real smooth minimal model $X$ that are well defined on the set of real points. The group is usually called the group of birational diffeomorphisms of $X$ and has been studied for instance in \cite{BH07, RV05,HM09,KM09,BM14,R15}. The birational diffeomorphisms of the sphere of prime order have been classified in \cite{R15}. \par
In the classification of the infinite algebraic subgroups of $\Bir_{\R}(\PP^2)$, two of the infinite families in the classification of the complex algebraic subgroups of the Cremona group split into two families each: There exist four isomorphism classes of real del Pezzo surfaces of degree $6$, and the automorphism group of two of them are infinite maximal algebraic subgroups. The family of conic bundles splits into real conic bundles coming from a del Pezzo surface of degree $6$ obtained by blowing up the sphere in a pair of non-real conjugate points and into real conic bundles coming from Hirzebruch surfaces. In both cases, an infinite number of them have automorphism groups that are maximal infinite algebraic subgroups of $\Bir_{\R}(\PP^2)$, pairwise non-conjugate. More concretely, the classification is as follows. \par
By $D_6$ we denote the dihedral group with twelve elements.

\begin{Thm}\label{thm:classification}
Every infinite algebraic subgroup of $\Bir_{\R}(\PP^2)$ is contained in a maximal algebraic subgroup. \par
An infinite maximal algebraic subgroup of $\Bir_{\R}(\PP^2)$ is conjugate to $G=\Aut_\R(X)$ where $X$ is a real del Pezzo surface or to $G=\Aut_\R(X,\pi)$ where $\pi\colon X\rightarrow\PP^1$ is a real conic bundle, and the pairs $(X,G)$ are described as follows:
\begin{enumerate}
\item\label{thm:class 1} $X\simeq\PP^2$ and $G\simeq\mathrm{PGL}_3(\R)$,
\item\label{thm:class 3} $X\simeq \mathcal{Q}_{3,1}\subset\PP^3$ is the real rational minimal surface defined by $w^2=x^2+y^2+z^2$ whose real part is diffeomorphic to the $2$-sphere $\mathbb{S}^2$, and $G\simeq\mathbb{P}\mathrm{O}_\R(3,1)$, 
\item\label{thm:class 2.2} $X\simeq\PP^1\times\PP^1$ and $G\simeq(\mathrm{PGL}_2(\R)\times\mathrm{PGL}_2(\R))\rtimes\langle\tau\rangle$, where $\tau\colon(x,y)\mapsto(y,x)$,
\item\label{thm:class 4} $X$ is a del Pezzo surface of degree $6$ obtained by blowing up a pair of non-real conjugate points on $\FF_0$, and the action of $G$ on $\mathrm{Pic}(X)$ induces the split exact sequence
\[1\rightarrow\mathrm{SO}_2(\R)^2\rightarrow G\rightarrow D_6\rightarrow 1,\] 
\item\label{thm:class 5} $X$ is a del Pezzo surface of degree $6$ obtained by blowing up two real points on $\FF_0$, and the action of $G$ on $\mathrm{Pic}(X)$ induces the split exact sequence
\[1\rightarrow(\R^*)^2\rightarrow G\rightarrow D_6\rightarrow 1,\]
\item\label{thm:class 2} $X\simeq\FF_n$, $n\geq2$, is the $n$-th Hirzebruch surface and $G\simeq\R^{n+1}\rtimes\mathrm{GL}_2(\R)/\mu_n$, where $\mu_n=\{\pm1\}$ if $n$ is even and $\mu_n=\{1\}$ if $n$ is odd.
\item\label{thm:class 7} $\eta\colon X\rightarrow X_{[2]}$ is a birational morphism of real conic bundles, where $\pi_{[2]}\colon X_{[2]}\rightarrow\PP^1$ is the conic bundle obtained by blowing up a pair of non-real conjugate points on $\mathcal{Q}_{3,1}$. The morphism $\eta$ blows up $n\geq1$ pairs of non-real conjugate points with non-real fibres on the non-real conjugate disjoint  $(-1)$-curves of $X_{[2]}$ that are the exceptional divisors of $X_{[2]}\rightarrow \mathcal{Q}_{3,1}$. The action of $G$ on $\PP^1$ induces the split exact sequence
\[1\rightarrow \mathrm{SO}_2(\R)\rtimes\Z/2\Z\rightarrow G\rightarrow H_{\Delta}\rightarrow 1\] 
where $H_{\Delta}\subset\mathrm{PGL}_2(\R)$ is the subgroup preserving the set of images in $\PP^1$ of the points blown up by $\eta$ 
and the interval $\pi(X(\R))=\pi_{[2]}(X_{[2]}(\R))$,
\item\label{thm:class 6} $\eta\colon X\rightarrow \FF_n$ is a birational morphism of real conic bundles that is the blow-up of $2n\geq4$ points on the zero section $s_n$ of self-intersection $n$ $($see Section~$\ref{ssec:CB F}$$)$. The action of $G$ on $\PP^1$ induces a split exact sequence 
\[1\rightarrow(\R^*/\mu_n)\rtimes\Z/2\Z\rightarrow G\rightarrow  H_{\Delta}\rightarrow1,\] 
where $H_{\Delta}\subset\mathrm{PGL}_2(\R)$ is the subgroup preserving the set of images in $\PP^1$ of the points blown up by $\eta$, and $\mu_n=\{\pm1\}$ if $n$ is even and $\mu_n=\{1\}$ if $n$ is odd.
\end{enumerate}
Furthermore, the families $(\ref{thm:class 1})-(\ref{thm:class 6})$ are distinct and pairwise non-conjugate in $\Bir_{\R}(\PP^2)$.
\end{Thm}

We also give the parameter space of the maximal infinite algebraic subgroups of $\Bir_{\R}(\PP^2)$.

\begin{Thm}\label{thm:parametrisation}
The families in Theorem~$\ref{thm:classification}$ are distinct families and the conjugacy classes in each family are parametrised by
\begin{itemize}
\item[] $(\ref{thm:class 1})$-$(\ref{thm:class 5})$ One point.
\item[] $(\ref{thm:class 2})$ One point for each $n\geq2$.
\item[] $(\ref{thm:class 7})$ For each $n\geq1$, the set of $n$ pairs of non-real conjugate points in $\PP^1$ modulo the action of $\Aut_\R(\PP^1,[0,\infty])$.
\item[] $(\ref{thm:class 6})$ For each $n\geq2$, the set of $2n$ points in $\PP^1$ consisting of real points or pairs of non-real conjugate points, modulo the action of $\Aut_\R(\PP^1)=\mathrm{PGL}_2(\R)$.
\end{itemize}
\end{Thm}

The groups listed in Theorem~\ref{thm:classification} are contained in infinite (complex) algebraic subgroups of $\Bir_{\C}(\PP^2)$ (classified in \cite[Theorem 2]{B10}), and they are all dense in their complex counterpart.\par
The elements of the group of birational diffeomorphisms of the sphere of prime order are contained in $\Aut_\R(\QQ_{3,1})$, or are automorphisms of real del Pezzo surfaces of degree $2$ or $4$, or are automorphisms of real conic bundles as in family (\ref{thm:class 7}) \cite{R15}. The finite subgroups of $\Bir_{\R}(\PP^2)$ of odd order are contained in automorphism groups appearing in the families (\ref{thm:class 1}), (\ref{thm:class 3}), (\ref{thm:class 2.2}), (\ref{thm:class 4}) or in automorphism groups of del Pezzo surfaces of degree $5$ \cite{Yas16}. The classification in Theorem~\ref{thm:classification} does not list automorphism groups of del Pezzo surfaces of degree $5$ because they are finite. \par

The real Cremona group of the plane is generated by the family of standard quintic transformations and $\Aut_\R(\PP^2)$ \cite{BM14}, and its abelianisation $\Bir_{\R}(\PP^2)/\langle\langle\Aut_\R(\PP^2)\rangle\rangle\simeq\bigoplus_\R\Z/2\Z$ is generated by the classes of the standard quintic transformations \cite{Z15}. The classification of the maximal infinite subgroups of $\Bir_{\R}(\PP^2)$ yields the following theorem.

\begin{Thm}\label{thm:quotient}
An infinite algebraic subgroup of $\Bir_{\R}(\PP^2)$ with non-trivial image in the abelianisation $\Bir_{\R}(\PP^2)/\langle\langle\Aut_\R(\PP^2)\rangle\rangle\simeq\bigoplus_\R\Z/2\Z$ has finite image in the abelianisation and is conjugate to a subgroup of an algebraic group in family $(\ref{thm:class 7})$ of Theorem~$\ref{thm:classification}$. \par
Furthermore, for each generator of $\bigoplus_\R\Z/2\Z$ there is a conjugacy class of infinite algebraic groups in family $(\ref{thm:class 7})$ of Theorem~$\ref{thm:classification}$ which is sent onto the generator.
\end{Thm}

\begin{Cor}\label{cor:quotient}
The group $\Bir_{\R}(\PP^2)$ is not generated by a countable union of infinite algebraic subgroups.
\end{Cor}

Quite the opposite is true for the Cremona group $\Bir_{\C}(\PP^2)$ of the complex plane. As the standard quadratic transformation of $\PP^2$ is conjugate to an automorphism of $\PP^1\times\PP^1$, it follows from the Noether-Castelnuovo theorem \cite{Cas01} that $\Bir_{\C}(\PP^2)$ is generated by the infinite algebraic groups $\Aut_\C(\PP^2)$ and $\Aut_\C(\PP^1\times\PP^1)$.

\vskip\baselineskip
\noindent {\bf Acknowledgements:} The authors would like to express their warmest thanks to J\'er\'emy Blanc for many helpful discussions and Michel Brion for his detailed remarks on algebraic groups and equivariant completions. They would also like to thank Egor Yasinsky for interesting discussions about the real Cremona group.

\section{First steps}
Throughout this section, $\k$ is a perfect field. 
\subsection{Algebraic groups}
For a projective algebraic variety $X$ defined over $\k$, we denote by $\Bir_\k(X)$ its group of birational self-maps defined over $\k$ and by $\Aut_{\k}(X)\subset\Bir_\k(X)$ the group of $\k$-automorphisms of $X$ defined over $\k$. \par
Note on the side: the group $\Aut_{\k}(X)$ is the group of $\k$-rational points of a group scheme locally of finite type over $\k$, having at most countably many connected components.

The definition of rational actions on varieties go back to Weil and Rosenlicht (see for instance \cite{W55,Ros56}).
\begin{Def}
Let $X$ be an algebraic variety over $\k$. 
\begin{enumerate}
\item Let $G$ be an algebraic group defined over $\k$. A {\em $\k$-rational action} $\rho$ of $G$ on $X$ is a $\k$-rational morphism $\rho\colon G\times X\dashrightarrow X$ such that 
\begin{itemize}
\item $\rho(e,x)=x$ for all $x\in X$, 
\item $\rho$ is associative whenever it is defined, i.e. $\rho(g_1,\rho(g_2,x))=\rho(g_1g_2,x)$ whenever $\rho$ is defined at $(g_2,x)$, at $(g_1,\rho(g_2,x))$ and at $\rho(g_1g_2,x)$, 
\item there exist open dense subsets $U,V\subset G\times X$ such that the $\k$-rational map $G\times X\dashrightarrow G\times X$, $(g,x)\dashmapsto(g,gx)$ restricts to an isomorphism $U\rightarrow V$ and the projection of $U$ and $V$ to the first factor is surjective onto $G$. 
\end{itemize} 
For each $\k$-rational point $g\in G$, we get a $\k$-birational map $\rho(g,\cdot)\colon X\dashrightarrow X$, and this induces a group homomorphism $G(\k)\rightarrow\Bir_\k(X)$.
\item We say that $X$ is a {\em $G$-variety} if $\rho$ is regular. Then the action induces a group homomorphism $G(\k)\rightarrow\Aut_{\k}(X)$.
\end{enumerate}
\end{Def}

\begin{Rmk}
Let $G$ be an algebraic group with a $\k$-rational action on $X$. For $g\in G$ we obtain a $\k$-birational map $\rho(g,\cdot)\colon X\dashrightarrow X$ if and only if $g\in G(\k)$. The induced homomorphism $G(\k)\rightarrow\Bir_\k(X)$ is injective if the action of $G$ on $X$ is faithful.
\end{Rmk}

In this paper, we classify the infinite (not necessarily connected) maximal algebraic subgroups $G(\k)$ up to conjugation inside $\Bir_\k(X)$ for $\k=\R$. The classification over algebraically closed fields of characteristic zero is done in \cite{B10} (in fact, it also classifies the finite maximal ones). \par
As result of looking at $G(\R)$ only, the groups in the classification are real linear algebraic groups, i.e. can be embedded into some $\mathrm{GL}_N(\R)$.

\begin{Def}
Let $G$ be an algebraic group with a $\k$-rational faithful action on $X$. Then the induced the group homomorphism $G(\k)\rightarrow\Bir_\k(X)$ is injective, and we call $G(\k)$ an {\em algebraic subgroup} of $\Bir_\k(X)$.
\end{Def}

\begin{Rmk}
Note that classically, the algebraic group $G$ is called algebraic subgroup of $\Bir_\k(X)$, and not the group $G(\k)$. However, we can only view the $\k$-rational points of $G$ as elements of $\Bir_\k(X)$.
\end{Rmk}

\begin{Rmk}
Let $G$ be an algebraic group. Recall that a {\em morphism} $G\rightarrow\Bir_\k(X)$ is defined as follows (see for instance \cite{BF13}): let $\mu\colon G\times X\dashrightarrow G\times X$ be a $\k$-rational map inducing an isomorphism $U\rightarrow V$, where $U,V\subset G\times X$ are open dense subsets whose projections onto $G$ are surjective. The rational map $\mu$ is given by $(g,x)\dashmapsto (g,p_2(\mu(g,x)))$, where $p_2$ is the second projection, and for each $\k$-rational point $g\in G$, the birational map $x\mapsto p_2(\mu(g,x))$ corresponds to an element $f_g\in\Bir_\k(X)$. The maps $g\mapsto f_g$ represent a map from $G(\k)$ to $\Bir_\k(X)$, which is called {\em morphism} from $G$ to $\Bir_\k(X)$. \par
If the map $G(\k)\rightarrow\Bir_\k(X)$ is a homomorphism of groups, then its image is an algebraic subgroup of $\Bir_\k(X)$. 
\end{Rmk}

The group $\Bir_\k(X)$ can be endowed with the so-called Zariski topology, introduced by \cite{BF13, Dem70,Ser10}, which is compatible with the concept of morphism of varieties into $\Bir_\k(X)$; a subset $F\subset\Bir_\k(X)$ is closed if for any algebraic variety $A$ and any morphism $A\rightarrow\Bir_\k(X)$, the pre-image of $F$ is closed. 
Endowed with the Zariski topology, $\Bir_{\k}(\PP^n)$ is not an ind-variety, algebraic stack or algebraic space if $n\geq2$ \cite[Theorem 1, Remark 3.5]{BF13}. The following lemma gives a sufficient and necessary condition for a subgroup of $\Bir_{\k}(\PP^n)$ to be an algebraic subgroup.

\begin{Lem}[{\cite[Corollary 2.18, Lemma 2.19, Remark 2.21]{BF13}}]\label{lem:BF13} Let $G\subset\Bir_{\k}(\PP^n)$ be a subgroup. Then $G$ is an algebraic subgroup of $\Bir_{\k}(\PP^n)$ if and only if it is  closed in the Zariski topology and of bounded degree. \par
Furthermore, any algebraic subgroup of $\Bir_{\k}(\PP^n)$ can be embedded into some $\mathrm{GL}_N(\k)$.
\end{Lem}

The fact that algebraic subgroups of $\Bir_{\k}(\PP^2)$ are linear algebraic groups, makes them approachable with classical tools. 

\begin{Lem}\label{lem:projective model} Let $\k$ be a field of characteristic zero, $G$ a linear algebraic group and $X$ a $G$-surface, all defined over $\k$. Then there exists a smooth projective $G$-surface $Y$ and a $G$-equivariant birational map $X\dashrightarrow Y$. \par
\end{Lem}
\begin{proof}
We can lift the action of $G$ to the normalisation $X'$ of $X$. The surface $X'$ has only finitely many singular points, the set of which is $G$-invariant. We blow up the singular points (this is $G$-equivariant) and normalise again, lifting the $G$-action onto the normalisation. This process ends by \cite[Theorem of Reduction of Singularities, p.688]{Z39} and we have obtained a smooth $G$-surface $X''$. 
Forgetting about the group action, we see that $X''$ is contained as open set in a complete surface, which can be desingularised. Smooth complete surfaces are projective, hence $X''$ is contained as open set in a projective surface and thus is quasi-projective. 
Let $G_0$ be the neutral component of $G$. We can apply \cite[Theorem 4.9]{Sum75}; there exists a $G_0$-equivariant smooth completion of $X''$. Equivalently, $X''$ admits a $G_0$-linearisable ample line bundle. \cite[Lemma 3.2]{Bri14} implies that $X''$ admits a $G$-linearised ample line bundle and hence a $G$-equivariant completion $Y$ of $X''$. We may replace $Y$ with a $G$-equivariat desingularisation. We have found a birational $G$-equivariant map $X\dashrightarrow Y$ to a smooth projective $G$-surface.
\end{proof}

By $\bk$ we denote the algebraic closure of the perfect field $\k$.
\begin{Def}
An algebraic variety $X$ over a field $\k$ is {\em geometrically rational} if $X_{\bk}:=X\times_{\mathrm{Spec}(\k)}\mathrm{Spec}(\bk)$ is rational, i.e. if it is rational as variety over $\bk$. \par
\noindent A geometrically rational variety $X$ is {\em $\k$-rational} if there is a birational map $X\dashrightarrow\PP^n$ defined over $\k$. 
\end{Def}

The following lemma is classical and states a necessary and sufficient condition for a (abstract) subgroup of $\Aut_{\k}(X)$ to have the structure of a linear algebraic group acting regularly on $X$.

\begin{Def}
For a smooth projective variety over $\k$ with $X(\k)\neq\emptyset$, we denote by $\mathrm{Pic}(X)=\mathrm{Pic}(X_{\bk})^{\mathrm{Gal}(\bk/\k)}$ Galois-invariant Picard group.
\end{Def}

The action of $\Aut_{\k}(X)$ on $X$ induces a homomorphism of (abstract) groups 
\[\Aut_{\k}(X)\rightarrow\Aut(\mathrm{Pic}(X)).\] 

\begin{Lem}\label{lem:lin alg}Let $X$ be a smooth projective variety defined over a field $\k$.
\begin{enumerate}
\item\label{lin alg 1} Let $D$ be a very ample divisor on $X$ defined over $\k$ and $G\subset\Aut_{\k}(X)$ the group of elements fixing $D$. Then the $\k$-embedding $X\hookrightarrow\PP^n$ given by the linear system of $D$ conjugates $G$ to a closed subgroup of $\mathrm{PGL}_{n+1}(\k)$. 
\item\label{lin alg 2} The kernel $K=\ker(\Aut_{\k}(X)\rightarrow\Aut(\mathrm{Pic}(X)))$ has the structure of a linear algebraic group acting regularly on $X$ via the given inclusion $K\subset\Aut_{\k}(X)$.
\item\label{lin alg 3} Any subgroup $G\subset\Aut_{\k}(X)$ containing $K$ whose action on $\mathrm{Pic}(X)$ is finite has the structure of a linear algebraic group acting regularly on $X$ via the inclusion $G\subset\Aut_{\k}(X)$.
\item\label{lin alg 4} Suppose that $G$ is a linear algebraic group and $X$ a $\k$-rational $G$-variety. If $G$ contains $K$, then $G$ has finite action on $\mathrm{Pic}(X)$.
\end{enumerate}
\end{Lem}
\begin{proof}
(\ref{lin alg 1}): The linear system of $D$ induces a closed embedding $\varphi_D\colon X\hookrightarrow\PP^N$ defined over $\k$. Let $G\subset\Aut_{\k}(X)$ be the subgroup of elements whose image in $\Aut(\mathrm{Pic}(X))$ fix $D$. Denote by $H\subset\mathrm{PGL}_{N+1}(\k)$ the subgroup preserving $\varphi_D(X)$, which is a closed subgroup. It is also a subgroup of $\Aut_\k(\varphi_D(X))$, and $G$ is conjugate to $H$ via $\varphi_D$. \par
(\ref{lin alg 2}): For any very ample divisor $D$ on $X$ defined over $\k$, the group $G$ from $(1)$ contains $K$ as a closed subgroup, which is therefore a linear algebraic group as well and acts on $X$ regularly.\par
(\ref{lin alg 3}): If $G\subset\Aut_{\k}(X)$ has finite action on $\mathrm{Pic}(X)$ and contains $K$, then $G/K$ is finite. As $K$ is a linear algebraic group, also $G$ is one.\par
(\ref{lin alg 4}): The group $G$ is a linear algebraic group by assumption and the group $K$ is a normal, closed linear algebraic group by $(\ref{lin alg 2})$, so the group $G/K$ is a linear algebraic group. If $X$ is a $\k$-rational smooth projective variety then $\mathrm{Pic}(X)$ is finitely generated and has no torsion, i.e. $\mathrm{Pic}(X)\simeq\Z^n$ for some $n\in\N$. By assumption, $X$ is a $G$-variety, so there is an inclusion $G\subset\Aut_{\k}(X)$ of abstract groups. Then $G/K\subset\Aut(\mathrm{Pic}(X))\subset\mathrm{GL}_n(\Z)$. Since $\mathrm{GL}_n(\Z)$ is countable, the group $G/K$ is finite. 
\end{proof}

\subsection{Minimal surfaces}

\begin{Def}
We denote by $(X,G)$ the pair consisting of a smooth projective surface $X$ defined over $\k$ and $G$ a subgroup of $\Aut_{\k}(X)$. 
\begin{enumerate}
\item We say that $(X,G)$ is a \emph{minimal pair} (or $X$ is {\em $G$-minimal}) if for any smooth projective surface $Y$ over $\k$ any birational $G$-morphism $X\to Y$ is an isomorphism.
\item Let $\pi\colon X\rightarrow C$ be a $G$-equivariant morphism, where $C$ a curve. We say that $\pi$ is {\em relatively $G$-minimal} if for any decomposition $\pi\colon X\stackrel{\eta}\rightarrow Y\stackrel{\pi'}\rightarrow C$, where $\pi'$ is a $G$-equivariant morphism and $\eta$ is a birational $G$-equivariant morphism, $\eta$ is in fact an isomorphism, 
\end{enumerate}
\end{Def} 
Note that for $G=\{1\}$, a $G$-minimal surface is just a minimal surface.

\begin{Def}
We say that a smooth projective $G$-surface $X$ admits a conic bundle structure if there exists $\pi\colon X\rightarrow C$, where $C$ is a smooth curve and the fibre over closed point $t$ is isomorphic to a reduced conic over the residue field $\k(t)$ of $t$.
\end{Def}

\begin{Rmk}
A $G$-surface admitting a conic bundle structure $\pi\colon X\rightarrow C$ is relatively $G$-minimal if any $\k$-birational $G$-equivariant morphism $X\rightarrow Y$ of conic bundles is an isomorphism. \par
If $X$ is a geometrically rational smooth surface, then $C_{\bk}\simeq\PP^1$. If moreover $X(\k)\neq\emptyset$, then $C(\k)\neq\emptyset$, and so $C\simeq\PP^1$ over $\k$.
\end{Rmk}

A real smooth projective surface $X$ can be seen as a pair $(X_{\C},\sigma)$ consisting of a smooth projective complex variety $X_\C$ and an antiholomorphic involution $\sigma$.  

\begin{Def}
For a real conic bundle $\pi\colon X\rightarrow \PP^1$, we define 
\[\Aut_\R(X,\pi):=\{f\in\Aut_\R(X)\mid \exists \alpha\in\Aut_\R(\PP^1): \pi f=\alpha\pi\}\subset\Aut_\R(X),\] 
the group of automorphisms preserving the conic bundle structure, and 
\[\Aut_\R(X/\pi):=\{f\in\Aut_\R(X,\pi)\mid \pi f=\pi\}\subset\Aut_\R(X,\pi),\] 
its subgroup acting trivially on $\PP^1$.
\end{Def}

Every Hirzebruch surface $\FF_n$ admits a natural real structure with real points: writing
\[\FF_n\simeq\{([x_0:x_1:x_2],[u:v])\in\PP^2\times\PP^1\mid x_1v^n=x_2u^n\}\]
 the standard antiholomorphic involution of $\PP^2\times\PP^1$, that is, the standard antiholomorphic involution on either factors, descends to a antiholomorphic involution on $\FF_n$.  \par

\begin{Def}
By $\QQ_{3,1}\subset\PP^3$, we denote the real surface given by $w^2=x^2+y^2+z^2$ endowed with the standard antiholomorphic involution on $\PP^3$. 
\end{Def}

\begin{Rmk}\label{rmk:real part}
Note that $\QQ_{3,1}(\R)=\mathbb{S}^2$ is the $2$-dimensional real sphere, $\FF_{2n}(\R)=\mathbb{S}^1\times\mathbb{S}^1$ is the real torus and $\FF_{2n+1}(\R)$ is the Klein bottle for any $n\geq0$. The isomorphism of complex surfaces 
\[
\begin{array}{rcll}
&(\QQ_{3,1})_{\C}&\longrightarrow&(\PP^1\times\PP^1)_{\C} \\ \vspace{.5em}
\varphi\colon &[w:x:y:z]&\mapsto&([w+z:y+{\bf i}x],[w+z:y-{\bf i}x])=([y-{\bf i}x:w-z],[y+{\bf i}x:w-z])\\ 
\varphi^{-1}\colon&([x_0:x_1],[y_0:y_1])&\mapsto&[x_0y_0+x_1y_1:{\bf i}(x_0y_1-x_1y_0):x_0y_1+x_1y_0:x_0y_0-x_1y_1]
\end{array}
\]
induces an isomorphism of real surfaces $\varphi\colon \QQ_{3,1}\rightarrow(\PP^1\times\PP^1,\sigma_S)$, where
\[\sigma_{S}\colon([x_0:x_1],[y_0:y_1])\mapsto([\bar{y}_0:\bar{y}_1],[\bar{x}_0:\bar{x}_1])\]
Note that $\QQ_{3,1}$ is a del Pezzo surface of degree $8$ with $\rk(\mathrm{Pic}(\QQ_{3,1}))=1$, whereas $\rk(\mathrm{Pic}(\PP^1\times\PP^1))=2$.
\end{Rmk}

\begin{Thm}[\cite{Com12}]
Let $X$ be a minimal geometrically rational real surface $X$ with $X(\R)\neq\emptyset$. If $X$ is $\R$-rational, then $X$ is $\R$-isomorphic to $\PP^2$, to the quadric $\QQ_{3,1}$ or to a real Hirzebruch surface $\FF_n$, $n\neq1$. 
\end{Thm}

\begin{Def}\label{def:X_2}
By $X_{[2]}$ we denote a del Pezzo surface of degree $6$ obtained by blowing up $\QQ_{3,1}$ in a pair of non-real conjugate points. The notation is motivated by the fact that $\rk{\mathrm{Pic}}(X_{[2]})=2$.
\end{Def}

\begin{Lem}\label{lem:X_2}
Any real surface $X_{[2]}$ is isomoprhic to
\[X_{[2]}\simeq\{([w:x:y:z],[u:v])\in\PP^3\times\PP^1\mid wz=x^2+y^2,uz=vw\}\]
endowed with the antiholomorphic involution that is the restriction of the standard antiholomorphic involution on $\PP^3\times\PP^1$. 
\end{Lem}
\begin{proof}
Consider the surface $S\subset\PP^3$ given by $wz=x^2+z^2$. The isomorphism $[w:x:y:z]\mapsto[w+z:x:y:w-z]$ yields an isomorphism $\QQ_{3,1}\simeq S$. Pick a non-real point $p\in S$ and denote by $\eta\colon X_{[2]}\rightarrow S$ the blow-up of $p,\bar{p}$. We find an automorphism of $S$ that sends $p$ onto $[0:1:{\bf i}:0]$ \cite[Lemma 4.8]{R15}. The blow-up of $p,\bar{p}$ on $\QQ_{3,1}$ is the restriction of the blow-up of $\PP^3$ along the line $l$ given by $w=z=0$, as the intersection of $l$ and $S$ is transversal and equal to the set $\{p,\bar{p}\}$. This yields the claim.
\end{proof}

\begin{Rmk}
The projection $\pi_{[2]}\colon X_{[2]}\rightarrow\PP^1$, $([w:x:y:z],[u;v])\mapsto[u:v]$ is real conic bundle morphism with two singular fibres, which lie over $0$ and $\infty$, and without sections.
\end{Rmk}

\begin{Prop}\label{prop which cases}
Let $G$ be an infinite algebraic subgroup of $\Bir_{\R}(\PP^2)$. Then there exists a $G$-equivariant real birational map $\PP^2\dashrightarrow X$ to a real smooth $G$-surface $X$, which is one of the following:
\begin{enumerate}
\item $X$ is a real del Pezzo surface of degree 6, 8 or 9 such that $\rk(\Pic(X)^G)=1$.
\item $X$ admits a real conic bundle structure $\pi_X\colon X\rightarrow\PP^1$ with $\rk(\Pic(X)^G)=2$ and $G\subset\Aut_\R(X,\pi_X)$. 
\end{enumerate}
Furthermore, in $(2)$, there is a birational morphism of conic bundles $\eta\colon X\to Y$, where $Y$ is a Hirzebruch surface $\FF_n$, $n\geq1$, or $Y\simeq X_{[2]}$. 
\end{Prop}
\begin{proof}
By Lemma~\ref{lem:BF13}, $G$ is a linear algebraic subgroup of $\Bir_{\R}(\PP^2)$. So, there is a real algebraic $G$-surface $X'$ and a $G$-equivariant real birational map $\phi\colon X'\dasharrow \PP^2$ \cite[Theorem 1]{Ros56}. By Lemma~\ref{lem:projective model} there exists a smooth real projective $G$-surface $X''$ and a $G$-equivariant real birational map $X\dashrightarrow X''$. After contracting all the sets of disjoint $G$-invariant real $(-1)$-curves and all sets of disjoint $G$-invariant pairs of non-real conjugate $(-1)$-curves, we obtain a real smooth projective $G$-variety $X$ that can be one of the following possibilities by \cite[Excerise 2.18]{KM08} and Lemma~\ref{lem:lin alg}~(\ref{lin alg 4}):
\begin{enumerate}[(i)]
\item $X$ is a del Pezzo surface and $\rk(\Pic(X)^G)=1$,
\item $X$ admits a real conic bundle structure $X\xrightarrow{\pi_X}\PP^1$ and $\rk(\Pic(X)^G)=2$.
\end{enumerate} 
Due to Comessatti \cite{Com12}, the minimal smooth geometrically rational real surfaces are $\R$-isomorphic to $\PP^2$, to $\QQ_{3,1}$ or to a real Hirzebruch surface $\FF_n$, $n\neq1$. Thus, forgetting about the action of $G$ and $\pi_X$, there is a real birational morphism $X\rightarrow Y$ where $Y$ is one of these three minimal surfaces.\par
In case (i), the situation is as follows: Any real geometrically rational del Pezzo surface is the blow-up of at most $8$ $\C$-points on one of the three real minimal surfaces. The automorphism group of a del Pezzo surface of degree $\leq5$ is finite \cite[Section 6]{DI09}. Any real del Pezzo surface $X$ of degree $7$ has three $(-1)$-curves, one of which is real. Hence $X$ is the blow-up of $\FF_0$ or $\QQ_{3,1}$ in one real point $p$ and any automorphism of $X$ preserves its exceptional divisor and hence is the lift of an automorphism from $\QQ_{3,1}$ or $\FF_0$. In particular, $\Aut_\R(X)$ is conjugate by the blow-up of $p$ to a subgroup of $\Aut_\R(\FF_0)$ or $\Aut_\R(\QQ_{3,1})$. This leaves degree $6$, $8$ and $9$. \par
In case (ii), the situation is as follows: Forgetting the action of $G$, there is a real birational morphism of real conic bundles $X\rightarrow Y$, which is the contraction of all disjoint real and disjoint pairs of non-real $(-1)$-curves in the fibres. We obtain a relatively minimal real conic bundle $\pi_Y\colon Y\rightarrow\PP^1$ with no real $(-1)$-curves and no non-real conjugate singular fibres, i.e. it has at most real singular fibres whose components are non-real conjugate $(-1)$-curves. Forgetting about $\pi_Y$, we obtain a real birational morphism $Y\rightarrow Z$ to a minimal real smooth rational surface, and $Z(\R)$ is connected and homeomorphic to the real projective plane, the sphere, the torus or the Klein bottle (see Remark~\ref{rmk:real part}). It follows that $Y$ has either none or exactly two singular fibres (otherwise $Y(\R)$ is not connected and thus not rational). In this case, $Y$ is the blow-up of $\QQ_{3,1}$ in a pair of non-real conjugate points, i.e. $Y\simeq X_{[2]}$, or $Y$ has no singular fibres and is isomorphic to a real Hirzebruch surface $\FF_n$, $n\neq1$. 
\end{proof}

\begin{Lem}[Real version of {\cite[Proposition 2.2.6]{B10}}]\label{lem aut} Let $X$ be a real smooth projective geometrically rational surface. 
\begin{enumerate}
\item If $X$ is a del Pezzo surface, then $\Aut_\R(X)$ is a linear algebraic group. 
\item If $\pi\colon X\rightarrow\PP^1$ is a real conic bundle, $\Aut_\R(X,\pi)$ is a linear algebraic group. 
\end{enumerate}
\end{Lem}
\begin{proof}
(1): Any element of $\Aut_\R(X)$ fixes any multiple of the anti-canonical divisor, so the claim follows from Lemma~\ref{lem:lin alg}~$(\ref{lin alg 1})$. \par
(2): Since $X$ is smooth projective and geometrically rational, $\mathrm{Pic}(X)$ is generated by $K_X$ and the real classes of the singular fibres of $\pi$ (or if there are none, the general fibre) of which there are finitely many. Furthermore, because $X$ is geometrically rational, $\Pic(X)\simeq\Z^n$. Let $K:=\ker(\Aut_\R(X)\rightarrow\Aut(\mathrm{Pic}(X)))$. Then $\Aut_\R(X)/K$ is a subgroup of $\mathrm{GL}_n(\Z)$ and fixes $K_X$ and the class of the general fibre. It therefore corresponds to a subgroup of permutations of the components of the singular fibres, and is thus a finite group. Lemma~\ref{lem:lin alg} (\ref{lin alg 3}) implies that $\Aut_\R(X)$ is a linear algebraic group.
\end{proof}

\begin{Rmk}\label{rmk:which groups} 
It follows from Proposition~\ref{prop which cases} that every infinite algebraic subgroup of $\Bir_{\R}(\PP^2)$ is contained in the automorphism group of one of the surfaces in Proposition~\ref{prop which cases}, which are linear algebraic groups by  Lemma~\ref{lem aut}. It now suffices to study the pairs $(X,\Aut_\R(X))$ and $(X,\Aut_\R(X,\pi))$ for the cases stated in Proposition~\ref{prop which cases} and to determine which automorphism groups are maximal algebraic groups up to conjugacy.
\end{Rmk}

\section{Real rational del Pezzo surfaces of degree 6}\label{sec:DP}

According to Proposition~\ref{prop which cases}, the maximal infinite algebraic subgroups of $\Bir_{\R}(\PP^2)$ are contained in the automorphism groups of real del Pezzo surfaces of degree $9,8$ or $6$ or the automorphism groups of real conic bundles. In this section, we first classify the real del Pezzo surfaces of degree $6$ and give their automorphism groups as explicitly as we dare. 

\begin{Lem}\label{lem:dP class} Let $X$ be a real del Pezzo surface of degree $6$. 
\begin{enumerate}
\item Then $X$ is the blow-up of $\QQ_{3,1}$ or $\FF_0$ in two real or a pair of non-real conjugate points and there are four isomorphism classes, represented in Figures~\ref{fig:X_2}, \ref{fig:X_3,S}, \ref{fig:X_3,T} and \ref{fig:X_4}.
\item The rank of their invariant Picard group is $2, 3, 3$ and $4$, respectively.
\end{enumerate}
\end{Lem}
\begin{proof}
The complex surface $X_\C$ is the blow-up of three points in $\PP^2$. It has thus exactly six $(-1)$-curves, which are arranged as a hexagon on $X$. The antiholomorphic involution $\sigma$ on $X$ acts on the hexagon as symmetry of order $2$. The only possible cases are shown in Figure~\ref{fig:X_2}, $\ref{fig:X_3,S}$, $\ref{fig:X_3,T}$ and $\ref{fig:X_4}$, the action of $\sigma$ indicated by arrows. The second claim follows from the first.
\end{proof}

Let $X$ be a real Del Pezzo surface of degree $6$. There is an exact sequence
\[1\rightarrow K\rightarrow\Aut_\R(X)\stackrel{\rho}\rightarrow\Aut(\Pic(X))\]
and $K$ is of finite index, because the action of $\Aut_\R(X)$ on $\Pic(X)$ is finite by Lemma~\ref{lem:lin alg} and Lemma~\ref{lem aut}.  The image of $\rho(\Aut_\R(X))$ is a subgroup of the dihedral group $D_6$ acting on the hexagon of $(-1)$-curves.


\begin{center}
\begin{minipage}[h]{0.4\textwidth}
\def\svgwidth{.7\textwidth}
\begingroup%
  \makeatletter%
  \providecommand\color[2][]{%
    \errmessage{(Inkscape) Color is used for the text in Inkscape, but the package 'color.sty' is not loaded}%
    \renewcommand\color[2][]{}%
  }%
  \providecommand\transparent[1]{%
    \errmessage{(Inkscape) Transparency is used (non-zero) for the text in Inkscape, but the package 'transparent.sty' is not loaded}%
    \renewcommand\transparent[1]{}%
  }%
  \providecommand\rotatebox[2]{#2}%
  \ifx\svgwidth\undefined%
    \setlength{\unitlength}{778.14509835bp}%
    \ifx\svgscale\undefined%
      \relax%
    \else%
      \setlength{\unitlength}{\unitlength * \real{\svgscale}}%
    \fi%
  \else%
    \setlength{\unitlength}{\svgwidth}%
  \fi%
  \global\let\svgwidth\undefined%
  \global\let\svgscale\undefined%
  \makeatother%
  \begin{picture}(1,1.56080091)%
    \put(0,0){\includegraphics[width=\unitlength,page=1]{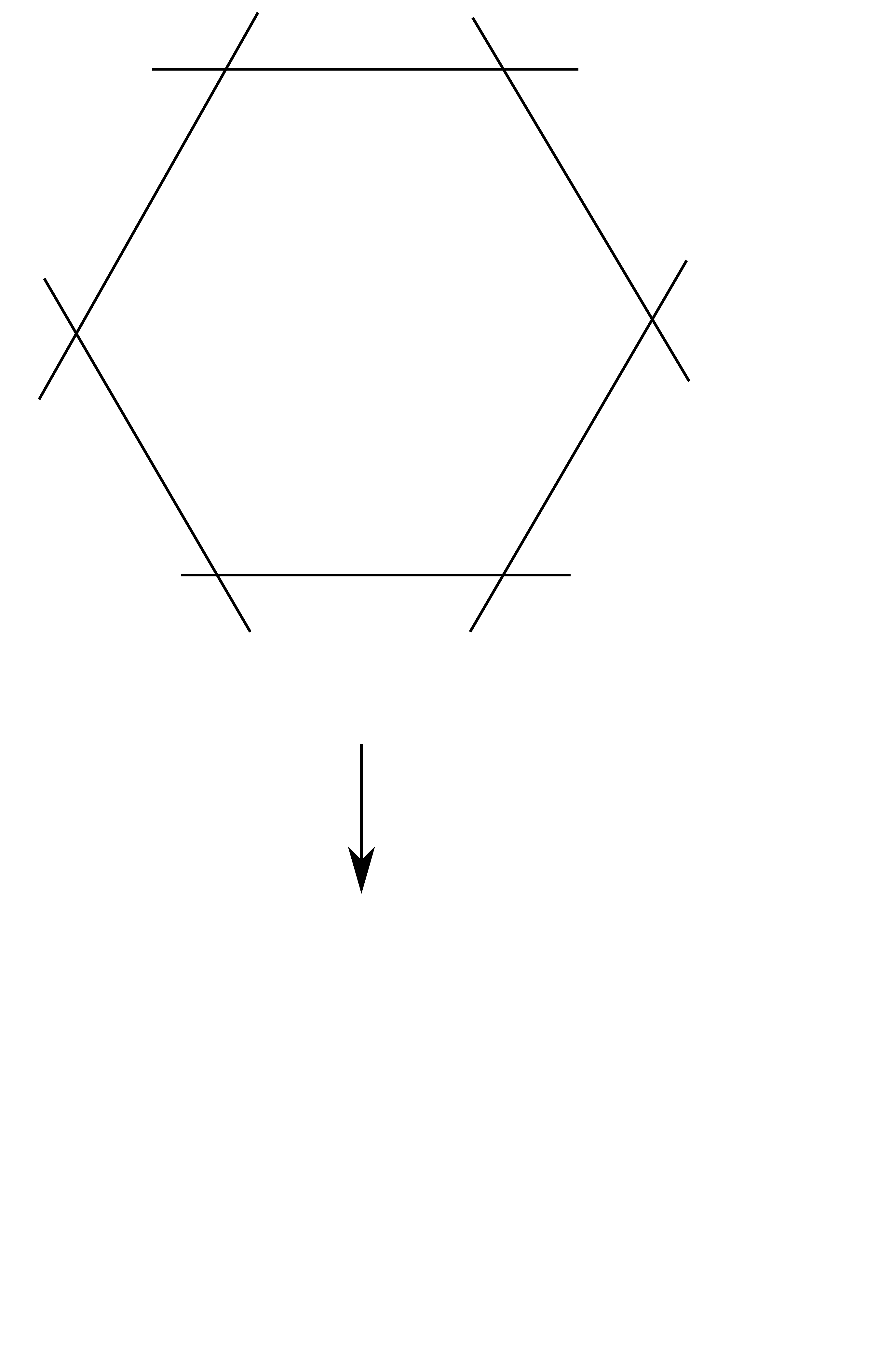}}%
    \put(0.70497319,1.36750687){\color[rgb]{0,0,0}\makebox(0,0)[lb]{\smash{\SB{$E_p$}}}}%
    \put(0.71966013,1.00033334){\color[rgb]{0,0,0}\makebox(0,0)[lb]{\smash{\SB{$f_p$}}}}%
    \put(0.37598571,0.77903276){\color[rgb]{0,0,0}\makebox(0,0)[lb]{\smash{\SB{$\overline{f_p}$}}}}%
    \put(0.04406082,0.98270902){\color[rgb]{0,0,0}\makebox(0,0)[lb]{\smash{\SB{$E_{\bar{p}}$}}}}%
    \put(0.07049732,1.35281994){\color[rgb]{0,0,0}\makebox(0,0)[lb]{\smash{\SB{$f_{\bar{p}}$}}}}%
    \put(0.3524866,1.52955673){\color[rgb]{0,0,0}\makebox(0,0)[lb]{\smash{\SB{$\overline{f_{\bar{p}}}$}}}}%
    \put(0,0){\includegraphics[width=\unitlength,page=2]{DP6_2.pdf}}%
    \put(0.60803936,0.42801719){\color[rgb]{0,0,0}\makebox(0,0)[lb]{\smash{\SB{$p$}}}}%
    \put(0.10868337,0.21358787){\color[rgb]{0,0,0}\makebox(0,0)[lb]{\smash{\SB{$\bar{p}$}}}}%
    \put(0,0){\includegraphics[width=\unitlength,page=3]{DP6_2.pdf}}%
    \put(0.60803936,0.42801719){\color[rgb]{0,0,0}\makebox(0,0)[lb]{\smash{\SB{$p$}}}}%
    \put(0.10868337,0.21358787){\color[rgb]{0,0,0}\makebox(0,0)[lb]{\smash{\SB{$\bar{p}$}}}}%
    \put(0.93482029,0.48088073){\color[rgb]{0,0,0}\makebox(0,0)[lb]{\smash{\SB{$\overline{f_{\bar{p}}}$}}}}%
    \put(0.94436322,0.13785683){\color[rgb]{0,0,0}\makebox(0,0)[lb]{\smash{\SB{$\overline{f_p}$}}}}%
    \put(0.65603922,0.00969849){\color[rgb]{0,0,0}\makebox(0,0)[lb]{\smash{\SB{$f_p$}}}}%
    \put(0.16570477,0.01009218){\color[rgb]{0,0,0}\makebox(0,0)[lb]{\smash{\SB{$f_{\bar{p}}$}}}}%
    \put(0.47291951,0.65959643){\color[rgb]{0,0,0}\makebox(0,0)[lb]{\smash{\SB{$p,\bar{p}$}\\ }}}%
    \put(0.81540063,1.46728657){\color[rgb]{0,0,0}\makebox(0,0)[lb]{\smash{$X_{[2]}$}}}%
    \put(0.91156951,0.61847318){\color[rgb]{0,0,0}\makebox(0,0)[lb]{\smash{$\QQ_{3,1}$}}}%
  \end{picture}%
\endgroup%

\captionof{figure}{$\rk(\Pic(X_{[2]}))=2$}\label{fig:X_2}
\end{minipage}
\begin{minipage}[h]{0.45\textwidth}
\def\svgwidth{1.4\textwidth}
\begingroup%
  \makeatletter%
  \providecommand\color[2][]{%
    \errmessage{(Inkscape) Color is used for the text in Inkscape, but the package 'color.sty' is not loaded}%
    \renewcommand\color[2][]{}%
  }%
  \providecommand\transparent[1]{%
    \errmessage{(Inkscape) Transparency is used (non-zero) for the text in Inkscape, but the package 'transparent.sty' is not loaded}%
    \renewcommand\transparent[1]{}%
  }%
  \providecommand\rotatebox[2]{#2}%
  \ifx\svgwidth\undefined%
    \setlength{\unitlength}{1627.58748693bp}%
    \ifx\svgscale\undefined%
      \relax%
    \else%
      \setlength{\unitlength}{\unitlength * \real{\svgscale}}%
    \fi%
  \else%
    \setlength{\unitlength}{\svgwidth}%
  \fi%
  \global\let\svgwidth\undefined%
  \global\let\svgscale\undefined%
  \makeatother%
  \begin{picture}(1,0.75213411)%
    \put(0,0){\includegraphics[width=\unitlength,page=1]{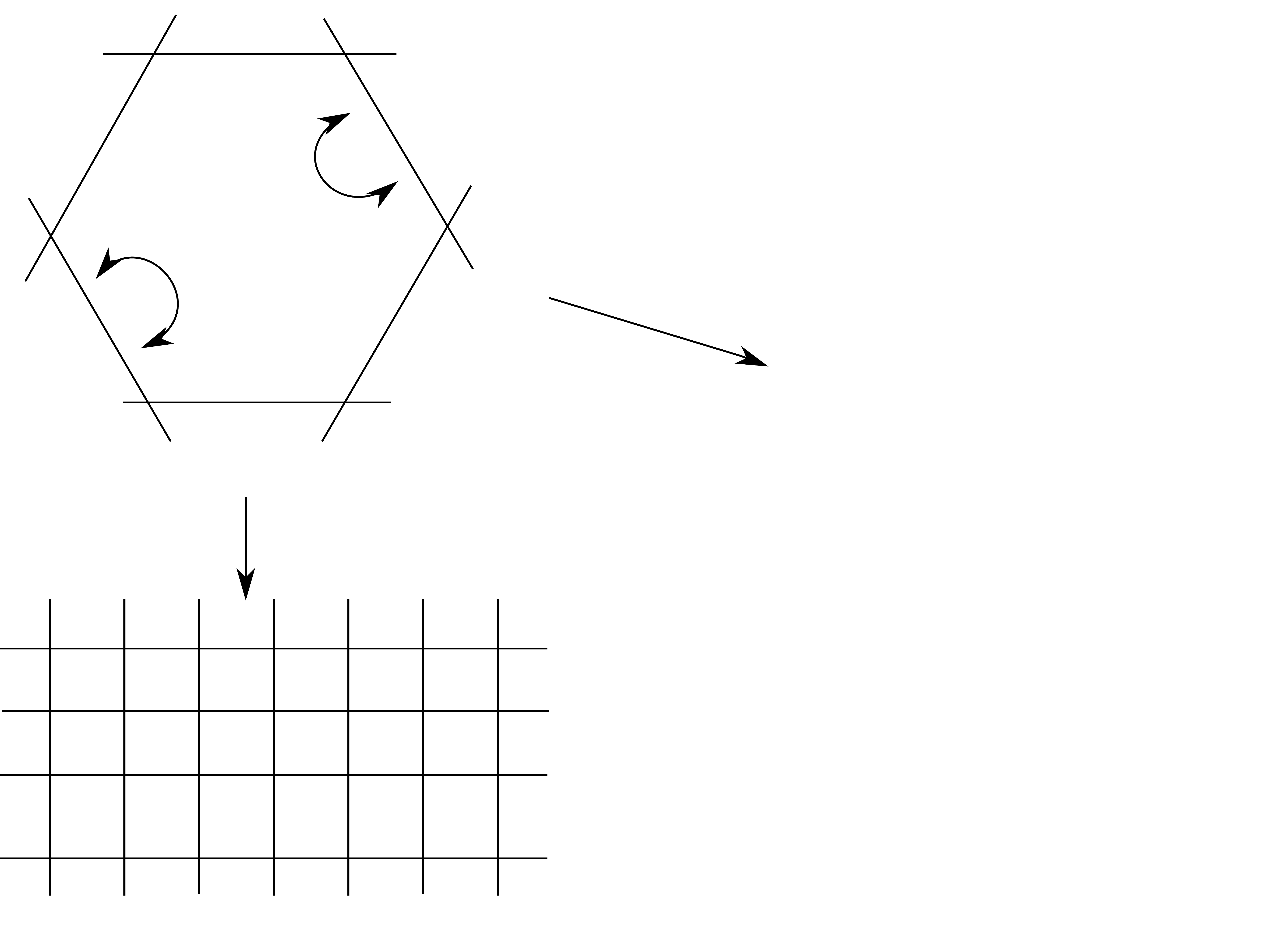}}%
    \put(0.17625669,0.39702473){\color[rgb]{0,0,0}\makebox(0,0)[lb]{\smash{\SB{$\overline{f_q}$}}}}%
    \put(0.17333001,0.73719636){\color[rgb]{0,0,0}\makebox(0,0)[lb]{\smash{\SB{$\overline{f_p}$}}}}%
    \put(0.33331794,0.6565991){\color[rgb]{0,0,0}\makebox(0,0)[lb]{\smash{\SB{$E_p$}}}}%
    \put(0.01926185,0.46794123){\color[rgb]{0,0,0}\makebox(0,0)[lb]{\smash{\SB{$E_q$}}}}%
    \put(0.33339282,0.48340046){\color[rgb]{0,0,0}\makebox(0,0)[lb]{\smash{\SB{$f_p$}}}}%
    \put(0.01633516,0.66267923){\color[rgb]{0,0,0}\makebox(0,0)[lb]{\smash{\SB{$f_q$}}}}%
    \put(0,0){\includegraphics[width=\unitlength,page=2]{DP6_3S.pdf}}%
    \put(0.05873996,0.09633965){\color[rgb]{0,0,0}\makebox(0,0)[lb]{\smash{\SB{$q$}}}}%
    \put(0.29182017,0.20679859){\color[rgb]{0,0,0}\makebox(0,0)[lb]{\smash{\SB{$p$}}}}%
    \put(0.45195998,0.2308132){\color[rgb]{0,0,0}\makebox(0,0)[lb]{\smash{\SB{$\overline{f_p}$}}}}%
    \put(0.45691267,0.07044875){\color[rgb]{0,0,0}\makebox(0,0)[lb]{\smash{\SB{$\overline{f_q}$}}}}%
    \put(0.31635898,0.00463684){\color[rgb]{0,0,0}\makebox(0,0)[lb]{\smash{\SB{$f_p$}}}}%
    \put(0.08125277,0.00546262){\color[rgb]{0,0,0}\makebox(0,0)[lb]{\smash{\SB{$f_q$}}}}%
    \put(0.62153468,0.35152192){\color[rgb]{0,0,0}\makebox(0,0)[lb]{\smash{\SB{$s$}}}}%
    \put(0.64624029,0.43737905){\color[rgb]{0,0,0}\makebox(0,0)[lb]{\smash{\SB{$L$}}}}%
    \put(0.89724302,0.36211345){\color[rgb]{0,0,0}\makebox(0,0)[lb]{\smash{\SB{$\bar{s}$}}}}%
    \put(0,0){\includegraphics[width=\unitlength,page=3]{DP6_3S.pdf}}%
    \put(0.79473626,0.50595714){\color[rgb]{0,0,0}\makebox(0,0)[lb]{\smash{\SB{$r$}}}}%
    \put(0.74798414,0.33553541){\color[rgb]{0,0,0}\makebox(0,0)[lb]{\smash{\SB{$M$}}}}%
    \put(0.84456479,0.45110312){\color[rgb]{0,0,0}\makebox(0,0)[lb]{\smash{\SB{$\bar{L}$}}}}%
    \put(0,0){\includegraphics[width=\unitlength,page=4]{DP6_3S.pdf}}%
    \put(0.46927706,0.3483984){\color[rgb]{0,0,0}\makebox(0,0)[lt]{\begin{minipage}{0.41288103\unitlength}\raggedright \end{minipage}}}%
    \put(0.47910756,0.3483984){\color[rgb]{0,0,0}\makebox(0,0)[lt]{\begin{minipage}{0.39322003\unitlength}\raggedright \end{minipage}}}%
    \put(0.50859906,0.3582289){\color[rgb]{0,0,0}\makebox(0,0)[lt]{\begin{minipage}{0.39322003\unitlength}\raggedright \end{minipage}}}%
    \put(0.21497753,0.32604183){\color[rgb]{0,0,0}\makebox(0,0)[lb]{\smash{\SB{$p,q$}}}}%
    \put(0.50508635,0.51179143){\color[rgb]{0,0,0}\makebox(0,0)[lb]{\smash{\SB{$r,s,\bar{s}$}}}}%
    \put(0,0){\includegraphics[width=\unitlength,page=5]{DP6_3S.pdf}}%
    \put(-0.00129349,0.74875338){\color[rgb]{0,0,0}\makebox(0,0)[lt]{\begin{minipage}{0.019661\unitlength}\raggedright \end{minipage}}}%
    \put(-0.00129349,0.74875338){\color[rgb]{0,0,0}\makebox(0,0)[lt]{\begin{minipage}{0.92406709\unitlength}\raggedright \end{minipage}}}%
    \put(0.44107587,0.28890691){\color[rgb]{0,0,0}\makebox(0,0)[lb]{\smash{$\QQ_{3,1}$}}}%
    \put(0.40175386,0.70506372){\color[rgb]{0,0,0}\makebox(0,0)[lb]{\smash{$X_{[3,\QQ_{3,1}]}$}}}%
    \put(0.82446858,0.57180441){\color[rgb]{0,0,0}\makebox(0,0)[lb]{\smash{$\PP^2$}}}%
  \end{picture}%
\endgroup%

\captionof{figure}{$\rk(\Pic(X_{[3,\QQ_{3,1}]}))=3$}\label{fig:X_3,S}
\end{minipage}
\end{center}
\vskip\baselineskip
\begin{center}
\begin{minipage}[h]{0.42\textwidth}
\def\svgwidth{.85\textwidth}
\begingroup%
  \makeatletter%
  \providecommand\color[2][]{%
    \errmessage{(Inkscape) Color is used for the text in Inkscape, but the package 'color.sty' is not loaded}%
    \renewcommand\color[2][]{}%
  }%
  \providecommand\transparent[1]{%
    \errmessage{(Inkscape) Transparency is used (non-zero) for the text in Inkscape, but the package 'transparent.sty' is not loaded}%
    \renewcommand\transparent[1]{}%
  }%
  \providecommand\rotatebox[2]{#2}%
  \ifx\svgwidth\undefined%
    \setlength{\unitlength}{912.07812692bp}%
    \ifx\svgscale\undefined%
      \relax%
    \else%
      \setlength{\unitlength}{\unitlength * \real{\svgscale}}%
    \fi%
  \else%
    \setlength{\unitlength}{\svgwidth}%
  \fi%
  \global\let\svgwidth\undefined%
  \global\let\svgscale\undefined%
  \makeatother%
  \begin{picture}(1,1.23332594)%
    \put(0,0){\includegraphics[width=\unitlength,page=1]{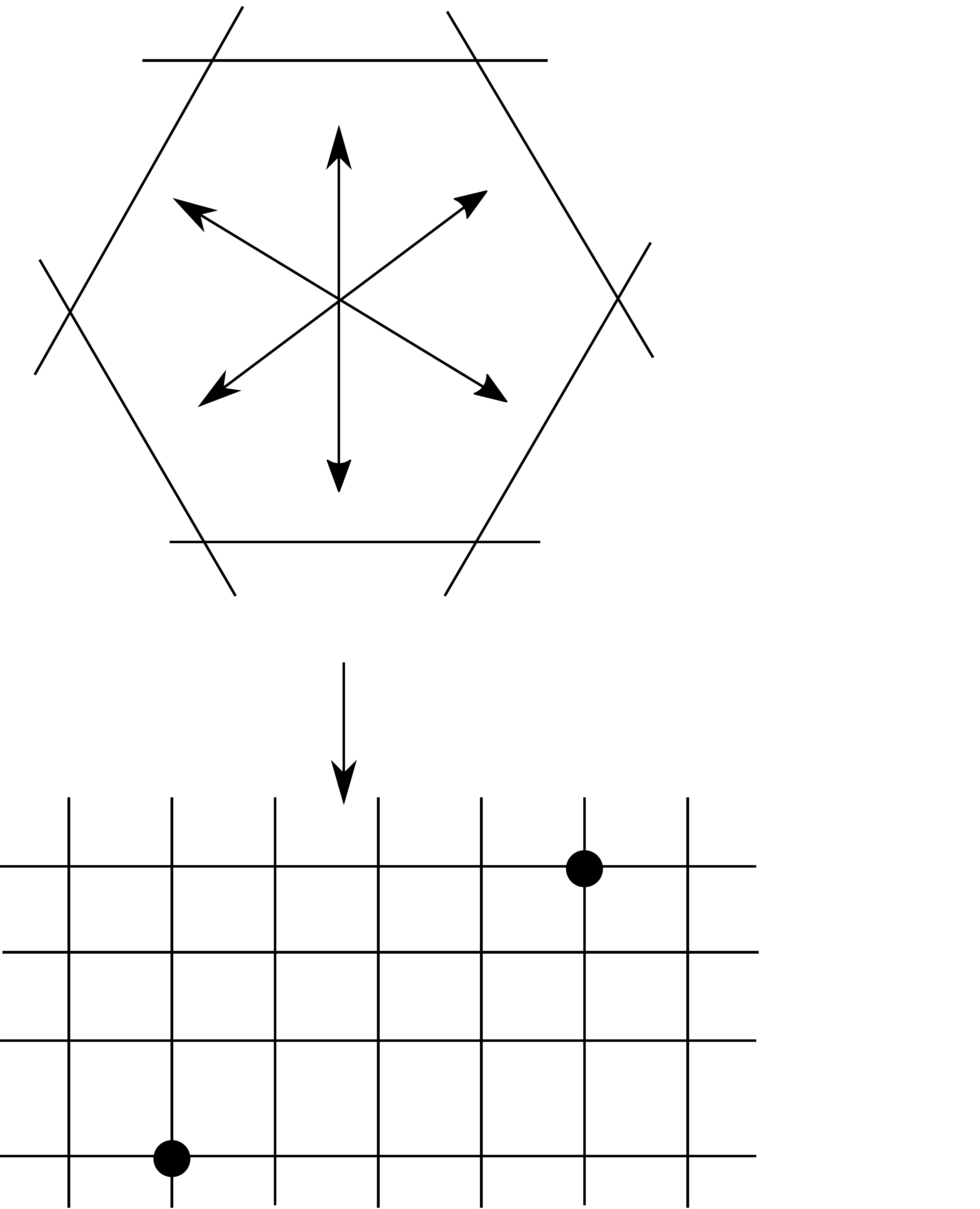}}%
    \put(0.10525415,0.10651336){\color[rgb]{0,0,0}\makebox(0,0)[lb]{\smash{\SB{$\bar{p}$}}}}%
    \put(0.52376469,0.28694899){\color[rgb]{0,0,0}\makebox(0,0)[lb]{\smash{\SB{$p$}}}}%
    \put(0.80694842,0.34208214){\color[rgb]{0,0,0}\makebox(0,0)[lb]{\smash{\SB{$g_p$}}}}%
    \put(0.80193631,0.03884996){\color[rgb]{0,0,0}\makebox(0,0)[lb]{\smash{\SB{$\overline{g_p}$}}}}%
    \put(0.55383729,-0.08394643){\color[rgb]{0,0,0}\makebox(0,0)[lb]{\smash{\SB{$f_p$}}}}%
    \put(0.12530255,-0.08394643){\color[rgb]{0,0,0}\makebox(0,0)[lb]{\smash{\SB{$\overline{f_p}$}}}}%
    \put(0.58390989,1.07886113){\color[rgb]{0,0,0}\makebox(0,0)[lb]{\smash{\SB{$E_p$}}}}%
    \put(0.04260287,0.75056844){\color[rgb]{0,0,0}\makebox(0,0)[lb]{\smash{\SB{$E_{\bar{p}}$}}}}%
    \put(0.31576248,0.59769935){\color[rgb]{0,0,0}\makebox(0,0)[lb]{\smash{\SB{$\overline{g_p}$}}}}%
    \put(0.60145225,0.77312291){\color[rgb]{0,0,0}\makebox(0,0)[lb]{\smash{\SB{$f_p$}}}}%
    \put(0.31576243,1.20666973){\color[rgb]{0,0,0}\makebox(0,0)[lb]{\smash{\SB{$g_p$}}}}%
    \put(0.04761497,1.04878851){\color[rgb]{0,0,0}\makebox(0,0)[lb]{\smash{\SB{$\overline{f_p}$}}}}%
    \put(0.38091975,0.50497546){\color[rgb]{0,0,0}\makebox(0,0)[lb]{\smash{\SB{$p,\bar{p}$}}}}%
    \put(0.71923664,1.16190648){\color[rgb]{0,0,0}\makebox(0,0)[lb]{\smash{$X_{[3,\FF_0]}$}}}%
    \put(0.81697264,0.44629925){\color[rgb]{0,0,0}\makebox(0,0)[lb]{\smash{$\FF_0$}}}%
  \end{picture}%
\endgroup%
\vspace{3mm}
\captionof{figure}{$\rk(\Pic(X_{[3,\FF_0]}))=3$}\label{fig:X_3,T}
\end{minipage}
\begin{minipage}[h]{0.43\textwidth}
\def\svgwidth{1.15\textwidth}
\begingroup%
  \makeatletter%
  \providecommand\color[2][]{%
    \errmessage{(Inkscape) Color is used for the text in Inkscape, but the package 'color.sty' is not loaded}%
    \renewcommand\color[2][]{}%
  }%
  \providecommand\transparent[1]{%
    \errmessage{(Inkscape) Transparency is used (non-zero) for the text in Inkscape, but the package 'transparent.sty' is not loaded}%
    \renewcommand\transparent[1]{}%
  }%
  \providecommand\rotatebox[2]{#2}%
  \ifx\svgwidth\undefined%
    \setlength{\unitlength}{1319.30099848bp}%
    \ifx\svgscale\undefined%
      \relax%
    \else%
      \setlength{\unitlength}{\unitlength * \real{\svgscale}}%
    \fi%
  \else%
    \setlength{\unitlength}{\svgwidth}%
  \fi%
  \global\let\svgwidth\undefined%
  \global\let\svgscale\undefined%
  \makeatother%
  \begin{picture}(1,0.92788838)%
    \put(0,0){\includegraphics[width=\unitlength,page=1]{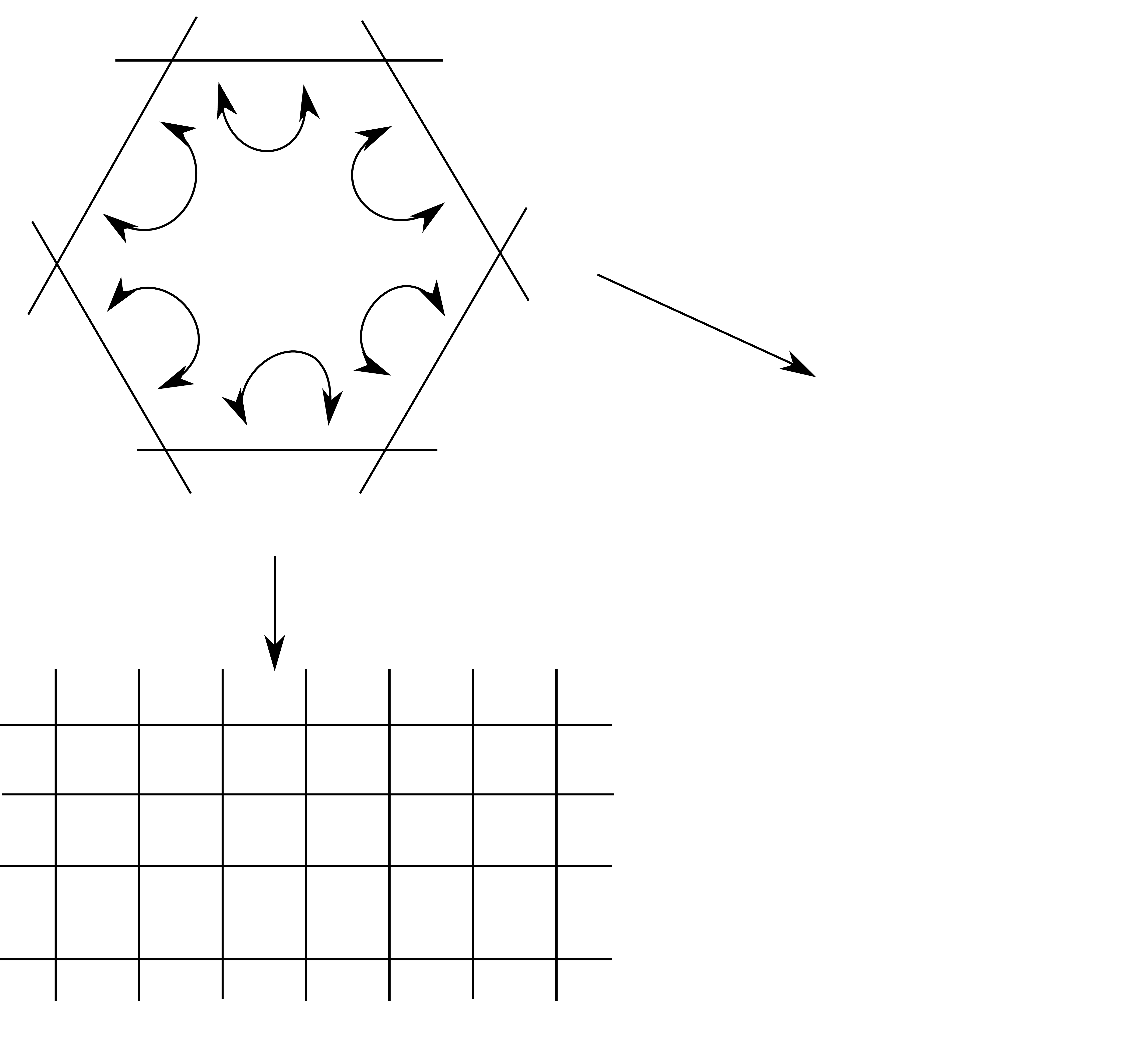}}%
    \put(0.21744331,0.48979913){\color[rgb]{0,0,0}\makebox(0,0)[lb]{\smash{\SB{$g_q$}}}}%
    \put(0.21383274,0.90946006){\color[rgb]{0,0,0}\makebox(0,0)[lb]{\smash{\SB{$g_p$}}}}%
    \put(0.41120572,0.81002931){\color[rgb]{0,0,0}\makebox(0,0)[lb]{\smash{\SB{$E_p$}}}}%
    \put(0.02376285,0.57728698){\color[rgb]{0,0,0}\makebox(0,0)[lb]{\smash{\SB{$E_q$}}}}%
    \put(0.41129809,0.59635864){\color[rgb]{0,0,0}\makebox(0,0)[lb]{\smash{\SB{$f_p$}}}}%
    \put(0.02015227,0.81753021){\color[rgb]{0,0,0}\makebox(0,0)[lb]{\smash{\SB{$f_q$}}}}%
    \put(0,0){\includegraphics[width=\unitlength,page=2]{DP6_4.pdf}}%
    \put(0.07246597,0.11885173){\color[rgb]{0,0,0}\makebox(0,0)[lb]{\smash{\SB{$q$}}}}%
    \put(0.36001099,0.25512206){\color[rgb]{0,0,0}\makebox(0,0)[lb]{\smash{\SB{$p$}}}}%
    \put(0.55757132,0.28474827){\color[rgb]{0,0,0}\makebox(0,0)[lb]{\smash{\SB{$g_p$}}}}%
    \put(0.56368133,0.0869108){\color[rgb]{0,0,0}\makebox(0,0)[lb]{\smash{\SB{$g_q$}}}}%
    \put(0.39028388,0.00572034){\color[rgb]{0,0,0}\makebox(0,0)[lb]{\smash{\SB{$f_p$}}}}%
    \put(0.10023944,0.0067391){\color[rgb]{0,0,0}\makebox(0,0)[lb]{\smash{\SB{$f_q$}}}}%
    \put(0.64628515,0.47418358){\color[rgb]{0,0,0}\makebox(0,0)[lb]{\smash{\SB{$r_1$}}}}%
    \put(0.68034734,0.55697617){\color[rgb]{0,0,0}\makebox(0,0)[lb]{\smash{\SB{$L_2$}}}}%
    \put(0,0){\includegraphics[width=\unitlength,page=3]{DP6_4.pdf}}%
    \put(0.95140692,0.48124854){\color[rgb]{0,0,0}\makebox(0,0)[lb]{\smash{\SB{$r_2$}}}}%
    \put(0.82352019,0.65888883){\color[rgb]{0,0,0}\makebox(0,0)[lb]{\smash{\SB{$r_3$}}}}%
    \put(0.7811617,0.46906846){\color[rgb]{0,0,0}\makebox(0,0)[lb]{\smash{\SB{$L_3$}}}}%
    \put(0.86967396,0.58270664){\color[rgb]{0,0,0}\makebox(0,0)[lb]{\smash{\SB{$L_1$}}}}%
    \put(0,0){\includegraphics[width=\unitlength,page=4]{DP6_4.pdf}}%
    \put(0.57893495,0.42981009){\color[rgb]{0,0,0}\makebox(0,0)[lt]{\begin{minipage}{0.50936064\unitlength}\raggedright \end{minipage}}}%
    \put(0.59106259,0.42981009){\color[rgb]{0,0,0}\makebox(0,0)[lt]{\begin{minipage}{0.48510537\unitlength}\raggedright \end{minipage}}}%
    \put(0.62744549,0.44193773){\color[rgb]{0,0,0}\makebox(0,0)[lt]{\begin{minipage}{0.48510537\unitlength}\raggedright \end{minipage}}}%
    \put(0.26771165,0.39972993){\color[rgb]{0,0,0}\makebox(0,0)[lb]{\smash{\SB{$p,q$}}}}%
    \put(0.57553634,0.68822189){\color[rgb]{0,0,0}\makebox(0,0)[lb]{\smash{\SB{$r_1,r_2,r_3$}}}}%
    \put(0.56840307,0.34340519){\color[rgb]{0,0,0}\makebox(0,0)[lb]{\smash{$\FF_0$}}}%
    \put(0.47521278,0.86733013){\color[rgb]{0,0,0}\makebox(0,0)[lb]{\smash{$X_{[4]}$}}}%
    \put(0.82308339,0.74180312){\color[rgb]{0,0,0}\makebox(0,0)[lb]{\smash{$\PP^2$}}}%
  \end{picture}%
\endgroup%

\captionof{figure}{$\rk(\Pic(X_{[4]}))=4$}\label{fig:X_4}
\end{minipage}
\end{center}


\subsection{The surfaces obtained by blowing up the sphere}\label{ssec:DP6S}
Blowing up the sphere in a pair of non-real conjugate points $p,\bar{p}$, we obtain the del Pezzo surface $X_{[2]}$ with $\rk(\Pic(X_{[2]}))=2$. The lift of the antiholomorphic involution is indicated in Figure~\ref{fig:X_2}. \par

\begin{Rmk}\label{rmk:O2}
For a non-real point $p\in \QQ_{3,1}$, we denote by $\Aut_\R(\QQ_{3,1},p,\bar{p})\subset\Aut_\R(\QQ_{3,1})$ the subgroup of $\Aut_\R(\QQ_{3,1})$ that fixes both points $p$ and $\bar{p}$. Choosing $p=([1:0],[0:1])$, it is isomorphic to the group $\{(d,\bar{d})\in\mathrm{PGL}_2(\C)\times\mathrm{PGL}_2(\C)\mid d\ \text{diagonal}\}$ \cite[Lemma 4.5]{R15}. 
Conjugating with the real birational map $\QQ_{3,1}\dashrightarrow(\PP^1\times\PP^1,\sigma_S)$ from Remark~\ref{rmk:real part}, we obtain that 
\[\Aut_\R(\QQ_{3,1},p,\bar{p})\simeq\{(d,\bar{d})\in\mathrm{PGL}_2(\C)\times\mathrm{PGL}_2(\C)\mid d\ \text{diagonal}\}\simeq\R_{>0}\times\mathrm{SO}_2(\R).\]
\end{Rmk}

\begin{Prop}\label{prop:X_2}\item
\begin{enumerate}
\item\label{X_2 1} The surface $X_{[2]}$ is isomorphic to 
\[X_{[2]}\simeq\{([w:x:y:z],[u:v])\in\PP^3\times\PP^1\mid wz=x^2+y^2,uz=vw\}.\]
\item\label{X_2 2} There is an exact sequence
\[1\rightarrow\ker(\rho)\rightarrow\Aut_\R(X_{[2]})\stackrel{\rho}\rightarrow \Z/2\Z\times\Z/2\Z\rightarrow 1\]
where $\ker(\rho)\simeq\Aut_\R(\QQ_{3,1},p,\bar{p})\simeq\R_{>0}\times\mathrm{SO}_2(\R)$ and $\Z/2\Z\times\Z/2\Z\simeq\langle\rho(\alpha_1)\rangle\times\langle\rho(\alpha_2)\rangle$, where
\[\alpha_1\colon([w:x:y:z],[u:v])\mapsto([z:-x:y:w],[v:u]),\]
\[\alpha_2\colon([w:x:y:z],[u:v])\mapsto([w:-x:y:z],[u:v]),\]
and $\rho(\alpha_1)$ is a rotation of order $2$ and $\rho(\alpha_2)$ is a reflection, both exchanging $E_p$ and $E_{\bar p}$.
\item\label{X_2 3} The automorphisms $\alpha_1,\alpha_2$ are lifts of elements of $\Aut(\QQ_{3,1},\{p,\bar{p}\})$. In particular, the pair $(X_{[2]},\Aut_\R(X_{[2]}))$ is not a minimal pair and the contraction morphism $X_{[2]}\rightarrow \QQ_{3,1}$ induces an embedding $\Aut_\R(X_{[2]})\hookrightarrow\Aut(\QQ_{3,1})$.
\end{enumerate}
\end{Prop}
\begin{proof}
(\ref{X_2 1}) is Lemma~\ref{lem:X_2}. Let $S\subset\PP^3$ be the surface given by $wz=x^2+z^2$ and $\psi\colon S\rightarrow\QQ_{3,1}$, $\psi\colon[w:x:y:z]\mapsto[w+z:2x:2y:w-z]$. Any automorphism of $X_{[2]}$ preserves the hexagon in Figure~\ref{fig:X_2} and so $\rho(\Aut_\R(X_{[2]}))$ is contained in $D_6$. The action of the antiholomorphic involution indicated in Figure~\ref{fig:X_2} implies that any element of $\Aut_\R(X_{[2]})$ preserves the set $\{E_p,E_{\bar{p}}\}$. The kernel of $\rho$ is contained in the subgroup of $\Aut_\R(X_{[2]})$ fixing $E_p$ and $E_{\bar{p}}$. Any such automorphism descends to an automorphism of $S$ fixing both points $p,\bar{p}$. Any element of $\Aut_\R(\QQ_{3,1},p,\bar{p})$ also fixes $f_p$ and $f_{\bar p}$ and thus lifts to an element of $\ker(\rho)$. It follows that $\ker(\rho)\simeq\Aut_\R(\QQ_{3,1},p,\bar{p})\simeq \R_{>0}\times\mathrm{SO}_2(\R)$ (see Remark~\ref{rmk:O2} for the isomorphisms).\par
The only non-trivial elements of $D_6$ preserving $\{E_p, E_{\bar p}\}$ and respecting the action of the antiholomorphic involution are the rotation of order $2$ and two reflections, one exchanging $E_p,E_{\bar p}$ and one fixing them. The automorphisms
\[\alpha_1\colon([w:x:y:z],[u:v])\mapsto([z:-x:y:w],[v:u])\]
and
\[\alpha_2\colon([w:x:y:z],[u:v])\mapsto([w:-x:y:z],[u:v])\]
are the lifts of automorphisms of $S$ exchanging $p$ and $\bar{p}$ and hence exchange $E_p$ and $E_{\bar{p}}$. In fact, via the $\R$-isomorphism $S\stackrel{\varphi\psi^{-1}}\longrightarrow(\PP^1\times\PP^1,\sigma_S)$, where $\varphi$ is as in Remark~\ref{rmk:real part}, $\alpha_1$ and $\alpha_2$ are conjugate to
\begin{align*}\varphi\psi^{-1}\alpha_1\psi\varphi^{-1}\colon&([x_0:x_1],[y_0:y_1])\mapsto([x_1:x_0],[y_1:y_0])\\
\varphi\psi^{-1}\alpha_2\psi\varphi^{-1}\colon&([x_0:x_1],[y_0:y_1])\mapsto([y_0:y_1],[x_0:x_1])
\end{align*}
and $\varphi\psi^{-1}(p)=([0:1],[1:0])$. This description implies that $\rho(\alpha_1)$ is a rotation of order $2$ and $\rho(\alpha_2)$ is the reflection exchanging $E_p,E_{\bar p}$. We have $\alpha_1^2=\alpha_2^2=(\alpha_1\alpha_2)^2=\Id$ and hence $\rho\colon\Aut_\R(X_{[4]})\rightarrow\Z/2\Z\times\Z/2\Z$ is surjective and has a section. On the other hand, it follows that every element of $\Aut_\R(X_{[2]})$ is the lift of an element of $\Aut_\R(\QQ_{3,1})$, which yields (\ref{X_2 3}).
\end{proof}

Blowing up the sphere in two real points $p,q$, we obtain the del Pezzo surface $X_{[3,\QQ_{3,1}]}$ with $\rk(\Pic(X_{[3,\QQ_{3,1}]}))=3$. The lift of the antiholomorphic involution is indicated in Figure~\ref{fig:X_3,S}. Contracting two non-real conjugate $(-1)$-curves and one real $(-1)$-curve that are pairwise disjoint, we obtain a birational morphism $X_{[3,\QQ_{3,1}]}\rightarrow\PP^2$ which is the blow-up of a real point and a pair of non-real conjugate points on $\PP^2$. 

\begin{Prop}\label{prop:DP63S}\item
\begin{enumerate}
\item\label{DP63S 1} The surface $X_{[3,\QQ_{3,1}]}$ is isomorphic to 
\[\{([x_0:x_1:x_2],[y_0:y_1:y_2]\in\PP^2\times\PP^2\mid x_0y_0=x_1y_2+x_2y_1,x_1y_1=x_2y_2\}.\] 
\item\label{DP63S 2} There is a split exact sequence
\[1\rightarrow\ker(\rho)\rightarrow\Aut_\R(X_{[3,\QQ_{3,1}]})\stackrel{\rho}\rightarrow \Z/2\Z\times\Z/2\Z\rightarrow 1,\]
where $\ker(\rho)\simeq\mathrm{SO}_2(\R)$ and $\Z/2\Z\times\Z/2\Z\simeq\langle\rho(\alpha_1)\rangle\times\langle\rho(\alpha_2)\rangle$, where
\begin{align*}&\alpha_1\colon([x_0:x_1:x_2],[y_0:y_1:y_2])\mapsto([y_0:y_1:y_2],[x_0:x_1:x_2]),\\
&\alpha_2\colon([x_0:x_1:x_2],[y_0:y_1:y_2])\mapsto([x_0:x_2:x_1],[y_0:y_2:y_1]),
\end{align*}
and $\rho(\alpha_1)$ is a rotation of order $2$ and $\rho(\alpha_2)$ a reflection fixing $E_p,E_q$.
\item\label{DP63S 3} The automorphisms $\alpha_1,\alpha_2$ are lifts of elements of $\Aut(\QQ_{3,1})$. In particular, $(X_{[3,\QQ_{3,1}]},\Aut_\R(X_{[3,\QQ_{3,1}]}))$ is not a minimal pair, and the contraction $X_{[3,\QQ_{3,1}]}\rightarrow \QQ_{3,1}$ induces an embedding $\Aut_\R(X_{[3,\QQ_{3,1}]})\hookrightarrow\Aut(\QQ_{3,1})$.
\end{enumerate}
\end{Prop}
\begin{minipage}{1\columnwidth}
\centering
\def\svgwidth{0.5\textwidth}
\begingroup%
  \makeatletter%
  \providecommand\color[2][]{%
    \errmessage{(Inkscape) Color is used for the text in Inkscape, but the package 'color.sty' is not loaded}%
    \renewcommand\color[2][]{}%
  }%
  \providecommand\transparent[1]{%
    \errmessage{(Inkscape) Transparency is used (non-zero) for the text in Inkscape, but the package 'transparent.sty' is not loaded}%
    \renewcommand\transparent[1]{}%
  }%
  \providecommand\rotatebox[2]{#2}%
  \ifx\svgwidth\undefined%
    \setlength{\unitlength}{276.51041587bp}%
    \ifx\svgscale\undefined%
      \relax%
    \else%
      \setlength{\unitlength}{\unitlength * \real{\svgscale}}%
    \fi%
  \else%
    \setlength{\unitlength}{\svgwidth}%
  \fi%
  \global\let\svgwidth\undefined%
  \global\let\svgscale\undefined%
  \makeatother%
  \begin{picture}(1,0.96349711)%
    \put(0,0){\includegraphics[width=\unitlength,page=1]{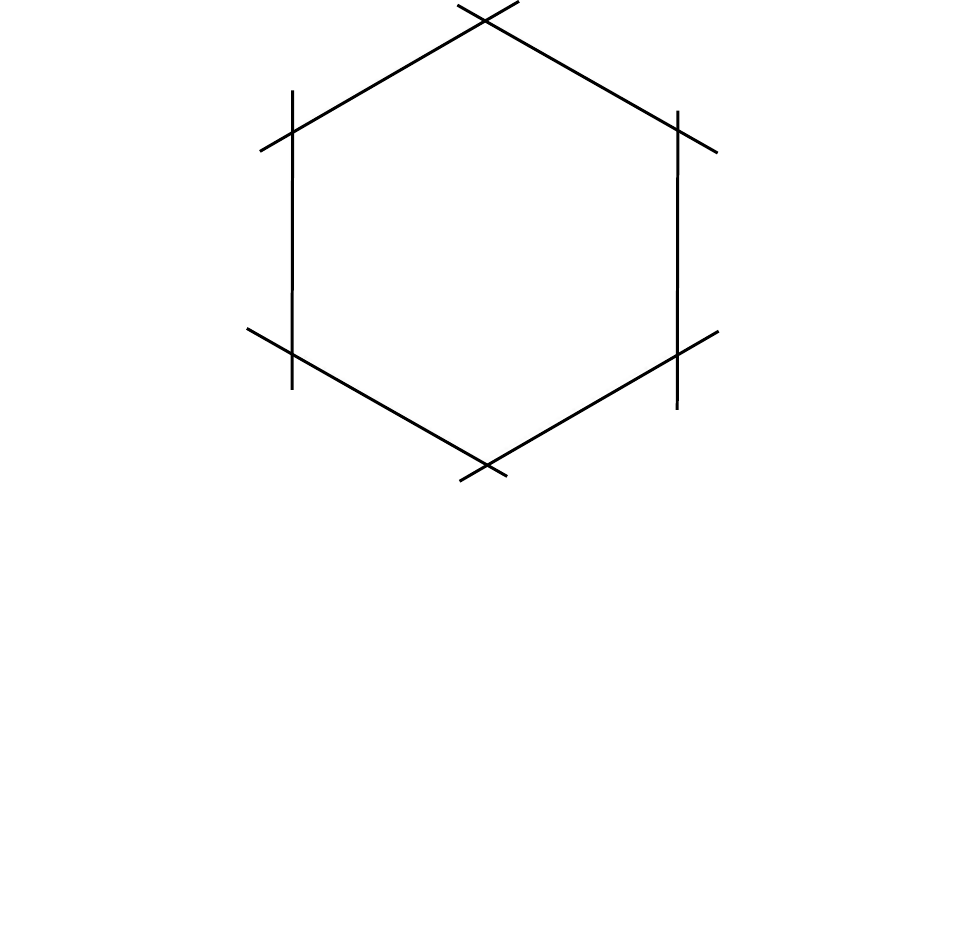}}%
    \put(0.60497763,0.90457009){\color[rgb]{0,0,0}\makebox(0,0)[lb]{\smash{\SB{$E_p$}}}}%
    \put(0.71279448,0.70509753){\color[rgb]{0,0,0}\makebox(0,0)[lb]{\smash{\SB{$f_p$}}}}%
    \put(0.60497763,0.49997415){\color[rgb]{0,0,0}\makebox(0,0)[lb]{\smash{\SB{$\overline{f_q}$}}}}%
    \put(0.41457804,0.90457009){\color[rgb]{0,0,0}\makebox(0,0)[rb]{\smash{\SB{$\overline{f_p}$}}}}%
    \put(0.41457804,0.49997415){\color[rgb]{0,0,0}\makebox(0,0)[rb]{\smash{\SB{$E_q$}}}}%
    \put(0.29507917,0.70509753){\color[rgb]{0,0,0}\makebox(0,0)[rb]{\smash{\SB{$f_q$}}}}%
    \put(0.39061339,0.8997852){\color[rgb]{0,0,0}\makebox(0,0)[lb]{\smash{}}}%
    \put(0,0){\includegraphics[width=\unitlength,page=2]{drawing3.pdf}}%
    \put(0.71264894,0.40097099){\color[rgb]{0,0,0}\makebox(0,0)[lb]{\smash{\SB{$E_q$, $f_p$, $\overline{f_p}$}}}}%
    \put(0.31230615,0.40011953){\color[rgb]{0,0,0}\makebox(0,0)[rb]{\smash{\SB{$E_p$, $f_q$, $\overline{f_q}$}}}}%
    \put(0,0){\includegraphics[width=\unitlength,page=3]{drawing3.pdf}}%
    \put(0.75226333,0.40794121){\color[rgb]{0,0,0}\makebox(0,0)[rb]{\smash{}}}%
    \put(0.33880241,0.40011953){\color[rgb]{0,0,0}\makebox(0,0)[lb]{\smash{\SB{$\varepsilon$}}}}%
    \put(0.68653967,0.38688092){\color[rgb]{0,0,0}\makebox(0,0)[rb]{\smash{\SB{$\eta$}}}}%
    \put(0,0){\includegraphics[width=\unitlength,page=4]{drawing3.pdf}}%
    \put(0.25735604,0.24218462){\color[rgb]{0,0,0}\makebox(0,0)[lb]{\smash{\SB{$r$}}}}%
    \put(0.40772246,0.03044418){\color[rgb]{0,0,0}\makebox(0,0)[lb]{\smash{\SB{$\bar{s}$}}}}%
    \put(0.09471485,0.0294213){\color[rgb]{0,0,0}\makebox(0,0)[lb]{\smash{\SB{$s$}}}}%
    \put(0,0){\includegraphics[width=\unitlength,page=5]{drawing3.pdf}}%
    \put(0.59082165,0.03146706){\color[rgb]{0,0,0}\makebox(0,0)[lb]{\smash{\SB{$r$}}}}%
    \put(0.74527968,0.24116174){\color[rgb]{0,0,0}\makebox(0,0)[lb]{\smash{\SB{$\bar{s}$}}}}%
    \put(0.88746297,0.03555866){\color[rgb]{0,0,0}\makebox(0,0)[lb]{\smash{\SB{$s$}}}}%
    \put(0,0){\includegraphics[width=\unitlength,page=6]{drawing3.pdf}}%
  \end{picture}%
\endgroup%

\captionof{figure}{The surface $X_{[3,\QQ_{3,1}]}$ and the blow-ups $\varepsilon$ and $\eta$.}
\label{Fig:DP3S}
\end{minipage}
\begin{proof}
Let $X_{[3,\QQ_{3,1}]}\rightarrow \QQ_{3,1}$ be the blow-up of two real points $p,q$ on $\QQ_{3,1}$. On $X$ there are two pairs of non-real conjugate $(-1)$-curves -- they are the strict transforms of the fibres $f_p,\overline{f_p}$ and $f_q,\overline{f_q}$ passing through $p$ and $q$ (see Figure~\ref{fig:X_3,S}), and by abuse of notation we denote them by $f_p,\overline{f_p}$ and $f_q,\overline{f_q}$ as well. The contraction of the disjoint $(-1)$-curves $E_p,f_q,\overline{f_q}$ yields a real birational morphism $\varepsilon\colon X_{[3,\QQ_{3,1}]}\rightarrow\PP^2$. We call the images of the $(-1)$-curves $r,s,\bar{s}$ respectively. Composing with an automorphism of $\PP^2$, we may choose $r=[1:0:0]$ and $s=[0:1:{\bf i}]$. The contraction of the disjoint $(-1)$-curves $E_q,f_p,\overline{f_p}$ yields a real birational morphism $\eta\colon X_{[3,\QQ_{3,1}]}\rightarrow\PP^2$ and we can assume that they are contracted onto $r,s,\bar{s}$ as well. Our choice implies that the pencil of lines through $r,s,\bar{s}$ respectively is sent onto the pencil of lines through $r,s,\bar{s}$ respectively (see Figure~\ref{Fig:DP3S}). In fact, $\eta\varepsilon^{-1}$ is -- up to automorphisms of $\PP^2$ fixing the points $r,s,\bar{s}$ -- just the birational involution
\[\eta\varepsilon^{-1}\colon[x_0:x_1:x_2]\dashmapsto[x_1^2+x_2^2:x_0x_2:x_0x_1].\]
The blow-ups $\varepsilon$ and $\eta$ yield an injection $\varepsilon\times\eta\colon X_{[3,\QQ_{3,1}]}\rightarrow\PP^2\times\PP^2$ whose image is described in (\ref{DP63S 1}).\par
The kernel of $\rho$ is isomorphic to the subgroup of $\mathrm{PGL}_3(\R)$ fixing the points $r,s,\bar{s}$, which is isomorphic to $\mathrm{SO}_2(\R)$. \par
Any automorphism of $X_{[3,\QQ_{3,1}]}$ preserves the hexagon in Figure~\ref{fig:X_3,S} and is hence a subgroup of the dihedral group $D_6$. The action of the antiholomorphic involution indicated in Figure~\ref{fig:X_3,S} shows that any automorphism of $X_{[3,\QQ_{3,1}]}$ preserves the set $\{E_p,E_q\}$, which means that $\rho(\Aut_\R(X_{[3,\QQ_{3,1}]}))$ contains, besides the identity map, at most a rotation of order 2 (exchanging $E_p,E_q$) and two reflections, one fixing $E_p$ and $E_q$ and one exchanging them. The automorphism
\[\alpha_1\colon([x_0:x_1:x_2],[y_0:y_1:y_2])\mapsto([y_0:y_1:y_2],[x_0:x_1:x_2])\]
is the lift of  the real birational involution $\eta\varepsilon^{-1}$ of $\PP^2$ and exchanges $E_p,E_q$ and $f_p,f_q$ and so $\rho(\alpha_1)$ is a rotation of order $2$. The automorphism
\[\alpha_2\colon([x_0:x_1:x_2],[y_0:y_1:y_2])\mapsto([x_0:x_2:x_1],[y_0:y_2:y_1])\]
is the lift of a linear map of $\PP^2$ exchanging $s,\bar{s}$ and fixing $r$. It therefore fixes $E_p$ and $E_q$ and exchanges $f_p,\overline{f_p}$, which means that $\rho(\alpha_2)$ is a reflection. Moreover, $\alpha_3:=\alpha_1\alpha_2=\alpha_2\alpha_1$ is the reflection exchanging $E_p$ and $E_q$. The relations $\alpha_1^2=\alpha_2^2=(\alpha_2\alpha_1)^2=\Id$ imply that $\Aut_\R(X_{[3,\QQ_{3,1}]})\stackrel{\rho}\rightarrow\Z/2\Z\times\Z/2\Z$ is surjective and $\rho(\alpha_i)\mapsto\alpha_i$ is a section of $\rho$. This yields (\ref{DP63S 2}). \par
The automorphisms $\alpha_1,\alpha_2$ both preserve the set $\{E_p,E_q\}$ and descend via the contractions of $E_p$ and $E_q$ to automorphisms of $S$ that respectively exchange or fix the points $p,q$. This yields (\ref{DP63S 3}). 
\end{proof}

\subsection{The surfaces obtained by blowing up $\FF_0$}\label{ssec:DP6T}
Blowing up a pair of non-real conjugate points $p,\bar{p}$ on $\FF_0$, we obtain the del Pezzo surface $X_{[3,\FF_0]}$ with $\rk(\Pic(X_{[3,\FF_0]}))=3$. The lift of the action of the antiholomorphic involution is indicated in Figure~\ref{fig:X_3,T} by arrows. By $\Aut(\FF_0,p,\bar{p},\pr)\subset\Aut_\R(\FF_0)$ we denote the subgroup fixing $p$ and $\bar{p}$ and preserving the fibrations.

\begin{Prop}\label{prop:X_3,T}\item
\begin{enumerate}
\item\label{X_3,T 1} The surface $X_{[3,\FF_0]}$ is isomorphic to 
\[\{([x_0:x_1],[y_0:y_1],[z_0:z_1])\in\PP^1\times\PP^1\times\PP^1\mid x_0y_0z_1+x_0y_1z_0+x_1y_0z_0-x_1y_1z_1=0\}.\]
\item\label{X_3, T 2} There is a split exact sequence 
\[0\rightarrow \ker(\rho)\rightarrow \Aut_\R(X_{[3,\FF_0]})\xrightarrow\rho D_6\rightarrow1\] 
where $\ker(\rho)\simeq\Aut_\R(\FF_0,p,\bar{p},\pr)\simeq\mathrm{SO}_2(\R)\times\mathrm{SO}_2(\R)$.
\item\label{X_3, T 3} The group $\rho(\Aut_\R(X_{[3,\FF_0]}))\simeq D_6$ is generated by the reflection $\rho(\alpha_1)$ fixing $E_p$ and $E_{\bar p}$, where
\[\alpha_1\colon([x_0:x_1],[y_0:y_1],[z_0:z_1])\mapsto([y_0:y_1],[x_0:x_1],[z_0:z_1]),\]
and the rotation $\rho(\alpha_2)$ of order $6$, where
\[\alpha_2\colon([x_0:x_1],[y_0:y_1],[z_0:z_1])\mapsto([z_1:z_0],[x_0:-x_1],[y_1:y_0]).\]
\item\label{X_3, T 4} The pair $(X_{[3,\FF_0]},\Aut_\R(X_{[3,\FF_0]}))$ is a minimal pair.
\item\label{X_3,T 5} There is exactly one finite $\Aut_\R(X_{[3,\FF_0]})$-orbit on $X_{[3,\FF_0]}$, namely the one of the six intersection points of the $(-1)$-curves.
\end{enumerate}
\end{Prop}

\begin{minipage}{1\columnwidth}
\centering
\def\svgwidth{0.5\textwidth}
\begingroup%
  \makeatletter%
  \providecommand\color[2][]{%
    \errmessage{(Inkscape) Color is used for the text in Inkscape, but the package 'color.sty' is not loaded}%
    \renewcommand\color[2][]{}%
  }%
  \providecommand\transparent[1]{%
    \errmessage{(Inkscape) Transparency is used (non-zero) for the text in Inkscape, but the package 'transparent.sty' is not loaded}%
    \renewcommand\transparent[1]{}%
  }%
  \providecommand\rotatebox[2]{#2}%
  \ifx\svgwidth\undefined%
    \setlength{\unitlength}{276.51041587bp}%
    \ifx\svgscale\undefined%
      \relax%
    \else%
      \setlength{\unitlength}{\unitlength * \real{\svgscale}}%
    \fi%
  \else%
    \setlength{\unitlength}{\svgwidth}%
  \fi%
  \global\let\svgwidth\undefined%
  \global\let\svgscale\undefined%
  \makeatother%
  \begin{picture}(1,0.96349711)%
    \put(0,0){\includegraphics[width=\unitlength,page=1]{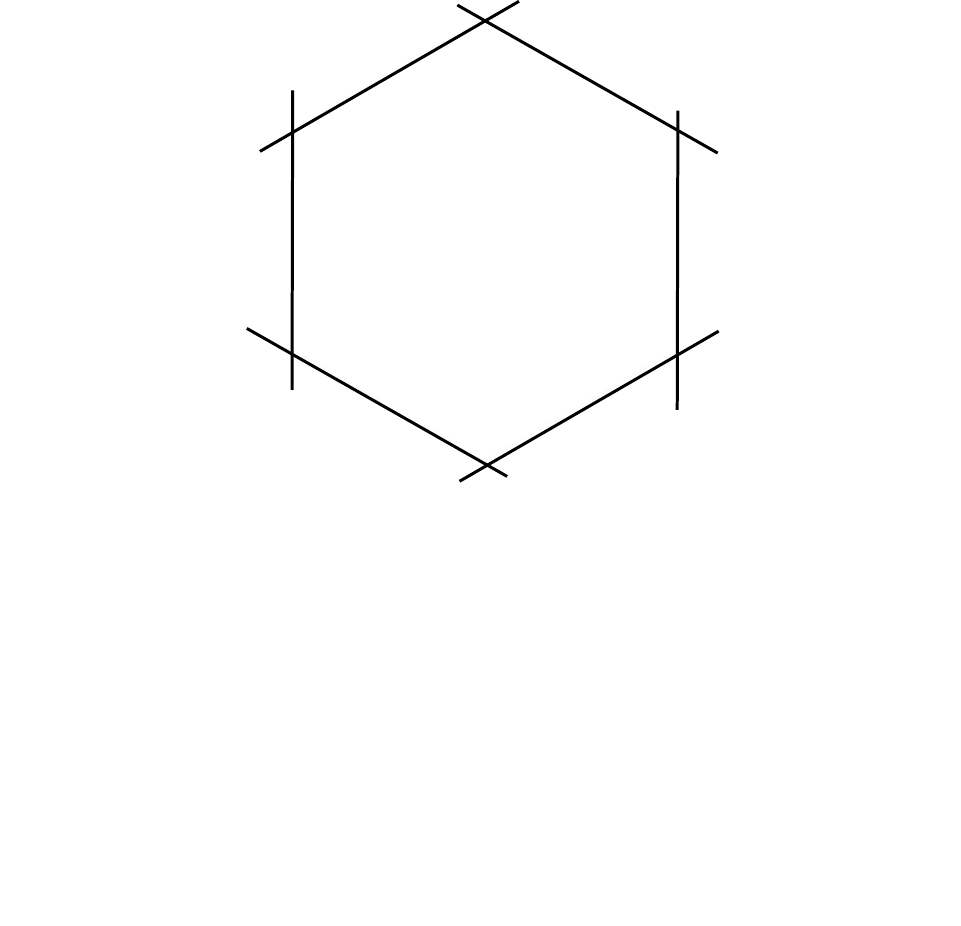}}%
    \put(0.60497763,0.90457009){\color[rgb]{0,0,0}\makebox(0,0)[lb]{\smash{\SB{$g_p$}}}}%
    \put(0.71279448,0.70509753){\color[rgb]{0,0,0}\makebox(0,0)[lb]{\smash{\SB{$E_p$}}}}%
    \put(0.60497763,0.49997415){\color[rgb]{0,0,0}\makebox(0,0)[lb]{\smash{\SB{$f_p$}}}}%
    \put(0.41457804,0.90457009){\color[rgb]{0,0,0}\makebox(0,0)[rb]{\smash{\SB{$\overline{f_p}$}}}}%
    \put(0.41457804,0.49997415){\color[rgb]{0,0,0}\makebox(0,0)[rb]{\smash{\SB{$\overline{g_p}$}}}}%
    \put(0.29507917,0.70509753){\color[rgb]{0,0,0}\makebox(0,0)[rb]{\smash{\SB{$E_{\bar p}$}}}}%
    \put(0.39061339,0.8997852){\color[rgb]{0,0,0}\makebox(0,0)[lb]{\smash{}}}%
    \put(0,0){\includegraphics[width=\unitlength,page=2]{drawing2.pdf}}%
    \put(0.71264894,0.40097099){\color[rgb]{0,0,0}\makebox(0,0)[lb]{\smash{\SBa{$g_p$, $\overline{g_p}$}}}}%
    \put(0.31230615,0.40011953){\color[rgb]{0,0,0}\makebox(0,0)[rb]{\smash{\SBa{$p$, $\bar p$}}}}%
    \put(0,0){\includegraphics[width=\unitlength,page=3]{drawing2.pdf}}%
    \put(0.96161044,0.13939778){\color[rgb]{0,0,0}\makebox(0,0)[lb]{\smash{\SB{$E_p$}}}}%
    \put(0.68589411,0.13939778){\color[rgb]{0,0,0}\makebox(0,0)[rb]{\smash{\SB{$E_{\bar p}$}}}}%
    \put(0.83064791,0.26851813){\color[rgb]{0,0,0}\makebox(0,0)[lb]{\smash{\SB{$\overline{f_p}$}}}}%
    \put(0.83064791,0.01391651){\color[rgb]{0,0,0}\makebox(0,0)[lb]{\smash{\SB{$f_p$}}}}%
    \put(0.19693615,0.26851813){\color[rgb]{0,0,0}\makebox(0,0)[rb]{\smash{\SB{$\overline{f_p}$}}}}%
    \put(0.19693615,0.01391651){\color[rgb]{0,0,0}\makebox(0,0)[rb]{\smash{\SB{$f_p$}}}}%
    \put(0.32539873,0.13939778){\color[rgb]{0,0,0}\makebox(0,0)[lb]{\smash{\SB{$g_p$}}}}%
    \put(0.05177169,0.15097058){\color[rgb]{0,0,0}\makebox(0,0)[rb]{\smash{\SB{$\overline{g_p}$}}}}%
    \put(0.75226333,0.40794121){\color[rgb]{0,0,0}\makebox(0,0)[rb]{\smash{}}}%
    \put(0,0){\includegraphics[width=\unitlength,page=4]{drawing2.pdf}}%
    \put(0.06416847,0.21495472){\color[rgb]{0,0,0}\makebox(0,0)[lb]{\smash{\SBa{$\bar p$}}}}%
    \put(0.30711102,0.0645083){\color[rgb]{0,0,0}\makebox(0,0)[rb]{\smash{\SBa{$p$}}}}%
    \put(0.33880241,0.40011953){\color[rgb]{0,0,0}\makebox(0,0)[lb]{\smash{\SB{$\varepsilon$}}}}%
    \put(0.68653967,0.38688092){\color[rgb]{0,0,0}\makebox(0,0)[rb]{\smash{\SB{$\eta$}}}}%
  \end{picture}%
\endgroup%

\captionof{figure}{The surface $X_{[3,\FF_0]}$ and the blow-ups $\varepsilon$ and $\eta$.}\label{Fig:EqDP6}
\end{minipage}

\begin{proof}
Let $\varepsilon\colon X_{[3,\FF_0]}\to \PP^1\times\PP^1$ be the blow-up of two non real conjugate points in $\PP^1\times\PP^1$. We may assume that the points are $p=([1:{\bf i}],[1:{\bf i}])$ and its conjugate. With this choice of the points, there is a birational morphism $\eta\colon X\to\PP^1\times\PP^1$ which corresponds to the contraction of the fibres $g_p$ and $\overline{g_p}$ (see Figure~\ref{Fig:EqDP6}). This yields an injection $\varepsilon\times\eta\colon X\rightarrow(\PP^1)^4$. The fibration given by $g$ (meaning the fibres linearly equivalent to $g$, drawn punctuated in Figure~\ref{Fig:EqDP6}) is preserved by the birational map $\eta\varepsilon^{-1}$ from $\PP^1\times\PP^1$ to itself. Composing $\eta$ with an automorphism of the first factor, we obtain that $\eta\varepsilon^{-1}$ is the identity map on the first factor. Furthermore, we calculate that -- up to isomorphism of the second factor -- the map $\eta\varepsilon^{-1}$ is given by
\[\eta\varepsilon^{-1}\colon([x_0:x_1],[y_0:y_1])\dashmapsto([x_0:x_1],[x_0y_0+x_1y_1:x_0y_1-x_1y_0]).\]
The projection of $(\PP^1)^4$ dropping the third factor thus yields an injection $\varphi\colon X\stackrel{\varepsilon\times\eta}\longrightarrow(\PP^1)^4\rightarrow\PP^1\times\PP^1\times\PP^1$ and we get $(\ref{X_3,T 1})$ after the isomorphism $y_1\mapsto-y_1$.

For the second part, the map $\rho$ stands for the induced map coming from the action of $\Aut(X)$ on $\Pic(X)$. As the action of any automorphism of $X$ preserves the hexagon in Figure~\ref{fig:X_3,T}, the image of $\rho$ is contained in the dihedral group $D_6$. 
The image contains the reflections $\rho(\alpha_1)$ and $\rho(\alpha_0)$, where 
\[\alpha_1\colon([x_0:x_1],[y_0:y_1],[z_0:z_1])\mapsto([y_0:y_1],[x_0:x_1],[z_0:z_1]),\]
which is the lift of an automorphism exchanging the fibrations of $\FF_0$ and whose image by $\rho$ exchanges $f_p$ and $g_p$ and fixes $E_p$, and 
\[\alpha_0 \colon([x_0:x_1],[y_0:y_1],[z_0:z_1])\mapsto([x_0:-x_1],[z_1:z_0],[y_1:y_0]),\]
which is the lift of an automorphism of $\FF_0$ (the one on the right side in Figure~\ref{Fig:EqDP6}) exchanging $g_p$ and $\overline{g_p}$ and whose image by $\rho$ exchanges $g_p,\overline{g_p}$ and $E_p,f_p$. Their composition
\[\alpha_2:=\alpha_1\alpha_0\colon([x_0:x_1],[y_0:y_1],[z_0:z_1])\mapsto([z_1:z_0],[x_0:-x_1],[y_1:y_0])\]
has order $6$. The image $\rho(\alpha_2)$ is the composition of the two reflections $\rho(\alpha_1)$ and $\rho(\alpha_0)$ and hence is a rotation of order $6$. The elements $\rho(\alpha_1)$ and $\rho(\alpha_2)$ generate $D_6$. This yields the exact sequence (\ref{X_3, T 2}). Moreover, the relations $\alpha_1^2=\alpha_2^6=(\alpha_2\alpha_1)^2=\mathrm{Id}$ imply that $\rho(\alpha_i)\mapsto\alpha_i$, $i=1,2$ is a section of $\rho\colon\Aut_\R(X_{[3,S]})\rightarrow D_6$ and the sequence splits.\par
Last but not least, $\ker(\rho)$ is the group of automorphisms of $\FF_0$ fixing the points $p$ and $\bar p$ and preserving the fibrations of $\FF_0$, and which is isomorphic to $\mathrm{SO}_2(\R)\times\mathrm{SO}_2(\R)$. Its only finite orbits are its fixed points $p,\bar{p}$ and the intersection points of $f_p$ with $\overline{g_p}$ and of $\overline{f_p}$ with $g_p$. The group $D_6$ acts transitively on the intersection points of the $(-1)$-curves of $X_{[3,\FF_0]}$, and this yields (\ref{X_3,T 5}).\par
Finally, the description of $\Aut_\R(X_{[3,\FF_0]})$ in (\ref{X_3, T 2})--(\ref{X_3, T 3}) implies that we cannot contract any $(-1)$ curves on $X_{[3,\FF_0]}$ $\Aut_\R(X_{[3,\FF_0]})$-equivariantly. In particular, the pair $(X_{[3,\FF_0]},\Aut_\R(X_{[3,\FF_0]}))$ is a minimal pair.
\end{proof}


Blowing up two real points $p,q$ on $\FF_0$, we obtain a del Pezzo surface $X_{[4]}$ with $\rk(\Pic(X_{[4]})=4$. The lift of the antiholomorphic involution is indicated on Figure~\ref{fig:X_4} by arrows. Blowing down three of the six real $(-1)$-curves on $X_{[4]}$ yields a birational morphism $X_{[4]}\rightarrow\PP^2$ which is the blow-up of three real non-collinear points $r_1,r_2,r_3$ on $\PP^2$.

\begin{Prop}\label{prop:DP64}\item
\begin{enumerate}
\item\label{X_4 1} The surface $X_{[4]}$ is isomorphic to 
\[\{([x_0:x_1:x_2],[y_0:y_1:y_2]\in\PP^2\times\PP^2\mid x_0y_0=x_1y_1=x_2y_2\}.\] 
\item\label{X_4 2} There is a split exact sequence
\[1\rightarrow\ker(\rho)\rightarrow\Aut_\R(X_{[4]})\stackrel{\rho}\rightarrow D_6\rightarrow 1\]
where $\ker(\rho)\simeq(\R^*)^2$ is the diagonal subgroup of $\mathrm{PGL}_3(\R)$.
\item The group $\rho(\Aut_\R(X_{[4]}))=D_6$ is generated by the rotation $\rho(\alpha_1)$ and the reflection $\rho(\alpha_2)$, where
\[\alpha_1\colon([x_0:x_1:x_2],[y_0:y_1:y_2])\mapsto([y_2:y_0:y_1],[x_2:x_0:x_1])\]
and
\[\alpha_2\colon([x_0:x_1:x_2],[y_0:y_1:y_2])\mapsto([x_1:x_0:x_2],[y_1:y_0:y_2]).\]
\item The pair $(X_{[4]},\Aut_\R(X_{[4]}))$ is a minimal pair. 
\item\label{X_4 5} There is only one finite $\Aut_\R(X_{[4]})$-orbit on $X_{[4]}$, namely the one of the six intersection points of the $(-1)$-curves.
\end{enumerate}
\end{Prop}
\begin{proof}
The surface $X_{[4]}$ is the blow-up of three non-collinear real points in $\PP^2$ and hence isomorphic to $\{([x_0:x_1:x_2],[y_0:y_1:y_2])\in\PP^2\times\PP^2\mid x_0y_0=x_1y_1=x_2y_2\}$, the blow-up of the real points $[1:0:0],[0:1:0],[0:0:1]$. \par
The kernel of $\rho$ is isomorphic to the subgroup of $\Aut_\R(\PP^2)=\mathrm{PGL}_3(\R)$ fixing each of the three points $[1:0:0],[0:1:0],[0:0:1]$, which is the diagonal subgroup.  \par
The image by $\rho$ of $\Aut_\R(X_{[4]})$ is contained in $D_6$ because any automorphism of $X_{[4]}$ preserves the hexagon in Figure~\ref{fig:X_4}. The involution
\[\beta_1\colon([x_0:x_1:x_2],[y_0:y_1:y_2])\mapsto([y_0:y_1:y_2],[x_0:x_1:x_2])\]
is the lift of the standard Cremona involution on $\PP^2$ and hence $\rho(\beta_1)$ is a rotation of order $2$. The automorphism
\[\beta_2\colon([x_0:x_1:x_2],[y_0:y_1:y_2])\mapsto([x_2:x_0:x_1],[y_2:y_0:y_1])\]
of order $3$ is the lift of the automorphism of $\PP^2$ that permutes the three points $[1:0:0],[0:1:0],[0:0:1]$ and so $\rho(\beta_2)$ is a rotation of order $3$. Their composition
\[\alpha_1:=\beta_2\beta_1=\beta_1\beta_2\colon([x_0:x_1:x_2],[y_0:y_1:y_2])\mapsto([y_2:y_0:y_1],[x_2:x_0:x_1])\]
is of order $6$ and $\rho(\alpha_1)$ is a rotation of order $6$ by construction. Furthermore, we find that the involution
\[\alpha_2\colon([x_0:x_1:x_2],[y_0:y_1:y_2])\mapsto([x_1:x_0:x_2],[y_1:y_0:y_2])\]
is the lift of the automorphism of $\PP^2$  that exchanges $[1:0:0]$ and $[0:1:0]$ and fixes $[0:0:1]$, hence $\rho(\alpha_2)$ acts as a reflection. It follows that $\rho(\alpha_1)$ and $\rho(\alpha_2)$ generated $D_6$, and we get the exact sequence in (\ref{X_4 2}). Furthermore, the relations $\alpha_1^6=\alpha_2^2=(\alpha_1\alpha_2)^2=\mathrm{Id}$ imply that $\rho(\alpha_i)\mapsto\alpha_i$ is a section of $\rho\colon\Aut_\R(X_{[4]})\rightarrow D_6$ and the sequence splits.\par
The above description of $\Aut_\R(X_{[4]})$ yields that we cannot contract any $(-1)$-curve on $X_{[4]}$ $\Aut_\R(X_{[4]})$-equivariantly. In particular, $(X_{[4]},\Aut_\R(X_{[4]}))$ is a minimal pair.\par
The fact that $D_6$ acts transitively on the intersection points of the six $(-1)$-curves and that the only fixed points of $\ker(\rho)$ are the points $[1:0:0],[0:1:0],[0:0:1]$ implies (\ref{X_4 5}).
\end{proof}


\section{Pairs of real conic bundles}\label{sec:CB}
Recall that for a real conic bundle $\pi\colon X\rightarrow\PP^1$ we denote by $\Aut_\R(X,\pi)\subset\Aut_\R(X)$ the subgroup of automorphisms respecting the conic bundle structure on $X$. The morphism $\pi$ induces a homomorphism $\alpha\colon\Aut_\R(X,\pi)\rightarrow\Aut_\R(\PP^1)=\mathrm{PGL}_2(\R)$ whose kernel we denote by $\ker(\alpha)=\Aut_\R(X/\pi)$. We get an exact sequence
\begin{equation}\tag{$\ast$}\label{stern}1\rightarrow\Aut_\R(X/\pi)\rightarrow\Aut_\R(X,\pi)\stackrel{\alpha}\rightarrow\Aut_\R(\PP^1).\end{equation}
Proposition~\ref{prop which cases} and Lemma~\ref{lem aut} imply that a minimal pair $(X,G)$ of a $\R$-rational $G$-variety is either a del Pezzo surface of degree $9,8$ or $6$ and $G=\Aut_\R(X)$, or it admits a real conic bundle $\pi\colon X\rightarrow\PP^1$ with a real birational morphism of conic bundles $X\rightarrow Y$, where $Y$ is the sphere blown up in a pair of non-real conjugate points or a Hirzebruch surface, and $G=\Aut_\R(X,\pi)$. In this section, we aim at classifying the real conic bundles $\pi\colon X\rightarrow\PP^1$ that are relatively $\Aut_\R(X,\pi)$-minimal.

The following lemma is an adaption of \cite[Lemma 4.3.5]{B10} to our purpose.
\begin{Lem}\label{lem:blowup CB}
Suppose $\pi\colon X\rightarrow\PP^1$ is a relatively $\Aut_\R(X,\pi)$-minimal real conic bundle with a morphism $\eta\colon X\rightarrow Y$ of real conic bundles where $Y$ is as above, and $\eta$ is not an isomorphism. Let 
\[G:=\Aut_\R(X/\pi)\cap\left(\ \ker(\Aut_\R(X,\pi)\rightarrow\Aut(\Pic(X))\ \right).\] \par
If $G$ is non-trivial, there exists $n\geq1$ and a $($perhaps non-real$)$ birational morphism $X\rightarrow\FF_n$ of conic bundles defined over $\C$ that blows up $2n$ points in a section $s$ with $s^2=n$ which is disjoint from the exceptional section of $\FF_n$, and the strict transform of these two sections are exchanged by an element of $\Aut_\R(X,\pi)$. \par
If $G$ is trivial, then $\Aut_\R(X/\pi)\simeq(\Z/2\Z)^r$ for $r\in\{0,1,2\}$.
\end{Lem}
\begin{proof}
By assumption, $\eta$ blows up at least one point on $Y$, hence $X$ has at least one singular fibre. All its singular fibres have exactly two irreducible components because $\pi\colon X\rightarrow\PP^1$ is relatively $\Aut_\R(X,\pi))$-minimal. Note that $G$ is a normal subgroup of $\Aut(X,\pi)$. Forgetting the antiholomorphic involution on $X$, we contract in each singular fibre one component and obtain a (perhaps non-real) $G$-equivariant morphism $\eta'\colon X\rightarrow\FF_n$ for some $n\geq0$. By changing the choice of the components we contract, we obtain $n\geq1$ and further that $\eta'$ does not blow-up any points on the exceptional section $E_n$ of $\FF_n$. Let $R:=\eta' G(\eta')^{-1}\subset\Aut_\R(\FF_n)$.\par
Suppose that $G$ is not trivial. The group $R$ fixes the points blown-up by $\eta'$ and it preserves $E_n$. Hence $G$ preserves the strict transform $\tilde{E}_n$ of $E_n$ in $X$. By construction of $\eta'$, the curve $\tilde{E}_n$ intersects exactly one component of each singular fibre of $X$. The morphism $\pi\colon X\rightarrow\PP^1$ is relatively $\Aut_\R(X,\pi))$-minimal, so there exists $h\in\Aut_\R(X,\pi)$ exchanging the components of singular fibres of $X$. As $G$ is normal in $\Aut_\R(X,\pi)$, the we have $hGh^{-1}=G$. Thus
$\mathrm{Fix}(G)=\mathrm{Fix}(hGh^{-1})=h(\mathrm{Fix}(G))$ contains the section $h(\tilde{E}_n)$, and $h(\tilde{E}_n)\neq\tilde{E}_n$. In particular, $R$ preserves the section $s:=\eta'(h(\tilde{E}_n))\neq E_n$. As $R\subset\mathrm{PGL}_2(\C(t))$ is a non-trivial subgroup, it fixes at most two points on each fibre of $\FF_n$, thus $E_n\cup s=\mathrm{Fix}(R)$ and $s$ contains all points blown up by $\eta'$. Further, the two curves $\tilde{E}_n$ and $h(\tilde{E}_n)$ have the same self-intersection, which is equal to $-n$ because $\eta'$ does not blow up any points on $E_n$. Moreover, since the elements of $R\subset\mathrm{PGL}_2(\C(t))$ fix exactly two points on all but finitely many fibres, they are diagonalisable and hence fix two points on {\em every} fibre. It follows that $E_n$ and $s$ are disjoint. Hence $s\sim E_n+nf$ and therefore $s^2=n$. It follows that $\eta'$ is the blow-up of $2n$ points on $s$.\par
Suppose that $G$ is trivial. Then every non-trivial element of $\Aut_\R(X/\pi)$ is an involution. As $\Aut_\R(X/\pi)\subset\mathrm{PGL}_2(\C(x))$, it follows that $\Aut_\R(X/\pi)$ is isomorphic to $(\Z/2\Z)^r$ for $r\in\{0,1,2\}$.
\end{proof}

\subsection{Real conic bundles obtained by blowing up a del Pezzo surface}\label{ssec:CBDP}
In this subsection, we study the ones with a real birational morphism of conic bundles $\eta\colon X\rightarrow Y$ to the surface $Y$ obtained by blowing up the sphere in a pair of non-real conjugate points. The surface $Y$ is a del Pezzo surface of degree $6$ with $\rk(\Pic(Y))=2$ and so, by Lemma~\ref{lem:dP class},
\[\pr_1\colon Y=X_{[2]}\rightarrow \QQ_{3,1}\]
is isomorphic to the blow-up of the points $r:=[0:1:{\bf i}:0]$ and $\bar{r}=[0:1:-{\bf i}:0]$ (see Proposition~\ref{prop:X_2}). The generic fibre of the projection
\[\pr\colon X_{[2]}\rightarrow\PP^1,\quad ([w:x:y:z],[u:v])\mapsto[u:v]\]
is the non-rational conic $C\subset\PP^2_{\R(t)}$ given by $x^2+y^2-tz^2=0$, which makes $\pr\colon X_{[2]}\rightarrow\PP^1$ a conic bundle. However, the description in Lemma~\ref{desc:X_2} of the conic bundle $(X_{[2]},\pr)$ will turn out to be more convenient. \par
As described in Remark~\ref{rmk:real part}, the isomorphism of complex surfaces $\varphi\colon \QQ_{3,1}\longrightarrow\PP^1\times\PP^1$, 
\begin{align*}
\varphi\colon& [w:x:y:z]&\longmapsto&([w+z:y+{\bf i}x],[w+z:y-{\bf i}x])=([y-{\bf i}x:w-z],[y+{\bf i}x:w-z])\\
\varphi^{-1}\colon& ([x_0:x_1],[y_0:y_1])&\mapsto&[x_0y_0+x_1y_1:{\bf i}(x_0y_1-x_1y_0):x_0y_1+x_1y_0:x_0y_0-x_1y_1]
\end{align*}
induces an isomorphism of real surfaces $\varphi\colon \QQ_{3,1}\rightarrow(\PP^1\times\PP^1,\sigma_S)$, where
\[\sigma_S\colon([x_0:x_1],[y_0:y_1])\mapsto([\bar{y}_0:\bar{y}_1],[\bar{x}_0:\bar{x}_1])\]
and $p:=\varphi(r)=([0:1],[1:0])$ and $\bar{p}=\varphi(\bar{r})=([1:0],[0:1])$.\par

\begin{Lem}\label{desc:X_2}\item
\begin{enumerate}
\item The real surface $X_{[2]}$ is isomorphic to
\[(X_{[2]},\ \sigma)\simeq\left(\{([x_0:x_1:x_2],[y_0:y_1:y_2])\in\PP^2\times\PP^2\mid x_0y_0=x_1y_1=x_2y_2\},\ \sigma_{[2]}\right)\]
where $\sigma_{[2]}\colon ([x_0:x_1:x_2],[y_0:y_1:y_2])\mapsto([\overline{y_1}:\overline{y_0}:\overline{y_2}],[\overline{x_1}:\overline{x_0}:\overline{x_2}])$ 
and the conic bundle structure $\pi_{[2]}\colon X_{[2]}\rightarrow\PP^1$ is given by
\[\pi_{[2]}\colon([x_0:x_1:x_2],[y_0:y_1:y_2])\longmapsto[x_0:x_1]=[y_1:y_0].\]
\item The irreducible components of the singular fibres of $\pi_{[2]}\colon X_{[2]}\rightarrow\PP^1$ are given by 
\[f_p\colon y_1=y_2=0,\quad\overline{f_p}\colon x_0=x_2=0,\quad f_{\bar p}\colon x_1=x_2=0,\quad\overline{f_{\bar p}}\colon y_0=y_2=0\]
$($see Figure~$\ref{fig:CBX_2})$ and the pair of non-real conjugate $(-1)$-``sections" by
\[s\colon x_0=x_1=0,\quad \bar{s}\colon y_0=y_1=0.\] 
\end{enumerate}
\end{Lem}
\begin{proof}
Over $\C$, there is only one del Pezzo surface of degree $6$ and it is isomorphic to 
\[Z:=\{([x_0:x_1:x_2],[y_0:y_1:y_2])\in\PP^2\times\PP^2\mid x_0y_0=x_1y_1=x_2y_2\}\] 
(see Proposition~\ref{prop:DP64}). The abstract birational morphism
\begin{align*}
Z&\longrightarrow \PP^1\times\PP^1\\
([x_0:x_1:x_2],[y_0:y_1:y_2])&\longmapsto([x_0:x_2],[x_2:x_1])=([y_2:y_0],[y_1:y_2])\\
([ru:sv:su],[sv:ru:rv])&\dashmapsfrom([r:s],[u:v])
\end{align*}
contracts the $(-1)$-curves $s_1=\{x_0=x_1=0\}$ and $s_2=\{y_0=y_1=0\}$ onto $p$ and $\bar{p}$, respectively. In Figure~\ref{fig:X_2}, they are therefore denoted by $E_p$ and $E_{\bar{p}}$.
The lift of the antiholomorphic involution $\sigma_S$ onto $Z$ is $\sigma_{[2]}$ 
and makes $(Z,\sigma_{[2]})$ a real del Pezzo surface isomorphic to $X_{[2]}$ (Lemma~\ref{lem:dP class}) and $s_2=\overline{s_1}$. 
The morphism
\[\pi_{[2]}\colon X_{[2]}\rightarrow\PP^1,\quad ([x_0:x_1:x_2],[y_0:y_1:y_2])\longmapsto[x_0:x_1]=[y_1:y_0]\]
is the projection onto the $(-1)$-curves $s$ and $\bar{s}$ and is thus a conic bundle. The antiholomorphic involution $\sigma_{[2]}$ descends to the standard antiholomorphic involution $[u:v]\mapsto[\bar{u}:\bar{v}]$ on $\PP^1$, which makes $\pi_{[2]}\colon X_{[2]}\rightarrow\PP^1$ a real conic bundle. The morphisms are visualised in Figure~\ref{fig:CBX_2}. The equations of the irreducible components of the singular fibres are checked by calculation.\par
By abuse of notation, we will denote the surface $Z$ by $X_{[2]}$ endowed with $\sigma_{[2]}$.
\end{proof}
\begin{minipage}{0.5\textwidth}
\centering
\def\svgwidth{1\columnwidth}
\begingroup%
  \makeatletter%
  \providecommand\color[2][]{%
    \errmessage{(Inkscape) Color is used for the text in Inkscape, but the package 'color.sty' is not loaded}%
    \renewcommand\color[2][]{}%
  }%
  \providecommand\transparent[1]{%
    \errmessage{(Inkscape) Transparency is used (non-zero) for the text in Inkscape, but the package 'transparent.sty' is not loaded}%
    \renewcommand\transparent[1]{}%
  }%
  \providecommand\rotatebox[2]{#2}%
  \ifx\svgwidth\undefined%
    \setlength{\unitlength}{459.97662402bp}%
    \ifx\svgscale\undefined%
      \relax%
    \else%
      \setlength{\unitlength}{\unitlength * \real{\svgscale}}%
    \fi%
  \else%
    \setlength{\unitlength}{\svgwidth}%
  \fi%
  \global\let\svgwidth\undefined%
  \global\let\svgscale\undefined%
  \makeatother%
  \begin{picture}(1,0.90167611)%
    \put(0,0){\includegraphics[width=\unitlength,page=1]{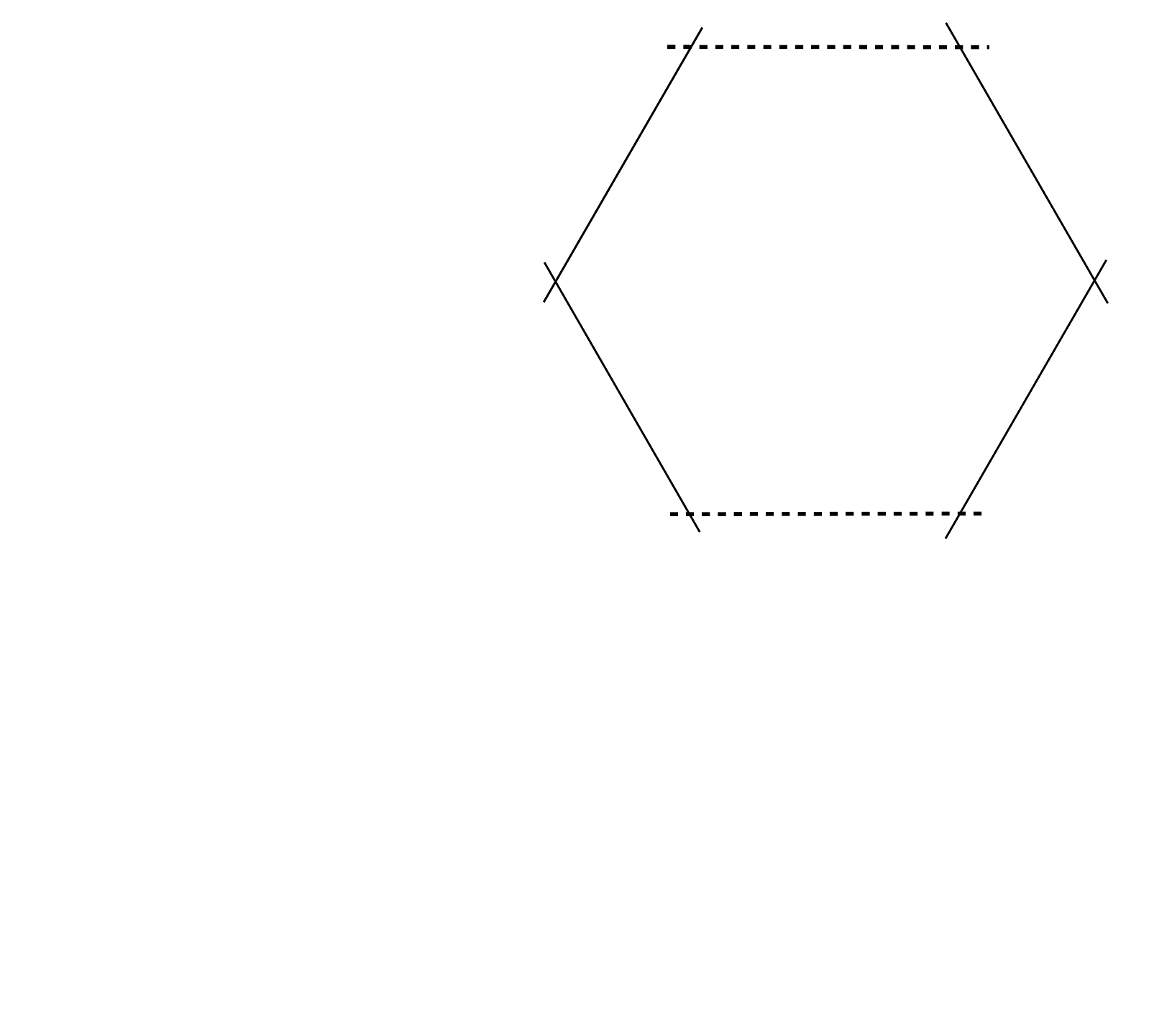}}%
    \put(0.3130771,0.64793616){\color[rgb]{0,0,0}\makebox(0,0)[lb]{\smash{\SB{$(X_{[2]},\sigma_Z)$}}}}%
    \put(0,0){\includegraphics[width=\unitlength,page=2]{fig06.pdf}}%
    \put(0.03480207,0.38705339){\color[rgb]{0,0,0}\makebox(0,0)[lb]{\smash{\SB{$(\PP^1\times\PP^1,\sigma_S)$}}}}%
    \put(0.69939536,0.8772017){\color[rgb]{0,0,0}\makebox(0,0)[lb]{\smash{\SB{$\bar{s}$}}}}%
    \put(0.69570526,0.4218377){\color[rgb]{0,0,0}\makebox(0,0)[lb]{\smash{\SB{$s$}}}}%
    \put(0.49490448,0.76968148){\color[rgb]{0,0,0}\makebox(0,0)[lb]{\smash{\SB{$\overline{f_p}$}}}}%
    \put(0.48805302,0.56492808){\color[rgb]{0,0,0}\makebox(0,0)[lb]{\smash{\SB{$f_p$}}}}%
    \put(0.8870193,0.78707367){\color[rgb]{0,0,0}\makebox(0,0)[lb]{\smash{\SB{$f_{\bar p}$}}}}%
    \put(0.92180367,0.54358306){\color[rgb]{0,0,0}\makebox(0,0)[lb]{\smash{\SB{$\overline{f_{\bar p}}$}}}}%
    \put(0.34364517,0.47770352){\color[rgb]{0,0,0}\makebox(0,0)[lb]{\smash{\SB{$p,\bar{p}$}}}}%
    \put(0,0){\includegraphics[width=\unitlength,page=3]{fig06.pdf}}%
    \put(0.02241675,0.0628082){\color[rgb]{0,0,0}\makebox(0,0)[lb]{\smash{\SB{$p$}}}}%
    \put(0,0){\includegraphics[width=\unitlength,page=4]{fig06.pdf}}%
    \put(0.92180364,0.07399395){\color[rgb]{0,0,0}\makebox(0,0)[lb]{\smash{\SB{$\PP^1$}}}}%
    \put(0.31307709,0.35226905){\color[rgb]{0,0,0}\makebox(0,0)[lb]{\smash{\SB{$\bar{p}$}}}}%
    \put(0.71594331,0.25726366){\color[rgb]{0,0,0}\makebox(0,0)[lb]{\smash{\SB{$\pi_{[2]}$}}}}%
    \put(0,0){\includegraphics[width=\unitlength,page=5]{fig06.pdf}}%
    \put(0.26090053,0.00442524){\color[rgb]{0,0,0}\makebox(0,0)[lb]{\smash{\SB{$f_{\bar p}$}}}}%
    \put(0.03480208,0.00442525){\color[rgb]{0,0,0}\makebox(0,0)[lb]{\smash{\SB{$f_p$}}}}%
    \put(0.38264582,0.10877838){\color[rgb]{0,0,0}\makebox(0,0)[lb]{\smash{\SB{$\overline{f_{\bar p}}$}}}}%
    \put(0.36525367,0.31748463){\color[rgb]{0,0,0}\makebox(0,0)[lb]{\smash{\SB{$\overline{f_p}$}}}}%
  \end{picture}%
\endgroup%
\\
\end{minipage}
\begin{minipage}{.45\textwidth}
{\scriptsize
\[\xymatrix{&([x_0:x_1:x_2],[y_0:y_1:y_2])\ar[ddddl]_{p,\bar{p}}\ar[dddd]^{\pi_{[2]}}\\ &\\&\\ &\\([x_0:x_2],[x_2:x_1])& [x_0:x_1]\\([y_2:y_0],[y_1:y_2])\ar@{=}[u]&[y_1:y_0]\ar@{=}[u]}\quad\]
}
\end{minipage}\vspace{-1em}
\captionof{figure}{The real conic bundle $\pi_{[2]}\colon X_{[2]}\rightarrow\PP^1$.}\label{fig:CBX_2}
\vskip\baselineskip

The following lemma gives a necessary condition for a real conic bundle $\pi\colon X\rightarrow\PP^1$ to be relatively $\Aut_\R(X,\pi))$-minimal. It will turn out in Proposition~\ref{prop:CBX_2} that the condition is also sufficient. A reference picture is drawn in Figure~\ref{fig:CBX_2 C}. 

\begin{Lem}\label{lem:CBDP}
Let $\pi\colon X\rightarrow\PP^1$ be a relatively $\Aut_\R(X,\pi)$-minimal real conic bundle with a birational morphism of real conic bundles $\eta\colon X\rightarrow X_{[2]}$ that is not an isomorphism. Then $\Aut_\R(X,\pi)$ is finite, or $\eta\colon X\rightarrow X_{[2]}$ is the blow-up of $n\geq1$ pairs of non-real conjugate points of $X_{[2]}$ contained in $s\cup\bar{s}$ and in non-real fibres.
\end{Lem}

\begin{proof}  
The morphism $\eta$ blows up at least one point, so $X$ has at least four singular fibres. Let $G:=\Aut_\R(X/\pi)\cap\left(\ \Aut_\R(X,\pi)\rightarrow\Aut(\Pic(X))\ \right)$.\par
Suppose that $G$ is non-trivial. By Lemma~\ref{lem:blowup CB}, there exists a non-real birational morphism $\eta'\colon X\rightarrow\FF_n$, $n\geq1$, of conic bundles defined over $\C$ which blows up $2n\geq2$ points on a section $s'$ disjoint from the exceptional section $E_n$ of $\FF_n$ with $(s')^2=n$. \par
Denote by $\tilde{s}'$ and $\tilde{E}_n$ the strict transforms of $s'$ and $E_n$ respectively. Note that they are the unique $(-n)$-curves on $X$ and hence are two real or a pair of non-real conjugate curves. They descend via $\eta\colon X\rightarrow X_{[2]}$ onto curves $c_1$ and $c_2$ on $X_{[2]}$. As $X_{[2]}$ does not have any real sections, we get $c_2=\bar{c}_1$, and hence $\tilde{s}'=\overline{\tilde{E}_n}$. In particular, the real morphism $\eta$ contracts $n-1$ components of singular fibres only intersecting $\tilde{s}'$ and $n-1$ components of singular fibres only intersecting $\tilde{E}_n$. In other words, $\eta$ blows up $n\geq1$ pairs of non-real conjugate points contained in $c_1\cup \bar{c}_2$, no two on the same fibre, so $c_1^2=-1$. In particular, $c_1\cup \bar{c}_1=s\cup\bar{s}$.\par
Suppose that $G$ is trivial. Lemma~\ref{lem:blowup CB} implies that $\Aut_\R(X/\pi)$ is isomorphic to $(\Z/2\Z)^r$ for $r\in\{0,1,2\}$. The group $\Aut_\R(X,\pi)$ preserves $X(\R)$, hence its image $H$ in $\mathrm{PGL}_2(\R)$ preserves $\pi(X(\R))=\pi_{[2]}(X_{[2]}(\R))=[0,\infty]$, i.e. 
\[H\subset\Aut_\R(\PP^1,[0,\infty])\simeq\R_{>0}\ltimes\Z/2\Z\]
Furthermore, $H$ preserves the set of the images in $\PP^1$ of the singular fibres of $X$, of which there are at least four. This implies that $H$ is finite. As $\Aut_\R(X/\pi)$ and $H$ are both finite, also $\Aut_\R(X,\pi)$ is finite.
\end{proof}

\begin{Def}[and construction]\label{def:sigmaC}
The abstract birational morphism
\begin{align*}
\varepsilon\colon X_{[2]}&\longrightarrow\PP^1\times\PP^1\\
([x_0:x_1:x_2],[y_0:y_1:y_2])&\ \longmapsto\ ([x_0:x_1],[x_2:x_0])=([y_1:y_0],[y_0:y_2])\\
([u_0v_1:u_1v_1:u_0v_0],[u_1v_0:u_0v_0:u_1v_1])&\ \dashmapsfrom\ ([u_0:u_1],[v_0,v_1])
\end{align*}
contracts the components $f_p$ and $f_{\bar p}$ of the singular fibres onto the points $([0:1],[1:0])$ and $([1:0],[0:1])$, respectively, and the sections $s,\bar{s}$ onto the sections $\varepsilon(s)=\PP^1\times\{[1:0]\}$ and $\varepsilon(\bar{s})=\PP^1\times\{[0:1]\}$. The antiholomorphic involution $\sigma_{[2]}$ descends to a rational antiholomorphic involution
\[\sigma_C\colon ([u_0:u_1],[v_0:v_1])\dashmapsto([\overline{u_0}:\overline{u_1}],[\overline{u_1v_1}:\overline{u_0v_0}])\]
on $\PP^1\times\PP^1$, not defined at $([1:0],[0:1])$ and $([0:1],[1:0])$. It makes $\varepsilon$ a real birational morphism of conic bundles, i.e. the diagram
\[\xymatrix{(X_{[2]},\sigma_{[2]})\ar[rd]_{\pi_{[2]}}\ar[rr]^{\varepsilon}&&(\PP^1\times\PP^1,\sigma_C)\ar[ld]^{\pr_1}\\ &\PP^1&}\]
is commutative. The construction is visualised in Figure~\ref{fig:CBX_2 C}.
\end{Def}

\begin{minipage}{0.98\textwidth}
\centering
\def\svgwidth{0.6\columnwidth}
\begingroup%
  \makeatletter%
  \providecommand\color[2][]{%
    \errmessage{(Inkscape) Color is used for the text in Inkscape, but the package 'color.sty' is not loaded}%
    \renewcommand\color[2][]{}%
  }%
  \providecommand\transparent[1]{%
    \errmessage{(Inkscape) Transparency is used (non-zero) for the text in Inkscape, but the package 'transparent.sty' is not loaded}%
    \renewcommand\transparent[1]{}%
  }%
  \providecommand\rotatebox[2]{#2}%
  \ifx\svgwidth\undefined%
    \setlength{\unitlength}{682.14630367bp}%
    \ifx\svgscale\undefined%
      \relax%
    \else%
      \setlength{\unitlength}{\unitlength * \real{\svgscale}}%
    \fi%
  \else%
    \setlength{\unitlength}{\svgwidth}%
  \fi%
  \global\let\svgwidth\undefined%
  \global\let\svgscale\undefined%
  \makeatother%
  \begin{picture}(1,0.82159904)%
    \put(0,0){\includegraphics[width=\unitlength,page=1]{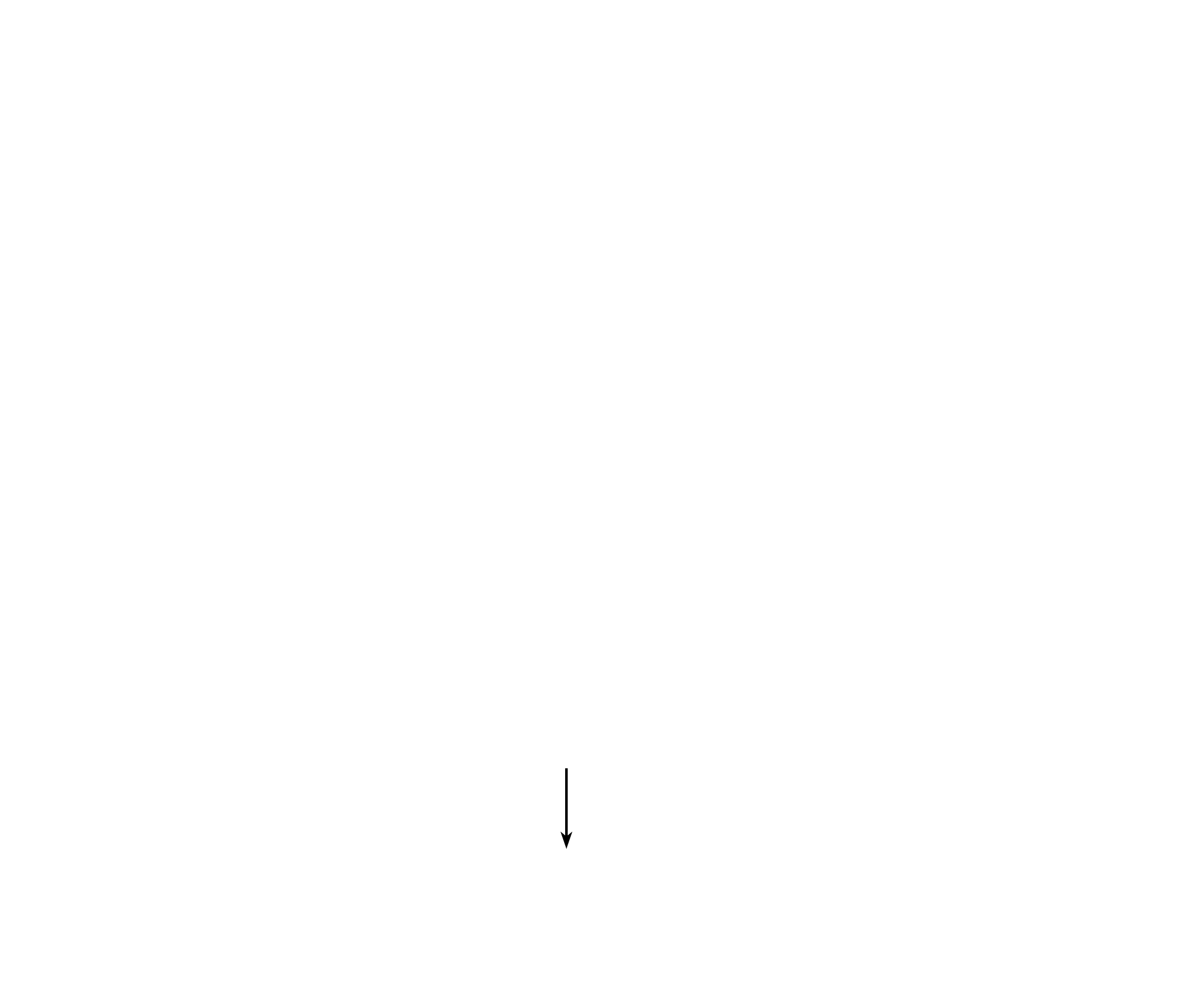}}%
    \put(0.59363361,0.253346){\color[rgb]{0,0,0}\makebox(0,0)[lb]{\smash{\SB{$\overline{f_{\bar p}}$}}}}%
    \put(0.00303435,0.27270655){\color[rgb]{0,0,0}\makebox(0,0)[lb]{\smash{\SB{$(\PP^1\times\PP^1,\sigma_S)$}}}}%
    \put(0,0){\includegraphics[width=\unitlength,page=2]{fig05.pdf}}%
    \put(0.00586549,0.0513693){\color[rgb]{0,0,0}\makebox(0,0)[lb]{\smash{\SB{$p$}}}}%
    \put(0.15823079,0.19622068){\color[rgb]{0,0,0}\makebox(0,0)[lb]{\smash{\SB{$\bar{p}$}}}}%
    \put(0.48544876,0.15043439){\color[rgb]{0,0,0}\makebox(0,0)[lb]{\smash{\SB{$\pi_{[2]}$}}}}%
    \put(0,0){\includegraphics[width=\unitlength,page=3]{fig05.pdf}}%
    \put(0.2029505,0.35065194){\color[rgb]{0,0,0}\makebox(0,0)[lb]{\smash{\SB{$(X_{[2]},\sigma_{[2]})$}}}}%
    \put(0.6540273,0.28763757){\color[rgb]{0,0,0}\makebox(0,0)[lb]{\smash{\SB{$\varepsilon$}}}}%
    \put(0,0){\includegraphics[width=\unitlength,page=4]{fig05.pdf}}%
    \put(0.45495669,0.1975271){\color[rgb]{0,0,0}\makebox(0,0)[lb]{\smash{\SB{$s$}}}}%
    \put(0.59515211,0.43844607){\color[rgb]{0,0,0}\makebox(0,0)[lb]{\smash{\SB{$f_{\bar p}$}}}}%
    \put(0,0){\includegraphics[width=\unitlength,page=5]{fig05.pdf}}%
    \put(0.46096341,0.50760031){\color[rgb]{0,0,0}\makebox(0,0)[lb]{\smash{\SB{$\bar{s}$}}}}%
    \put(0.32295298,0.43669825){\color[rgb]{0,0,0}\makebox(0,0)[lb]{\smash{\SB{$\overline{f_p}$}}}}%
    \put(0.33481843,0.25222234){\color[rgb]{0,0,0}\makebox(0,0)[lb]{\smash{\SB{$f_p$}}}}%
    \put(0,0){\includegraphics[width=\unitlength,page=6]{fig05.pdf}}%
    \put(0.44193519,0.46648648){\color[rgb]{0,0,0}\makebox(0,0)[lb]{\smash{\SB{$\bar{q}$}}}}%
    \put(0.51863295,0.24932366){\color[rgb]{0,0,0}\makebox(0,0)[lb]{\smash{\SB{$q$}}}}%
    \put(0.22971464,0.28689106){\color[rgb]{0,0,0}\makebox(0,0)[lb]{\smash{\SB{$p,\bar{p}$}}}}%
    \put(0,0){\includegraphics[width=\unitlength,page=7]{fig05.pdf}}%
    \put(0.52103045,0.39873814){\color[rgb]{0,0,0}\makebox(0,0)[lb]{\smash{\SB{$F_q$}}}}%
    \put(0.44352578,0.40077773){\color[rgb]{0,0,0}\makebox(0,0)[lb]{\smash{\SB{$F_{\bar q}$}}}}%
    \put(0,0){\includegraphics[width=\unitlength,page=8]{fig05.pdf}}%
    \put(0.19581564,0.80509575){\color[rgb]{0,0,0}\makebox(0,0)[lb]{\smash{\SB{$\bar{s}'$}}}}%
    \put(0.10267637,0.7362333){\color[rgb]{0,0,0}\makebox(0,0)[lb]{\smash{\SB{$\overline{f_p}$}}}}%
    \put(0.09373465,0.5914802){\color[rgb]{0,0,0}\makebox(0,0)[lb]{\smash{\SB{$f_p$}}}}%
    \put(0.19796731,0.49743297){\color[rgb]{0,0,0}\makebox(0,0)[lb]{\smash{\SB{$s'$}}}}%
    \put(0.40369601,0.60233116){\color[rgb]{0,0,0}\makebox(0,0)[lb]{\smash{\SB{$\overline{f_{\bar p}}$}}}}%
    \put(0.37487549,0.74336604){\color[rgb]{0,0,0}\makebox(0,0)[lb]{\smash{\SB{$f_{\bar p}$}}}}%
    \put(0,0){\includegraphics[width=\unitlength,page=9]{fig05.pdf}}%
    \put(0.28749648,0.70977691){\color[rgb]{0,0,0}\makebox(0,0)[lb]{\smash{\SB{$F_q$}}}}%
    \put(0.19877402,0.60371777){\color[rgb]{0,0,0}\makebox(0,0)[lb]{\smash{\SB{$F_{\bar q}$}}}}%
    \put(0.20285319,0.70977689){\color[rgb]{0,0,0}\makebox(0,0)[lb]{\smash{\SB{$E_{\bar q}$}}}}%
    \put(0.27933813,0.60677718){\color[rgb]{0,0,0}\makebox(0,0)[lb]{\smash{\SB{$E_q$}}}}%
    \put(0.41403845,0.56365045){\color[rgb]{0,0,0}\makebox(0,0)[lb]{\smash{\SB{$\eta$}}}}%
    \put(0.81573401,0.28608355){\color[rgb]{0,0,0}\makebox(0,0)[lb]{\smash{\SB{$(\PP^1\times\PP^1,\sigma_C)$}}}}%
    \put(0,0){\includegraphics[width=\unitlength,page=10]{fig05.pdf}}%
    \put(0.35185977,0.10089914){\color[rgb]{0,0,0}\makebox(0,0)[lb]{\smash{\SB{$\PP^1$}}}}%
    \put(0,0){\includegraphics[width=\unitlength,page=11]{fig05.pdf}}%
    \put(0.49107926,0.05400399){\color[rgb]{0,0,0}\makebox(0,0)[lb]{\smash{\SB{$\pi(q)$}}}}%
    \put(0.42115245,0.05474505){\color[rgb]{0,0,0}\makebox(0,0)[lb]{\smash{\SB{$\pi(\bar{q})$}}}}%
    \put(0.56858528,0.05578076){\color[rgb]{0,0,0}\makebox(0,0)[lb]{\smash{\SB{$([0:1],[1:0])$}}}}%
    \put(0.90070321,0.25209129){\color[rgb]{0,0,0}\makebox(0,0)[lb]{\smash{\SB{$([1:0],[0:1])$}}}}%
    \put(0,0){\includegraphics[width=\unitlength,page=12]{fig05.pdf}}%
    \put(0.93809674,0.227275){\color[rgb]{0,0,0}\makebox(0,0)[lb]{\smash{\SB{$\varepsilon(\bar{s})$}}}}%
    \put(0.94382788,0.08214107){\color[rgb]{0,0,0}\makebox(0,0)[lb]{\smash{\SB{$\varepsilon(s)$}}}}%
    \put(0.84652213,0.05935697){\color[rgb]{0,0,0}\makebox(0,0)[lb]{\smash{\SB{$\varepsilon(q)$}}}}%
    \put(0.74827711,0.24866249){\color[rgb]{0,0,0}\makebox(0,0)[lb]{\smash{\SB{$\varepsilon(\bar{q})$}}}}%
    \put(0.83052197,0.03730356){\color[rgb]{0,0,0}\makebox(0,0)[lb]{\smash{\SB{$F_q$}}}}%
    \put(0.77341319,0.03730356){\color[rgb]{0,0,0}\makebox(0,0)[lb]{\smash{\SB{$F_{\bar q}$}}}}%
  \end{picture}%
\endgroup%
\\
\captionof{figure}{The real birational morphism $\varepsilon\colon(X_{[2]},\sigma_{[2]})\rightarrow(\PP^1\times\PP^1,\sigma_C)$.}\label{fig:CBX_2 C}
\end{minipage}

\begin{Prop}\label{prop:CBX_2}
Let $\eta\colon X\rightarrow X_{[2]}$ be the blow-up of $n\geq1$ pairs of non-real conjugate points in $s\cup \bar{s}$ and in non-real fibres. Then $\pi:=\pi_{[2]}\eta\colon X\rightarrow\PP^1$ is a relatively $\Aut_\R(X,\pi))$-minimal real conic bundle. \par
Let $\Delta\subset\PP^1$ be the image of the $2n+2$ singular fibres of $X$ and $H_{\Delta}\subset\mathrm{PGL}_2(\R)$ be the subgroup preserving $\Delta$ and $\pi(X(\R))=[0,\infty]$. Then
\begin{enumerate}
\item\label{CBX_2 1} there exists a split exact sequence
\[1\rightarrow\Aut(X/\pi)\rightarrow\Aut(X,\pi)\rightarrow H_{\Delta}\rightarrow 1\]
where $\Aut(X/\pi)\simeq\Aut_\R(\QQ_{3,1},p,\bar{p})\rtimes\Z/2\Z\simeq\mathrm{SO}_2(\R)\rtimes\Z/2\Z$,
\item\label{CBX_2 2} an element of $\mathrm{SO}_2(\R)\subset\Aut(X/\pi)$ fixes the two $\left(-(n+1)\right)$-sections of $X$ and the generator of $\Z/2\Z$ exchanges them,
\item\label{CBX_2 3} an element of $\Aut(X/\pi)\setminus\mathrm{SO}_2(\R)$ is an involution fixing an irreducible curve on $X_{[2]}$ which is a double cover of $\PP^1$ ramified at $\Delta$,
\item\label{CBX_2 4} the group $\mathrm{SO}_2(\R)$ acts trivially on $\mathrm{Pic}(X)$. 
\end{enumerate}
\end{Prop}
\begin{proof}
Any automorphism of $X$ preserves the set of real points $X(\R)$, which is diffeomorphic via $\eta$ to $X_{[2]}(\R)$ and is mapped to the interval $[0,\infty]$ by $\pi$. Therefore, the exact sequence (\ref{stern}) yields the exact sequence (\ref{CBX_2 1}). Any element of $H_{\Delta}$ lifts to a real automorphism of $X_{[2]}$ fixing the points blown-up by $\eta$ and thus lifts to an automorphism of $X$. The sequence splits.\par
Over $\C$, there is a birational morphism $X\stackrel{\eta}\rightarrow X_{[2]}\rightarrow\FF_0$, hence $X$ has exactly two $(-(n+1))$-sections \cite[Lemma 4.3.1]{B10}, and they are the strict transforms $s',\bar{s}'$ of the $(-1)$-curves $s,\bar{s}$ on $X_{[2]}$. So $\Aut_\R(X/\pi)$ acts on $\{s',\bar{s}'\}$ and we claim that it acts non-trivially; we now construct a birational involution of $X_{[2]}$ whose lift onto $X$ is an automorphism respecting $\pi$ and exchanging $s',\bar{s}'$. \par
Let $q_1,\dots,q_n\in s$ and $\overline{q_1},\dots,\overline{q_n}\in \bar{s}$ be the points blown up by $\eta$, and define $p_i:=\pi_{[2]}(q_i)$. Let $\varepsilon\colon (X_{[2]},\sigma_{[2]})\rightarrow(\PP^1\times\PP^1,\sigma_C)$ be the birational morphism of real conic bundles given in Definition~\ref{def:sigmaC}. Then $\varepsilon(q_i)=(p_i,[1:0])$ and $\varepsilon(\overline{q_i})=(\overline{p_i},[0:1])$. Let $m_1,\dots,m_n\in\C[u_0,u_1]$ be homogenous linear polynomials vanishing on $p_1,\dots,p_n$ respectively and define $P(u_0,u_1):=\prod_{i=1}^nm_i(u_0,u_1)$. 
The involution $\varphi\colon\PP^1\times\PP^1\dashrightarrow\PP^1\times\PP^1$,
\[\varphi\colon ([u_0:u_1],[v_0:v_1])\dasharrow([u_0:u_1],[u_1v_1\overline{P}(u_0,u_1):u_0v_0P(u_0,u_1)])\]
commutes with the antimeromorphic involution $\sigma_C$ and is undefined exactly at $\varepsilon(q_1),\dots,\varepsilon(q_n)$, $\varepsilon(\overline{q_1}),\dots,\varepsilon(\overline{q_n})$, $([1:0],[0:1])$, $([0:1],[1:0])$ and exchanges $\varepsilon(s)$ and $\varepsilon(\bar{s})$. The map $\varphi$ is visualised in Figure~\ref{fig:CBX_2 phi}. \par
The involution $\varphi$ thus lifts via $\varepsilon$ to a real birational involution of $X_{[2]}$ that exchanges $s,\bar{s}$ and is undefined exactly at $q_1,\dots,q_n$, $\overline{q_1},\dots,\overline{q_n}$. So, it lifts to a real automorphism of $X$ that exchanges $s',\bar{s}'$. Therefore $\Aut_\R(X/\pi)$ acts non-trivially on the set $\{s',\bar{s}'\}$, which yields the split exact sequence
\[1\rightarrow K\rightarrow\Aut_\R(X/\pi)\rightarrow\Z/2\Z\rightarrow0.\]
This also shows that we cannot contract any components of the singular fibres $\Aut_\R(X/\pi)$-equivariantly, and hence also not $\Aut_\R(X,\pi)$-equivariantly. In particular, $\pi\colon X\rightarrow\PP^1$ is relatively $\Aut_\R(X,\pi))$-minimal.\par
By definition of $K$, all of its elements fix $s_1',s_2'$ and thus descend to a subgroup of $\Aut_\R(X_{[2]}/\pi_{[2]})$ and hence via $\pr_1\colon X_{[2]}\rightarrow\QQ_{3,1}$ to a subgroup of $\Aut(\QQ_{3,1},p,\bar{p})$, the automorphism group of $\QQ_{3,1}$ fixing $p$ and $\bar{p}$, which is isomorphic to $\{(A,\bar{A})\in\mathrm{PGL}_2(\C)^2\mid A\ \text{diagonal}\}$ via $\QQ_{3,1}\simeq (\PP^1\times\PP^1,\sigma_S)$ \cite[Lemma 4.5]{R15}. On the other hand, any $(\mathrm{diag}(a,1),\mathrm{diag}(\bar{a},1))\in\Aut(\QQ_{3,1},p,\bar{p})$ lifts to the real automorphism 
\[\beta_a\colon([x_0:x_1:x_2],[y_0:y_1:y_2])\mapsto([a\bar{a}x_0:x_1:\bar{a}x_2],[y_0:a\bar{a}y_1:a y_2])\]
of $X_{[2]}$ which fixes $s$ and $\bar{s}$. On them, it acts by $[y_0:y_1]\mapsto[y_0:a\bar{a}y_1]$ and $[x_0:x_1]\mapsto[a\bar{a}x_0:x_1]$, respectively. Hence, $\beta_a$ lifts to an automorphism of $X$ if and only if it fixes the points blown up by $\eta$ (there is at least one), which is equivalent to $a\bar{a}=1$. The lift of $\beta_a$ on $X$ then descends via $\pi_{[2]}$ to the identity map on $\PP^1$.
It follows that $(\mathrm{diag}(a,1),\mathrm{diag}(\bar{a},1))\in\Aut_\R(\QQ_{3,1},p,\bar{p})$ is contained in $\Aut(X/\pi)$ if and only if $a\bar{a}=1$, which implies that $K=\{(\mathrm{diag}(a,1),\mathrm{diag}(\bar{a},1))\in\mathrm{PGL}_2(\C)^2\mid a\bar{a}=1\}$. Conjugating $K$ with the real isomorphism $\mathcal{Q}_{3,1}\longrightarrow(\PP^1\times\PP^1,\sigma_S)$ from Remark~\ref{rmk:real part} yields $K\simeq\mathrm{SO}_2(\R)$. This finishes the proof of (\ref{CBX_2 1}) and yields (\ref{CBX_2 2}), (\ref{CBX_2 4}) and the first half of (\ref{CBX_2 3}).\par
The group $\mathrm{SO}_2(\R)$ acts via $\varepsilon$ on $(\PP^1\times\PP^1,\sigma_C)$ by 
\[([u_0:u_1],[v_0:v_1])\mapsto([u_0:u_1],[v_0:av_1]).\] 
It follows that $(a,\varphi)\in\Aut(X/\pi)\setminus \mathrm{SO}_2(\R)$ is an involution on $(\PP^1\times\PP^1,\sigma_C)$ fixing the irreducible curve $au_0v_0^2P(u_0,u_1)-u_1v_1^2\overline{P}(u_0,u_1)=0$, which is an irreducible double cover of $\PP^1$ ramified at $\Delta$ and $([1:0],[0:1])$, $([0:1],[1:0])$. Its strict transform on $X_{[2]}$ is an irreducible double cover over $\PP^1$ ramified over $\Delta$. This yields the second part of (\ref{CBX_2 3}).
\end{proof}

\begin{Rmk}\label{rmk:link}
Note that the generator of $\Z/2\Z$ in the description of $\Aut(X,\pi)$ in Proposition~\ref{prop:CBX_2} is the composition of elementary links of the real conic bundle $\pi_{[2]}\colon X_{[2]}\rightarrow\PP^1$, each blowing up a pair of non-real points on $s\cup\bar{s}$.
\end{Rmk}
\begin{minipage}{1\textwidth}
\centering
\def\svgwidth{0.65\columnwidth}
\begingroup%
  \makeatletter%
  \providecommand\color[2][]{%
    \errmessage{(Inkscape) Color is used for the text in Inkscape, but the package 'color.sty' is not loaded}%
    \renewcommand\color[2][]{}%
  }%
  \providecommand\transparent[1]{%
    \errmessage{(Inkscape) Transparency is used (non-zero) for the text in Inkscape, but the package 'transparent.sty' is not loaded}%
    \renewcommand\transparent[1]{}%
  }%
  \providecommand\rotatebox[2]{#2}%
  \ifx\svgwidth\undefined%
    \setlength{\unitlength}{460.4227814bp}%
    \ifx\svgscale\undefined%
      \relax%
    \else%
      \setlength{\unitlength}{\unitlength * \real{\svgscale}}%
    \fi%
  \else%
    \setlength{\unitlength}{\svgwidth}%
  \fi%
  \global\let\svgwidth\undefined%
  \global\let\svgscale\undefined%
  \makeatother%
  \begin{picture}(1,0.70952476)%
    \put(0,0){\includegraphics[width=\unitlength,page=1]{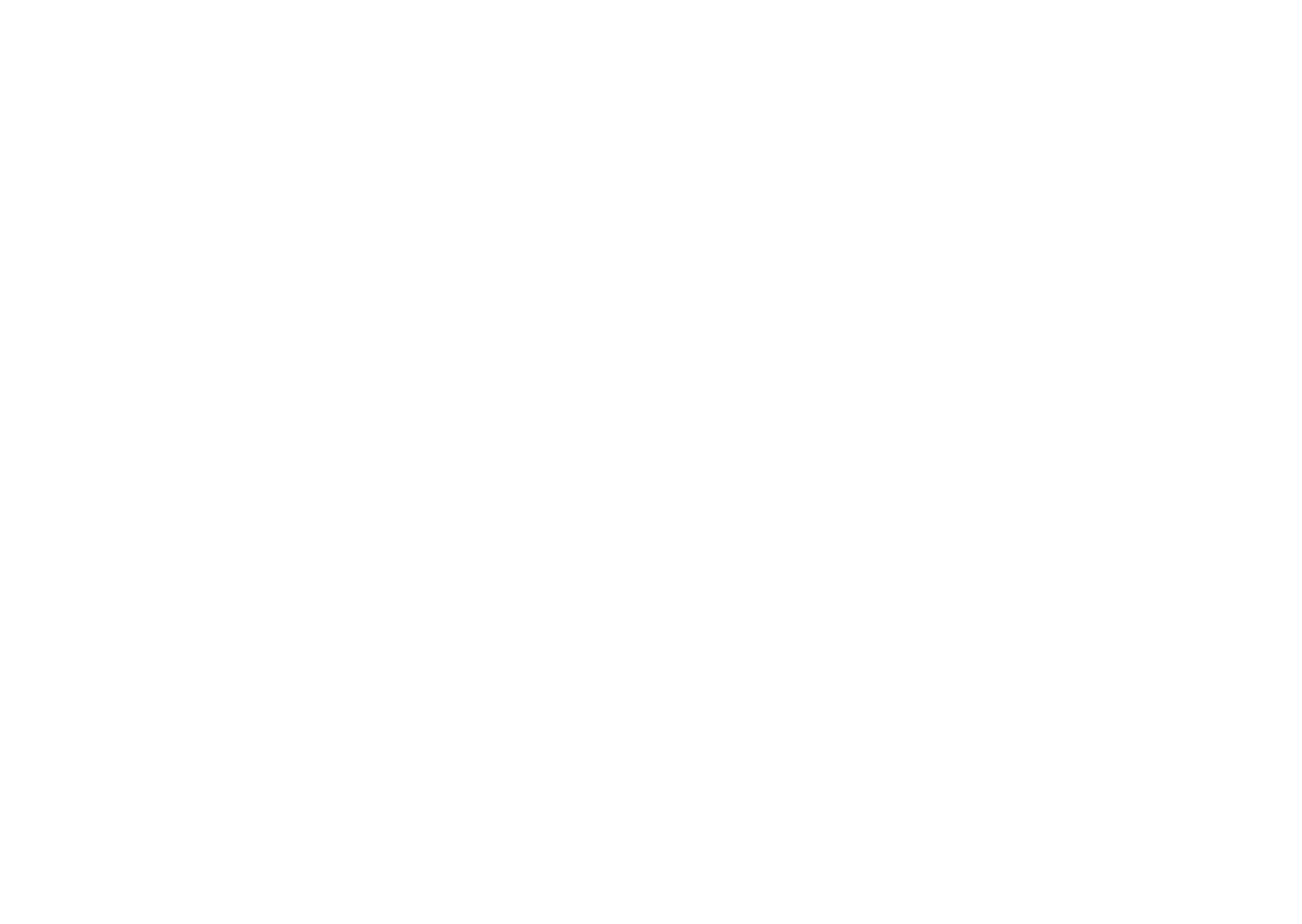}}%
    \put(-0.09124769,0.34174819){\color[rgb]{0,0,0}\makebox(0,0)[lb]{\smash{\SB{$(\PP^1\times\PP^1,\sigma_C)$}}}}%
    \put(-0.14510059,-0.00164326){\color[rgb]{0,0,0}\makebox(0,0)[lb]{\smash{\SB{$([0:1],[1:0])$}}}}%
    \put(0.29901923,0.28814612){\color[rgb]{0,0,0}\makebox(0,0)[lb]{\smash{\SB{$([1:0],[0:1])$}}}}%
    \put(0,0){\includegraphics[width=\unitlength,page=2]{fig07.pdf}}%
    \put(0.21874644,0.00487365){\color[rgb]{0,0,0}\makebox(0,0)[lb]{\smash{\SB{$p_i$}}}}%
    \put(0.0874793,0.22275112){\color[rgb]{0,0,0}\makebox(0,0)[lb]{\smash{\SB{$p_j$}}}}%
    \put(0,0){\includegraphics[width=\unitlength,page=3]{fig07.pdf}}%
    \put(0.47834438,0.6910953){\color[rgb]{0,0,0}\makebox(0,0)[lb]{\smash{\SB{$s_2'$}}}}%
    \put(0.28848432,0.61551195){\color[rgb]{0,0,0}\makebox(0,0)[lb]{\smash{\SB{$F_{[([0:1],[1:0])}$}}}}%
    \put(0.28724077,0.45254793){\color[rgb]{0,0,0}\makebox(0,0)[lb]{\smash{\SB{$E_{([0:1],[1:0])}$}}}}%
    \put(0.48074716,0.34838426){\color[rgb]{0,0,0}\makebox(0,0)[lb]{\smash{\SB{$s_1'$}}}}%
    \put(0.68657192,0.44986066){\color[rgb]{0,0,0}\makebox(0,0)[lb]{\smash{\SB{$F_{([1:0],[0:1])}$}}}}%
    \put(0.67830294,0.6221609){\color[rgb]{0,0,0}\makebox(0,0)[lb]{\smash{\SB{$E_{([1:0],[0:1])}$}}}}%
    \put(0,0){\includegraphics[width=\unitlength,page=4]{fig07.pdf}}%
    \put(0.47481501,0.46735262){\color[rgb]{0,0,0}\makebox(0,0)[lb]{\smash{\SB{$F_{p_j}$}}}}%
    \put(0.56416744,0.58583246){\color[rgb]{0,0,0}\makebox(0,0)[lb]{\smash{\SB{$F_{p_i}$}}}}%
    \put(0.48057245,0.58598083){\color[rgb]{0,0,0}\makebox(0,0)[lb]{\smash{\SB{$E_{p_j}$}}}}%
    \put(0.55840093,0.46721698){\color[rgb]{0,0,0}\makebox(0,0)[lb]{\smash{\SB{$E_{ p_i}$}}}}%
    \put(0,0){\includegraphics[width=\unitlength,page=5]{fig07.pdf}}%
    \put(0.92184121,0.34390844){\color[rgb]{0,0,0}\makebox(0,0)[lb]{\smash{\SB{$(\PP^1\times\PP^1,\sigma_C)$}}}}%
    \put(0.48998025,0.23125813){\color[rgb]{0,0,0}\makebox(0,0)[lb]{\smash{\SB{$\varphi(F_{([0:1],[1:0])})$}}}}%
    \put(0.96016239,-0.01512156){\color[rgb]{0,0,0}\makebox(0,0)[lb]{\smash{\SB{$=([1:0],[0:1])$}}}}%
    \put(0,0){\includegraphics[width=\unitlength,page=6]{fig07.pdf}}%
    \put(0.84926058,0.22064993){\color[rgb]{0,0,0}\makebox(0,0)[lb]{\smash{\SB{$p_i$}}}}%
    \put(0.7226145,0.00728376){\color[rgb]{0,0,0}\makebox(0,0)[lb]{\smash{\SB{$p_j$}}}}%
    \put(0,0){\includegraphics[width=\unitlength,page=7]{fig07.pdf}}%
    \put(0.88365333,-0.02683591){\color[rgb]{0,0,0}\makebox(0,0)[lb]{\smash{\SB{$E_{([1:0],[0:1])}$}}}}%
    \put(0.25701502,-0.03004161){\color[rgb]{0,0,0}\makebox(0,0)[lb]{\smash{\SB{$F_{([1:0],[0:1])}$}}}}%
    \put(-0.02112439,-0.0286813){\color[rgb]{0,0,0}\makebox(0,0)[lb]{\smash{\SB{$F_{[([0:1],[1:0])}$}}}}%
    \put(0.61258252,-0.02596513){\color[rgb]{0,0,0}\makebox(0,0)[lb]{\smash{\SB{$E_{[([0:1],[1:0])}$}}}}%
    \put(0.81845372,-0.0299714){\color[rgb]{0,0,0}\makebox(0,0)[lb]{\smash{\SB{$E_{p_i}$}}}}%
    \put(0.19252828,-0.02736892){\color[rgb]{0,0,0}\makebox(0,0)[lb]{\smash{\SB{$F_{p_i}$}}}}%
    \put(0.10880677,-0.02817353){\color[rgb]{0,0,0}\makebox(0,0)[lb]{\smash{\SB{$F_{p_j}$}}}}%
    \put(0.74691335,-0.02754125){\color[rgb]{0,0,0}\makebox(0,0)[lb]{\smash{\SB{$E_{p_j}$}}}}%
    \put(0,0){\includegraphics[width=\unitlength,page=8]{fig07.pdf}}%
    \put(0.46004507,0.15733771){\color[rgb]{0,0,0}\makebox(0,0)[lb]{\smash{\SB{$\varphi$}}}}%
    \put(0.27138827,0.5166892){\color[rgb]{0,0,0}\makebox(0,0)[lb]{\smash{\SB{$X$}}}}%
    \put(0.34282201,0.367537){\color[rgb]{0,0,0}\makebox(0,0)[lb]{\smash{\SB{$\eta\varepsilon$}}}}%
    \put(0.71527049,0.367537){\color[rgb]{0,0,0}\makebox(0,0)[lb]{\smash{\SB{$\eta\varepsilon$}}}}%
    \put(0.98130528,0.25155224){\color[rgb]{0,0,0}\makebox(0,0)[lb]{\smash{\SB{$\eta\varepsilon(s_1)=\varphi(\eta\varepsilon(s_2))$}}}}%
    \put(0.98343361,0.03553825){\color[rgb]{0,0,0}\makebox(0,0)[lb]{\smash{\SB{$\eta\varepsilon(s_2)=\varphi(\eta\varepsilon(s_1))$}}}}%
    \put(-0.08934546,0.24674777){\color[rgb]{0,0,0}\makebox(0,0)[lb]{\smash{\SB{$\eta\varepsilon(s_2)$}}}}%
    \put(-0.08541116,0.03691765){\color[rgb]{0,0,0}\makebox(0,0)[lb]{\smash{\SB{$\eta\varepsilon(s_1)$}}}}%
    \put(0.49226357,0.20043292){\color[rgb]{0,0,0}\makebox(0,0)[lb]{\smash{\SB{$=([0:1],[1:0])$}}}}%
    \put(0.96383789,0.01096997){\color[rgb]{0,0,0}\makebox(0,0)[lb]{\smash{\SB{$\varphi(F_{([1:0],[0:1])})$}}}}%
    \put(0,0){\includegraphics[width=\unitlength,page=9]{fig07.pdf}}%
    \put(0.78588886,0.52229541){\color[rgb]{0,0,0}\makebox(0,0)[lb]{\smash{\SB{$\varphi$}}}}%
  \end{picture}%
\endgroup%
\\

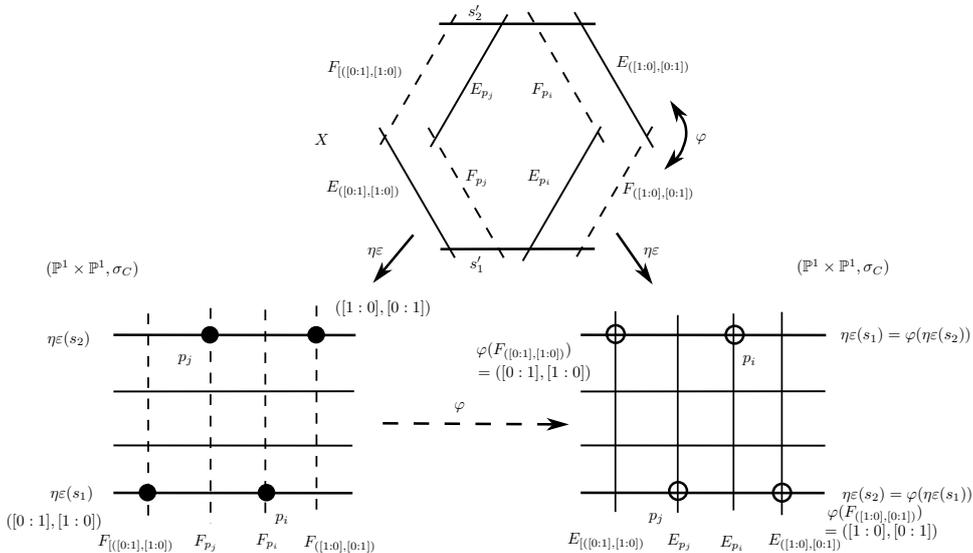
\captionof{figure}{The real birational map $\varphi\colon(\PP^1\times\PP^1,\sigma_C)\rightarrow(\PP^1\times\PP^1,\sigma_C)$ that lifts to an automorphism of $X$ generating $\Z/2\Z\subset\Aut_\R(X/\pi)$}\label{fig:CBX_2 phi}
\end{minipage}

\subsection{Real conic bundles obtained by blowing up a Hirzebruch surface}\label{ssec:CB F}
The $n$-th Hirzebruch surface is given by
\begin{align*}\FF_n&=\{([x_0:x_1:x_2],[u:v])\in\PP^2\times\PP^1\mid u^nx_2=v^nx_1\}\subset\PP^2\times\PP^1\\
\end{align*}
The canonical projection $\mathrm{pr}_n\colon\FF_n\rightarrow\PP^1$ onto the second factor makes it a real conic bundle. In this subsection, we study real conic bundles with a birational morphism $X\rightarrow\FF_n$ of real conic bundles.\par
If $n=0$, then $\FF_0=\PP^1\times\PP^1$ with the standard antiholomorphic involution on it, and $\Aut_\R(\FF_0,\pr_0)=\mathrm{PGL}_2(\R)^2$.\par
If $n=1$, then $\FF_1$ is isomorphic to the blow-up of $[1:0:0]\in\PP^2$, and any automorphism preserves the unique $(-1)$-curve on it, which yields $\Aut_\R(\FF_1)=\Aut_\R(\FF_1,\pr_1)\simeq\Aut_\R(\PP^2,[1:0:0])$.\par 
If $n\geq2$, the automorphism group of $\FF_n$ is
\[\Aut_\R(\FF_n)\simeq\R^{n+1}\rtimes\mathrm{GL}_2(\R)/\mu_n,\]
where $\mu_n=\{\mu\Id\mid \mu^n=1\}$ \cite{B10}. An element
\[\left((a_0,\dots,a_n),\left(\begin{matrix}a&b\\c&d\end{matrix}\right)\right)\in\R^{n+1}\rtimes\mathrm{GL}_2(\R)/\mu_n\]
acts on the chart $u\neq0$ by
\begin{align*}([xu^n:&yu^n:yv^n],[u:v])\mapsto\\
&([xu^n+y(a_0v^n+a_1uv^{n-1}+\cdots+a_nu^n):y(au+bv)^n:y(cu+dv)^n],[au+bv:cu+dv])
\end{align*}
and in particular respects the conic bundle structure on $\FF_n$. Multiples of the identity matrix act trivially on the base, and we get
\[\Aut_\R(\FF_n/\pr_n)\simeq \R^{n+1}\rtimes(\R^*/\mu_n),\] 
where we see $\mu_n\subset\R^*$ as $\mu_n=\{\pm1\}$ if $n$ is even and $\mu_n=\{1\}$ if $n$ is odd. \par
We denote by $E_n:=\{([1:0:0],[u:v])\mid [u:v]\in\PP^1\}\subset\FF_n$ its $(-n)$-section, by $f$ the general fibre of $\pr_n$ and by $s_n\subset\FF_n$ the section given by $x_0=0$, i.e. 
\[s_n:=\{([0:u^n:v^n],[u:v])\mid[u:v]\in\PP^1\}.\]
The conditions $s_nE_n=0$ and $s_nf=1$ yield $s_n\sim E_n+nf$ as divisors and hence $s_n^2=n$.\par
Let us give a necessary description of minimal pairs $(X,\Aut_\R(X,\pi))$ equipped with a birational morphism $\eta\colon X\rightarrow\FF_n$ of real conic bundles. \par

\begin{Lem}\label{lem:CBF}
Let $\pi\colon X\rightarrow\PP^1$ be a relatively $\Aut_\R(X,\pi))$-minimal real conic bundle equipped with a brational morphism $\eta\colon X\rightarrow\FF_n$ of real conic bundles that is not an isomorphism. \par
If $\Aut_\R(X/\pi)\cap\ker\left(\ \Aut_\R(X,\pi)\rightarrow\Aut(\Pic(X))\ \right)$ contains a non-trivial element, then there exists a birational morphism $X\rightarrow\FF_N$ of real conic bundles blowing up $2N\geq2$ points, all contained in $s_N$ with pairwise distinct fibres.\par
Else, $\Aut_\R(X/\pi)\simeq(\Z/2\Z)^r$ for $r\in\{0,1,2\}$.
\end{Lem}
\begin{proof}
If $G:=\Aut_\R(X/\pi)\cap\ker\left(\ \Aut_\R(X,\pi)\rightarrow\Aut(\Pic(X))\ \right)$ is trivial, the claim follows from Lemma~\ref{lem:blowup CB}. Suppose that $G$ is non-trivial. By Lemma~\ref{lem:blowup CB}, there exists a (perhaps non-real) birational morphism $X\rightarrow\FF_N$, $N\geq1$, blowing up $2N$ points on a section $s$ of $\FF_N$ disjoint from $E_N$ and of self-intersection $N$. By $c_1$ and $c_2$ we denote the strict transforms of $s$ and $E_N$ in $X$. 
The real birational morphism $\eta$ blows down one component of each fibre. Furthermore, $c_1$ is sent onto a curve of self-intersection $N-r$, where $r$ is the number of contracted components intersecting $\eta(c_1)$, and $\eta(c_2)$ is a curve of self-intersection $N-(2N-r)=-N+r$. We can assume that $r\leq N$ (else we exchange the indices of $c_1$ and $c_2$). Then $\eta(c_2)^2\leq0$, hence $\eta(c_2)$ is the exceptional section of $\FF_n$. In particular, it is a real curve, hence also $c_2$ is a real curve. By Lemma~\ref{lem:blowup CB} there is an element of $\Aut_\R(X,\pi)$ exchanging $c_1$ and $c_2$, thus $c_1$ is a real curve as well. Therefore, only contracting components intersecting $c_1$ commutes with the antiholomorphic involution of $X$. It follows that the birational morphism $X\rightarrow\FF_N$ from Lemma~\ref{lem:blowup CB} is in fact a real morphism.
\end{proof}

In the first assertion of the lemma, the cases $N=0$ and $N=1$ yield relatively $\Aut_\R(X,\pi))$-minimal conic bundles, but the groups $\Aut_\R(X,\pi)$ are not maximal algebraic subgroups of $\Bir_{\R}(\PP^2)$. 

\begin{Prop}\label{prop:CBF}
Let $\eta\colon X\rightarrow\FF_n$ be the blow-up of $2n\geq 4$ points contained in $s_n$. Then $\pi:=\pr_n\eta\colon X\rightarrow\PP^1$ is a is relatively $\Aut_\R(X,\pi))$-minimal conic bundle. \par
Let $\Delta\subset\PP^1$ be the projection of the points blown up by $\eta$, $H_{\Delta}\subset\mathrm{PGL}_2(\R)$ the subgroup preserving $\Delta$ and $\mu_n=\{\pm1\}$ if $n$ is even and $\mu_n=\{1\}$ if $n$ is odd. Then: 
\begin{enumerate}
\item\label{CBF 1} There is a split exact sequence
\[1\rightarrow\Aut_\R(X/\pi)\rightarrow\Aut_\R(X,\pi)\rightarrow H_{\Delta}\rightarrow 1\]
where $\Aut_\R(X/\pi)\simeq (\R^*/\mu_n)\rtimes\Z/2\Z$.
\item\label{CBF 2} An element of $(\R^*/\mu_n)\subset\Aut(X/\pi)$ fixes the two $(-n)$-sections of $X$ and the generator of $\Z/2\Z$ exchanges them.
\item\label{CBF 3} An element of $\Aut(X/\pi)\setminus(\R^*/\mu_n)$ is an involution fixing an irreducible curve on $\FF_n$ which is a double cover of $\PP^1$ ramified at $\Delta$.
\item\label{CBF 4} The group $\R^*/\mu_n$ acts trivially on $\Pic(X)$.
\end{enumerate}
\end{Prop}
\begin{proof}
Denote by $\tilde{s}_n,\tilde{E}_n\subset X$ the strict transforms of $s_n$ and $E_n$, respectively. Then $\pi:=\pr\eta\colon X\rightarrow\PP^1$ is a conic bundle with $2n$ singular fibres, whose components either intersect $\tilde{s}_n$ or $\tilde{E}_n$. The action of $\Aut_\R(X,\pi)$ descends to an action on $\PP^1$ that preserves the set $\Delta$. On the other hand, any element of $H_{\Delta}$ lifts to an automorphism of $\FF_n$ preserving the set of the points blown-up by $\eta$ and hence lifts to an automorphism of $X$ that permutes the singular fibres. Thus the sequence splits.\par
The group $\Aut_\R(X/\pi)$ acts on the set $\{\tilde{E}_n,\tilde{s}_n\}$ non-trivially: We can assume that the points blown up by $\eta$ are in the chart $u=1$. They are thus of the form $p_i=([0:1:v_i^n],[1:v_i])\in s_n$, $i=1,\dots,2n$. The number of non-real points in $\Delta$ is even, hence the number of real points in $\Delta$ is even as well. We order the points such that $p_1,\dots,p_{2k}$ are real points and $p_{2k+1},p_{2k+2}:=\overline{p_{2k+1}},\dots,p_{2n-1},p_{2n}:=\overline{p_{2n-1}}$ are pairs of non-real conjugate points. For $i=1,\dots,2k$, let $l_i:=t-v_i^n\in\R[t]$. For $i=1,\dots,n-k$, let $m_i:=(t-v_{2(k+i)-1}^n)(t-v_{2(k+i)}^n)\in\R[t]$. We define 
\[P(r):=\prod_{i=1}^{k}l_i\prod_{i=1}^{n-k}m_i\quad\in\ \R[t].\] 
Then the rational map $\varphi\colon\FF_n\dashmapsto\FF_n$ given on the chart $u=1$ by
\[\varphi\colon([x_0:x_1:x_1v^n],[1:v])\dashmapsto([x_1P(v):x_0:x_0v^n],[1:v])\]
is a real involution respecting the conic bundle structure of $\FF_n$. It is undefined exactly at the points $p_1,\dots,p_{2n}$ and exchanges $E_n$ and $s_n$. It furthermore contracts the fibre through $p_i$ onto $p_i$. The map $\varphi$ is visualised in Figure~\ref{fig:CBF phi}.\par
Therefore, $\varphi$ lifts to an automorphism of the conic bundle $X$ that exchanges $\tilde{s}_n$ and $\tilde{E}_n$. This induces the split exact sequence
\[1\rightarrow K\rightarrow\Aut_\R(X/\pi)\rightarrow\Z/2\Z\rightarrow1.\]
It also proves that we cannot contract any components of the singular fibres on $X$ $\Aut_\R(X/\pi)$-equivariantly and thus also not $\Aut_\R(X,\pi)$-equivariantly. In particular, $\pi\colon X\rightarrow\PP^1$ is relatively $\Aut_\R(X,\pi))$-minimal.\par
By definition of $K$, all of its elements fix $\tilde{s}_n$ and $\tilde{E}_n$ pointwise and thus descend to elements of $\Aut_\R(\FF_n/\pr)$. On the other hand, an element $((a_0,\dots,a_n),r)\in\Aut_\R(\FF_n/\pr_n)\simeq\R^{n+1}\rtimes(\R^*\Id)/\mu_n$ acts on the chart $u\neq0$ by
\[([x_0:x_1:v^nx_1],[1:v])\longmapsto([x_0+x_1(a_nv^n+a_{n-1}v^{n-1}+\cdots a_0):r^nx_1:r^nv^nx_1],[1:v])\]
and lifts to an automorphism of $X$ if and only if it preserves $s_n$ (which is given by $x_0=0$). It follows that $K=\R^*/\mu_n$. This completes (\ref{CBF 1}) and (\ref{CBF 2}).\par
Every element $(r,\varphi)\in\Aut_\R(X/\pi)\setminus(\R^*/\mu_n)$ fixes the curve $x_1^2P(v)-r^nx_0^2=0$ (given on the chart $u\neq0$), which is a double cover of $\PP^1$ ramified over $\Delta$. This is (\ref{CBF 3}).\par
Finally, the action of $\R^*/\mu_n\subset\Aut_\R(X/\pi)$ fixes each fibre and it fixes $E_n$ and $s_n$. Hence it acts trivially on $\Pic(X)$, which is (\ref{CBF 4}).
\end{proof}
\begin{minipage}{1\textwidth}
\centering
\def\svgwidth{0.54\columnwidth}
\begingroup%
  \makeatletter%
  \providecommand\color[2][]{%
    \errmessage{(Inkscape) Color is used for the text in Inkscape, but the package 'color.sty' is not loaded}%
    \renewcommand\color[2][]{}%
  }%
  \providecommand\transparent[1]{%
    \errmessage{(Inkscape) Transparency is used (non-zero) for the text in Inkscape, but the package 'transparent.sty' is not loaded}%
    \renewcommand\transparent[1]{}%
  }%
  \providecommand\rotatebox[2]{#2}%
  \ifx\svgwidth\undefined%
    \setlength{\unitlength}{439.70408042bp}%
    \ifx\svgscale\undefined%
      \relax%
    \else%
      \setlength{\unitlength}{\unitlength * \real{\svgscale}}%
    \fi%
  \else%
    \setlength{\unitlength}{\svgwidth}%
  \fi%
  \global\let\svgwidth\undefined%
  \global\let\svgscale\undefined%
  \makeatother%
  \begin{picture}(1,0.7620834)%
    \put(0,0){\includegraphics[width=\unitlength,page=1]{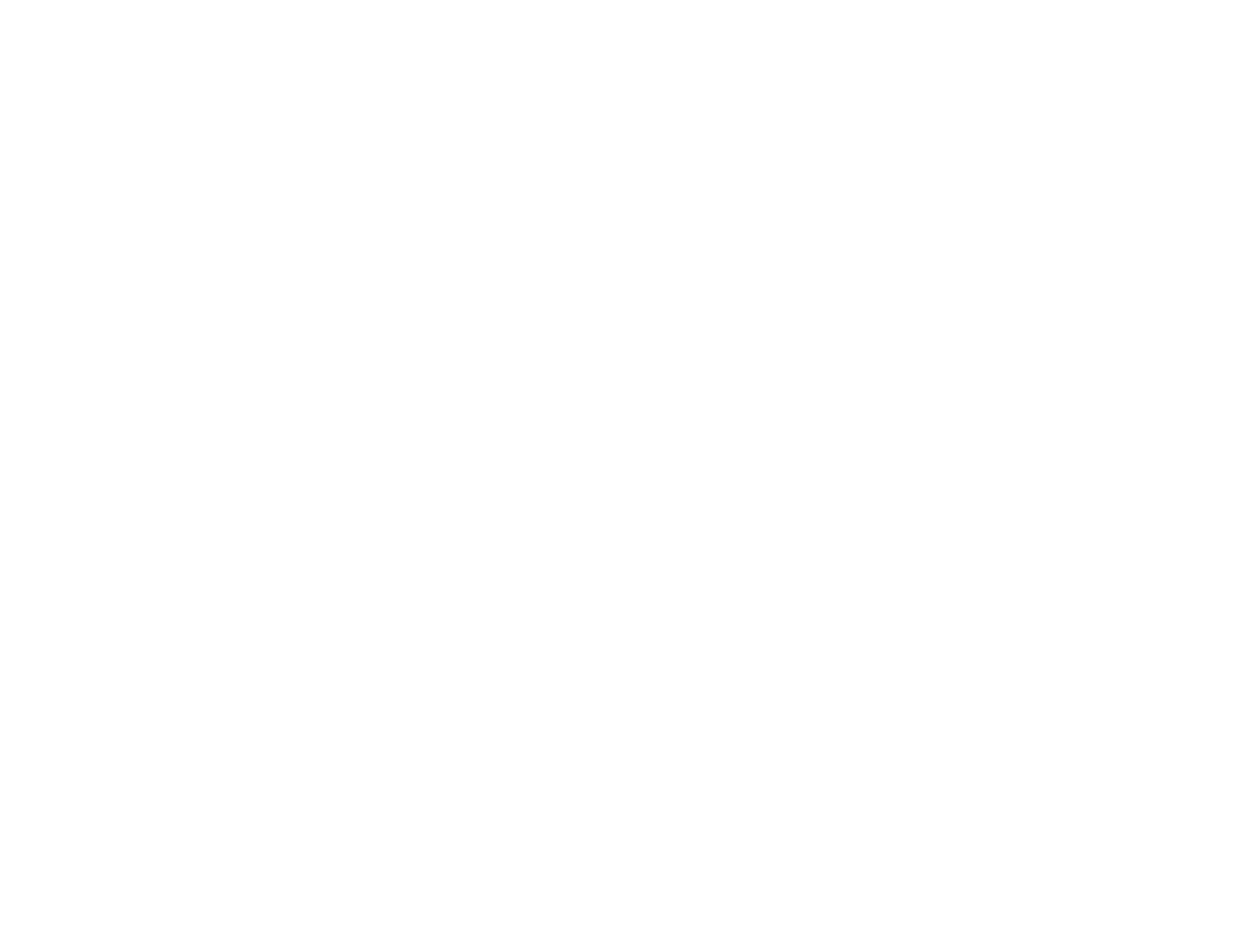}}%
    \put(-0.05186703,0.38962282){\color[rgb]{0,0,0}\makebox(0,0)[lb]{\smash{\SB{$\FF_n$}}}}%
    \put(0,0){\includegraphics[width=\unitlength,page=2]{fig08.pdf}}%
    \put(0.32004816,0.67628623){\color[rgb]{0,0,0}\makebox(0,0)[lb]{\smash{\SB{$E_{p_1}$}}}}%
    \put(0.32320564,0.5056434){\color[rgb]{0,0,0}\makebox(0,0)[lb]{\smash{\SB{$f_{p_1}$}}}}%
    \put(0.66432306,0.5028295){\color[rgb]{0,0,0}\makebox(0,0)[lb]{\smash{\SB{$f_{p_4}$}}}}%
    \put(0.65566445,0.68324847){\color[rgb]{0,0,0}\makebox(0,0)[lb]{\smash{\SB{$E_{p_4}$}}}}%
    \put(0,0){\includegraphics[width=\unitlength,page=3]{fig08.pdf}}%
    \put(0.44258824,0.52114568){\color[rgb]{0,0,0}\makebox(0,0)[lb]{\smash{\SB{$f_{p_2}$}}}}%
    \put(0.53615092,0.64520825){\color[rgb]{0,0,0}\makebox(0,0)[lb]{\smash{\SB{$E_{p_3}$}}}}%
    \put(0.44861697,0.64536361){\color[rgb]{0,0,0}\makebox(0,0)[lb]{\smash{\SB{$E_{p_2}$}}}}%
    \put(0.54574822,0.52100365){\color[rgb]{0,0,0}\makebox(0,0)[lb]{\smash{\SB{$f_{p_3}$}}}}%
    \put(0,0){\includegraphics[width=\unitlength,page=4]{fig08.pdf}}%
    \put(0.90956135,0.38518396){\color[rgb]{0,0,0}\makebox(0,0)[lb]{\smash{\SB{$\FF_n$}}}}%
    \put(0,0){\includegraphics[width=\unitlength,page=5]{fig08.pdf}}%
    \put(0.48283637,0.19652298){\color[rgb]{0,0,0}\makebox(0,0)[lb]{\smash{\SB{$\varphi$}}}}%
    \put(0.25066479,0.57280699){\color[rgb]{0,0,0}\makebox(0,0)[lb]{\smash{\SB{$X$}}}}%
    \put(0.30103297,0.41328397){\color[rgb]{0,0,0}\makebox(0,0)[lb]{\smash{\SB{$\eta$}}}}%
    \put(0.79065325,0.40548398){\color[rgb]{0,0,0}\makebox(0,0)[lb]{\smash{\SB{$\eta$}}}}%
    \put(0.27359763,0.00366228){\color[rgb]{0,0,0}\makebox(0,0)[lb]{\smash{\SB{$f_{p_4}$}}}}%
    \put(0.03259915,0.0095387){\color[rgb]{0,0,0}\makebox(0,0)[lb]{\smash{\SB{$f_{p_1}$}}}}%
    \put(0.20159997,0.00698805){\color[rgb]{0,0,0}\makebox(0,0)[lb]{\smash{\SB{$f_{p_3}$}}}}%
    \put(0.12171831,0.00696873){\color[rgb]{0,0,0}\makebox(0,0)[lb]{\smash{\SB{$f_{p_2}$}}}}%
    \put(0.33736827,0.06910141){\color[rgb]{0,0,0}\makebox(0,0)[lb]{\smash{\SB{$E_n$ $[-n]$}}}}%
    \put(0.24833323,0.26425225){\color[rgb]{0,0,0}\makebox(0,0)[lb]{\smash{\SB{$p_4$}}}}%
    \put(0,0){\includegraphics[width=\unitlength,page=6]{fig08.pdf}}%
    \put(0.16996293,0.26437379){\color[rgb]{0,0,0}\makebox(0,0)[lb]{\smash{\SB{$p_3$}}}}%
    \put(0.09160118,0.26501864){\color[rgb]{0,0,0}\makebox(0,0)[lb]{\smash{\SB{$p_2$}}}}%
    \put(0,0){\includegraphics[width=\unitlength,page=7]{fig08.pdf}}%
    \put(0.00739704,0.26649591){\color[rgb]{0,0,0}\makebox(0,0)[lb]{\smash{\SB{$p_1$}}}}%
    \put(0.34628507,0.29580556){\color[rgb]{0,0,0}\makebox(0,0)[lb]{\smash{\SB{$s_n$ $[n]$}}}}%
    \put(0.48801985,0.06910141){\color[rgb]{0,0,0}\makebox(0,0)[lb]{\smash{\SB{$\varphi(E_n)=s_n$ $[n]$}}}}%
    \put(0,0){\includegraphics[width=\unitlength,page=8]{fig08.pdf}}%
    \put(0.81937985,0.10682665){\color[rgb]{0,0,0}\makebox(0,0)[lb]{\smash{\SB{$\varphi(f_{p_2})$}}}}%
    \put(0,0){\includegraphics[width=\unitlength,page=9]{fig08.pdf}}%
    \put(0.95035217,0.01369089){\color[rgb]{0,0,0}\makebox(0,0)[lb]{\smash{\SB{$E_{p_4}$}}}}%
    \put(0.70489656,0.01399582){\color[rgb]{0,0,0}\makebox(0,0)[lb]{\smash{\SB{$E_{p_1}$}}}}%
    \put(0.876126,0.01255942){\color[rgb]{0,0,0}\makebox(0,0)[lb]{\smash{\SB{$E_{p_3}$}}}}%
    \put(0.78064441,0.01142585){\color[rgb]{0,0,0}\makebox(0,0)[lb]{\smash{\SB{$E_{p_2}$}}}}%
    \put(0.73517582,0.10908702){\color[rgb]{0,0,0}\makebox(0,0)[lb]{\smash{\SB{$\varphi(f_{p_1})$}}}}%
    \put(0.48764795,0.29358205){\color[rgb]{0,0,0}\makebox(0,0)[lb]{\smash{\SB{$\varphi(s_n)=E_n$ $[-n]$}}}}%
    \put(0.63724919,0.73853118){\color[rgb]{0,0,0}\makebox(0,0)[lb]{\smash{\SB{$s_n$ $[-n]$}}}}%
    \put(0.63377462,0.41804542){\color[rgb]{0,0,0}\makebox(0,0)[lb]{\smash{\SB{$E_n$ $[-n]$}}}}%
    \put(0,0){\includegraphics[width=\unitlength,page=10]{fig08.pdf}}%
    \put(0.76781529,0.57555414){\color[rgb]{0,0,0}\makebox(0,0)[lb]{\smash{\SB{$\varphi$}}}}%
    \put(0.89650025,0.10517137){\color[rgb]{0,0,0}\makebox(0,0)[lb]{\smash{\SB{$\varphi(f_{p_3})$}}}}%
    \put(0.98577705,0.1043882){\color[rgb]{0,0,0}\makebox(0,0)[lb]{\smash{\SB{$\varphi(f_{p_4})$}}}}%
  \end{picture}%
\endgroup%
\\
\captionof{figure}{The real birational map $\varphi\colon\FF_n\dashrightarrow\FF_n$ that lifts to an automorphism of $X$.}\label{fig:CBF phi}
\end{minipage}

The following lemma is an adapted version of Lemma~\cite[Lemma 5.2.1]{B10}.

\begin{Lem}\label{lem:finite}
Let $\pi\colon X\rightarrow\PP^1$ be a relatively $\Aut_\R(X,\pi))$-minimal conic bundle with a birational morphism $\eta\colon X\rightarrow\FF_n$ of real conic bundles. Suppose that $\Aut_\R(X/\pi)\cap\ker(\Aut_\R(X,\pi)\rightarrow\Aut_\R(\Pic(X)))=\{1\}$. Then $\Aut_\R(X,\pi)$ is finite or strictly contained in the automorphism group of a real del Pezzo surface.
\end{Lem}
\begin{proof}
Suppose $X$ has at least three singular fibres. The action of $\Aut_\R(X,\pi)$ on $\PP^1$ induces the exact sequence 
\[1\rightarrow\Aut_\R(X/\pi)\rightarrow\Aut_\R(X,\pi)\rightarrow H,\]
where $H\subset\Aut_\R(\PP^1)$ is the subgroup fixing the image of the set of the singular fibres. The conic bundle $X$ having three singular fibres implies that $H$ is finite. By Lemma~\ref{lem:CBF}, $\Aut_\R(X/\pi)$ is finite as well, so $\Aut_\R(X,\pi)$ is finite.\par
Suppose $X$ has one or two singular fibres. Let $\tilde{E}_n$ be the strict transform of the $-n$ curve of $\FF_n$. It is of self-intersection $-r\leq -n$. As the pair $(X,\Aut_\R(X,\pi))$ is minimal, the singular fibres intersect $\tilde{E}_n$, have exactly two components and there exists $g\in\Aut_\R(X,\pi)$ exchanging the components of each fibre. Then $s:=g(\tilde{E}_n)\neq E_n$ is a real section of self-intersection $-r$. Forgetting about about the action of $\Aut_\R(X,\pi)$, we contract in each fibre the component intersecting $s$. This is a blow-down $\eta'\colon X\rightarrow\FF_r$, and $r\leq\eta'(s)^2=-r+m$, where $1\leq m\leq 2$ is the number of points blown-up by $\eta'$. It follows that $m=2$ and $r\in\{0,1\}$. The case $r=0$ is not possible, so $r=1$. Hence $\eta'$ blows up two points in different fibres of $\FF_1$, and $X$ is a del Pezzo surface of degree $6$. Therefore, $\Aut_\R(X,\pi)\subset\Aut_\R(X)$. Figures~\ref{fig:X_4} and \ref{fig:X_3,T} and Propositions~\ref{prop:X_3,T} and \ref{prop:DP64} imply that the inclusion is strict.
\end{proof}

\section{The maximal infinite algebraic subgroups}
This section aims at proving Theorem~\ref{thm:classification} and Theorem~\ref{thm:parametrisation}. We first prove that any infinite algebraic subgroup of $\Bir_{\R}(\PP^2)$ is contained in one of the groups in Theorem~\ref{thm:classification}. We then have to prove that all listed groups are in fact maximal. 
                                                                                                   
\begin{Prop}\label{prop:among maximal}
Let $G\subset\Bir_{\R}(\PP^2)$ be an infinite algebraic subgroup. Then $G$ is conjugate to a subgroup of one of the groups in Theorem~$\ref{thm:classification}$.
\end{Prop}
\begin{proof}
Proposition~\ref{prop which cases} states that for an infinite algebraic subgroup $G$ of $\Bir_{\R}(\PP^2)$ there exists a $G$-equivariant birational morphism $\PP^2\dashrightarrow X$ where $G$ acts on $X$ regularly and $X$ is one of the following:
\begin{enumerate}
\item $X$ is a Del Pezzo surface of degree 6, 8 or 9 such that $\rk(\Pic(X)^G)=1$.
\item $X$ admits a real conic bundle structure $\pi_X\colon X\rightarrow\PP^1$ with $\rk(\Pic(X)^G)=2$ and $G\subset\Aut_\R(X,\pi_X)$, and there is a birational morphism of conic bundles $\eta\colon X\to Y$, where $Y\simeq X_{[2]}$ is the sphere blown up in a pair of non-real conjugate points or $Y$ is a real Hirzebruch surface $Y=\FF_n$, $n\neq1$. 
\end{enumerate}
So $G$ is conjugate to a subgroup of $\Aut_\R(X)$ or $\Aut_\R(X,\pi)$, where $X$ is as in (1) or (2) respectively. The pairs $(X,\Aut_\R(X))$ and $(X,\Aut_\R(X,\pi))$ are described as follows. \par
In the first case, we get:
\begin{itemize}
\item If $\deg(X)=9$, then $X=\PP^2$ and $\Aut_\R(\PP^2)\simeq\mathrm{PGL}_3(\R)$.
\item If $\deg(X)=8$, then
$X=\FF_1$ and $\Aut_\R(X)$ is conjugate to a subgroup of $\mathrm{PGL}_3(\R)$, or $X$ is one of the following two by \cite{Com12}:\par
 $X= \mathcal{Q}_{3,1}$ and $\Aut_\R(\mathcal{Q}_{3,1})\simeq\mathbb{P}\mathrm{O}_\R(3,1)$. \par
 $X=\FF_0$ and $\Aut_\R(\FF_0)=\mathrm{PGL}_2(\R)^2\rtimes\langle\tau'\rangle$, where $\tau':(x,y)\mapsto(y,x)$.
\item If $\deg(X)=6$, then Lemma~\ref{lem:dP class} and Propositions~\ref{prop:X_2}, \ref{prop:DP63S}, \ref{prop:X_3,T} and \ref{prop:DP64} imply that \par
$X=X_{[2]}$ and $\Aut_\R(X_{[2]})$ is conjugate to a subgroup of $\Aut_\R(\mathcal{Q}_{3,1})$ (Proposition~\ref{prop:X_2}).\par
$X=X_{[3,\mathcal{Q}_{3,1}]}$ and $\Aut_\R(X_{[3,\mathcal{Q}_{3,1}]})$ is conjugate to a subgroup of $\Aut_\R(\mathcal{Q}_{3,1})$ (Proposition~\ref{prop:DP63S}),\par
$X=X_{[3,\FF_0]}$ and by Proposition~\ref{prop:X_3,T} states the action of $\Aut_\R(X)$ on $\mathrm{Pic}(X)$ induces the split exact sequence 
\[1\rightarrow\mathrm{SO}_2(\R)\times\mathrm{SO}_2(\R)\rightarrow\Aut_\R(X_{[3,\FF_0]})\rightarrow D_6\rightarrow1,\] 
or $X= X_{[4]}$ and Proposition~\ref{prop:DP64} states that the action of $\Aut_\R(X)$ induces the split exact sequence 
\[1\rightarrow(\R^*)^2 \Aut_\R(X_{[4]})\rightarrow D_6\rightarrow1.\]
\end{itemize}
In the second case, we look up the results of Section~\ref{sec:CB}:
\begin{itemize}
\item Lemma~\ref{lem:CBDP} implies that $\eta\colon X\rightarrow Y\simeq X_{[2]}$ is the blow-up of $n\geq1$ pairs of non-real conjugate points in the exceptional divisors of $X_{[2]}\rightarrow\mathcal{Q}_{3,1}$ and contained in non-real fibres, and $\pi=\pi_{[2]}\eta\colon X\rightarrow \PP^1$ is relatively $\Aut_\R(X,\pi)$-minimal. Proposition~\ref{prop:CBX_2} implies that the action of $\Aut_\R(X)$ on $\PP^1$ induces the split exact sequence 
\[1\rightarrow\mathrm{SO}_2(\R)\rtimes\Z/2\Z\rightarrow\Aut_\R(X)\rightarrow\rtimes H_{\Delta}\rightarrow,\] 
where $H_\Delta\subset\mathrm{PGL}_2(\R)$ is the subgroup preserving the image in $\PP^1$ of the $2n+2$ singular fibres of $X$ and the interval $\pi(X(\R))=\pi_{[2]}(X_{[2]}(\R))=[0,\infty]$.
\item If $\eta\colon X\rightarrow Y=\FF_n$, then the following possibilities occur:\par
$-$ If $\eta$ is an isomorphism, then $X=\FF_n$ and 
\begin{align*}
&\Aut_\R(\FF_0,\pr_0)\subsetneq\Aut_\R(\FF_0)\\
&\Aut_\R(\FF_1,\pr_1)\subsetneq\Aut_\R(\PP^2)\quad \text{(see beginning of Section~\ref{ssec:CB F})}\\
&\Aut_\R(\FF_n,\pr_n)=\Aut_\R(\FF_n)\simeq\R^{n+1}\rtimes(\mathrm{GL}_2(\R)/\{\mu\Id\mid \mu^n=1\}),\ n\geq2
\end{align*}
\indent $-$ Else, $\Aut_\R(X,\pi)$ being infinite and maximal, Lemma~\ref{lem:CBF} and Lemma~\ref{lem:finite} imply that there exists a birational morphism $\eta'\colon X\rightarrow\FF_N$ that is the blow-up of $2N$ points on $s_N$ (see definition in beginning of Section~\ref{ssec:CB F}). \par
If $N=0$, then $\eta'\colon Y\rightarrow\FF_0$ is an isomorphism, and we have already listed this case.\par
If $N=1$, then $Y$ is a del Pezzo surface of degree $6$; if $\eta'$ blows up two real points, then $	Y\simeq X_{[4]}$ and if $\eta'$ blows up a pair of non-real conjugate points, then $Y\simeq X_{[3,\FF_0]}$ (see Section~\ref{sec:DP}). In either case, $\Aut_\R(Y,\pi)\subseteq\Aut_\R(Y)$. \par
If $N\geq2$, then Proposition~\ref{prop:CBF} states that $\pi=\pi_N\eta\colon Y\rightarrow\PP^1$ is relatively $\Aut_\R(Y,\pi)$-minimal and the action of $\Aut_\R(X)$ on $\PP^1$ induces the split exact sequence
\[1\rightarrow\R^*/\mu_N\rtimes\Z/2\Z\rightarrow\Aut_\R(Y)\rightarrow H_{\Delta}\rightarrow1,\] 
where $H_{\Delta}\subset\mathrm{PGL}_2(\R)$ is the subgroup fixing the image in $\PP^1$ of the $2N$ points and $\mu_N\subset\R^*$ the group of $N$-th roots of unity.
\end{itemize}
\end{proof}

\begin{Prop}\label{prop:maximal}
The groups in Theorem~$\ref{thm:classification}$ are maximal algebraic subgroups of $\Bir_{\R}(\PP^2)$ and the classes are pairwise non-conjugate.
\end{Prop}
\begin{proof}
Let $(X,G)$ be a pair in the list of Theorem~\ref{thm:classification}. To prove that the group $G$ is a maximal algebraic subgroup of $\Bir_{\R}(\PP^2)$, we have to check because of Proposition~\ref{prop:among maximal} that any $G$-equivariant birational map $f\colon X\dashrightarrow Y$, where $Y$ is one of the surfaces listed in Theorem~\ref{thm:classification}, is in fact a $G$-equivariant isomorphism. This will also prove that all the classes are distinct.\par
As $X$ and $Y$ are del Pezzo surfaces or conic bundles, \cite[Theorem 2.5]{Isk96} (and \cite[Appendix]{Cor95}) implies that any birational map $X\dashrightarrow Y$ is an isomorphism or the composition of elementary links, which are divided into type I--IV, shown by the commutative diagrams below, where the horizontal maps are blow-ups defined over $\R$, and $T\in\{\PP^1,\ast\}$. 
\begin{center}
\noindent\begin{minipage}[h]{0.23\textwidth}
\centering
\hspace{0.5em}\xymatrix{X\ar[d]&X_1\ar[l]\ar[d]\\ \ast&\PP^1}
\captionof*{figure}{Type I}
\end{minipage}
\begin{minipage}[h]{.23\textwidth}
\centering
\xymatrix{ X\ar[d]&Z\ar[r]^{\eta_2}\ar[l]_{\eta_1}&X_1\ar[d] \\ T\ar[rr]^{\simeq}&&T} 
\captionof*{figure}{Type II}
\end{minipage}
\begin{minipage}[h]{0.24\textwidth}
\centering\hspace{0.5em}\xymatrix{X\ar[r]\ar[d]&X_1\ar[d]\\ \PP^1&\ast}
\captionof*{figure}{Type III}
\end{minipage}
\begin{minipage}[h]{0.24\textwidth}
\centering\hspace{0.5em}\xymatrix{X\ar[r]^{\simeq}\ar[d]&X_1\ar[d]\\ \PP^1&\PP^1}
\captionof*{figure}{Type IV}
\end{minipage}
\end{center}

The decomposition into links can be made $G$-equivariant because $G$ is a linear algebraic group (Lemma~\ref{lem aut}) and its action on the Picard group is finite by Lemma~\ref{lem:lin alg}. The horizontal maps of the $G$-equivariant links blow up the finite $G$-orbit of a real point or the finite $G$-orbit of a pair of non-real points. \par
Suppose that $f$ is not an isomorphism and let $f=\Phi_n\cdots\Phi_1$ be its decomposition into $G$-equivariant links of type I--IV. Then:  \par
$\bullet$ If $\Phi_1$ is a link of type I, then \cite[Theorem 2.6]{Isk96} implies that $X$ is a del Pezzo surface of degree $9,8$ or $4$. The latter case does not appear, and $X=\PP^2$ or $X=\mathcal{Q}_{3,1}$ or $X=\FF_0$ by \cite{Com12}. However, there are no finite $G$-orbits on $X$, which makes such a link impossible. \par
$\bullet$ If $\Phi_1$ is a link of type III, then \cite[Theorem 2.6]{Isk96} implies that $X\simeq\FF_1$, $X\simeq X_{[3,S]}$ or $X\simeq X_{[2]}$ or $X\simeq X_{[4]}$. Only the latter is in our list. \cite[Theorem 2.6]{Isk96} says that $\Phi_1$ must be the contraction of an orbit of order $2$, which does not exist by Proposition~\ref{prop:DP64}.\par
$\bullet$ If $\Phi_1=\eta_2\eta_1^{-1}$ is a link of type II, then either $X,X_1$ are both del Pezzo surfaces and $T=\ast$ or $X,X_1$ are both conic bundles and $T=\PP^1$. We look at these cases separately:\par
If $X$ and $X_1$ are del Pezzo surfaces, then \cite[Theorem 2.6]{Isk96} implies that the degree of $X$ is $9,8,6,5,4,3$ or $2$. Only the first three degrees appear in our list. If $X$ has degree $9$ or $8$, then, again, there are no finite $G$-orbits on $X$, so a link of type II is not possible. Suppose that $X$ is of degree $6$, i.e. $X\simeq X_{[3,\FF_0]}$ or $X\simeq X_{[4]}$ in Lemma~\ref{lem:dP class}. By \cite[Theorem 2.6]{Isk96}, $\eta_1$ is the blow-up of at most $5$ points on $X$. That is impossible because the only finite $G$-orbit on $X$ has cardinality six by Propositions~\ref{prop:X_3,T} and \ref{prop:DP64}.\par
Suppose that $X$ and $X_1$ are real conic bundles. If $X=\FF_n$ for some $n\in\N$, then $G$ does not have a finite orbit and hence a link of type II cannot start with $\FF_n$.\par
If there exists a birational morphism $X\rightarrow X_{[2]}$ of real conic bundles, then Proposition~\ref{prop:CBX_2} states that $\Aut_\R(X/\pi)$ contains an element exchanging the two unique $(-(n+1))$-sections of $X$. If there exists a birational morphism $X\rightarrow\FF_n$ of conic bundles that is not an isomorphism, then Proposition~\ref{prop:CBF} states that $\Aut_\R(X/\pi)$ contains an element exchanging the two $(-n)$-sections of $X$. In either case, a $G$-orbit on $X$ contains at least two points that are contained in the same fiber, which is not allowed. So there is no link of type II starting with $X$.\par
$\bullet$ If $\Phi_1$ is a link of type IV, then \cite[Theorem 2.6]{Isk96} implies that $K_X^2=8,4,2,1$. So, consulting our list, $K_X^2=8$ and, again by \cite[Theorem 2.6]{Isk96}, $X=X_1=\FF_0$ and $\Phi_1$ exchanges the two fibrations. In particular, $\Phi_1\in\Aut_\R(\FF_0)$.\par
Summarised, one cannot find a decomposition of $f$ into elementary links and hence $f$ is an isomorphism.
\end{proof}

\begin{proof}[Proof of Theorem~$\ref{thm:classification}$]
By Proposition~\ref{prop:among maximal}, any algebraic subgroup of $\Bir_{\R}(\PP^2)$ is conjugate to a subgroup of one of the groups in the list. By Proposition~\ref{prop:maximal}, all  of these groups are maximal and pairwise non-conjugate.
\end{proof}

\begin{proof}[Proof of Theorem~$\ref{thm:parametrisation}$]
The claim is clear for families (\ref{thm:class 1})-(\ref{thm:class 2.2}) and (\ref{thm:class 2}). The rest of the claim follows from the description of $\Aut_\R(X)$ in Propositions~\ref{prop:X_3,T} and \ref{prop:DP64} and of $\Aut_\R(X,\pi)$ in Propositions~\ref{prop:CBX_2} and \ref{prop:CBF}.
\end{proof}

\begin{center}$\ast\ast\ast$\end{center}
\vskip\baselineskip

Let us take a look at which infinite algebraic subgroups survive the abelianisation of $\Bir_{\R}(\PP^2)$.

\begin{Rmk}\label{rmk:same fibre}\item
(1) The construction of the abelianisation $\varphi\colon\Bir_{\R}(\PP^2)\rightarrow\bigoplus_\R\Z/2\Z$ in \cite[Definition 3.10, Proposition 4.3]{Z15} yields the following:
Let $g_1$ and $g_2$ be elementary links of the real conic bundle $\pi_{[2]}\colon X_{[2]}\rightarrow\PP^1$ contracting pairs of non-real cojugate fibres $f_1,\bar{f}_1$ and $f_2,\bar{f}_2$ respectively. Then $g_1$ and $g_2$ have the same image in the quotient if and only if 
\[\pi_{[2]}(f_1)\in\ \R_{>0}\cdot\pi_{[2]}(f_2)\ \cup\ \R_{>0}\cdot\pi_{[2]}(\bar{f}_2)\quad\text{in}\ \PP^1.\] 

(2) Let $\pi\colon X\stackrel{\eta}\rightarrow X_{[2]}\stackrel{\pi_{[2]}}\rightarrow\PP^1$ be a surface as in Theorem~$\ref{thm:classification}~(\ref{thm:class 7})$. We can see elements of $\Aut_\R(X,\pi)$ as birational transformations of $X_{[2]}$ preserving the conic bundle structure. \par
Let $(q_1,\bar{q}_1),\dots,(q_n,\bar{q}_n)$ be the pairs of non-real conjugate points blown up by $\eta$. Then, by definition of $\varphi$ and Proposition~\ref{prop:CBX_2}, we have
\[\varphi(g)=\sum_{i=1}^ne_{\nu(\pi(q_i))}\ \text{if}\ g\in\Aut_\R(X,\pi)\setminus\Aut_\R(X/\pi),\qquad\varphi(g)=0\ \text{if}\ g\in\Aut_\R(X/\pi)\]
where $\nu([a+{\bf i}b:1])=\nu([a-{\bf i}b:1])=\frac{a}{|b|}$ and $e_r$ is the ``standard vector" with entry $1$ at $r$ and zero everywhere else. 
\end{Rmk}

\begin{proof}[Proof of Theorem~$\ref{thm:quotient}$]
It suffices to check that the maximal infinite subgroups of $\Bir_{\R}(\PP^2)$ have trivial image in the quotient $\Bir_{\R}(\PP^2)/\langle\langle\Aut_\R(\PP^2)\rangle\rangle\simeq\bigoplus_\R\Z/2\Z$. The quotient is abelian, so it suffices to check the groups listed in Theorem~\ref{thm:classification}. \par 
Groups (\ref{thm:class 1})--(\ref{thm:class 2.2}) have trivial image because their elements are conjugate to transformations of $\PP^2$ of degree at most $2$. \par
The generators of group (\ref{thm:class 4}) either descend to automorphisms of $\FF_0$ or to birational transformations of $\FF_0$ sending one fibration onto the other (Proposition~\ref{prop:X_3,T}). The latter are conjugate to transformations of $\PP^2$ sending a pencil of lines through a real point onto the pencil of lines through another real point. So, also the generators of (\ref{thm:class 4}) have trivial image in the quotient. \par
The generators of the group (\ref{thm:class 5}) are conjugate to transformations of $\PP^2$ of degree at most $2$ by Proposition~\ref{prop:DP64}, so they are contained in $\langle\langle\Aut_\R(\PP^2)\rangle\rangle$. \par
The groups in families (\ref{thm:class 2}) and (\ref{thm:class 6}) are conjugate to transformations of $\PP^2$ preserving a pencil of lines through a point. So, they have trivial image in the quotient.\par 
By Remark~\ref{rmk:same fibre} there exist real conic bundles $\pi\colon X\stackrel{\eta}\rightarrow X_{[2]}\stackrel{\pi_{[2]}}\rightarrow\PP^1$ as in family (\ref{thm:class 7}) that have non-trivial image in $\bigoplus_\R\Z/2\Z$. The image is finite and they are mapped onto the generator $e_r$ if and only if $\eta$ blows up exactly one pair of non-real conjugate points $q,\bar{q}$ such that $\nu(\pi(p))=r$. 
\end{proof}

\begin{proof}[Proof of Corollary~$\ref{cor:quotient}$]
This is a direct consequence of the fact that the images of the algebraic subgroups are finite by Theorem~\ref{thm:quotient}.
\end{proof}


\end{document}